\documentclass[12pt,reqno]{amsart}
\usepackage[margin=1in]{geometry}
\makeatletter
\def\section{\@startsection{section}{1}%
	\z@{.7\linespacing\@plus\linespacing}{.5\linespacing}%
	{\bfseries\normalfont\scshape
		\centering
}}
\def\@secnumfont{\bfseries}
\makeatother
\usepackage{hyperref}
\usepackage{amsmath,amsfonts,amsthm,txfonts}
\usepackage{dsfont}
\usepackage{amsaddr}

\usepackage{accents}
\usepackage{trimclip}

\newcommand{\newhat}{\scalebox{1.5}[.75]{\trimbox{0pt 1.1ex}{\textasciicircum}}}

\newcommand{\stretchedhat}[1]{\accentset{\newhat}{#1}}
\def\shat{\stretchedhat}

\usepackage{graphicx,amssymb}
\usepackage{amsmath,amssymb,amsfonts}
\usepackage{hyperref}
\usepackage{amscd}
\usepackage{amsthm}
\usepackage{xcolor}
\numberwithin{equation}{section}
\newtheorem{Thm}{Theorem}[section]
\newtheorem{Lem}{Lemma}[section]
\newtheorem{Pro}{Proposition}[section]
\newtheorem{Rem}{Remark}[section]

\newtheorem{Cor}{Corollary}[section]
\newtheorem{Def}{Definition}[section]

\usepackage{refcheck}

\usepackage[shortlabels]{enumitem}

\usepackage{cancel}

\usepackage{hyperref}

\usepackage{cases} 

\allowdisplaybreaks

\title[Optimal control problem  for the 2D  Landau-Lifshitz-Gilbert Equation]
{Analysis of the magnetization control problem for the 2D \\ \vspace{.1in} evolutionary  Landau-Lifshitz-Gilbert equation}
\author[Sidhartha Patnaik and  Sakthivel Kumarasamy ]{Sidhartha Patnaik and Sakthivel Kumarasamy}
\address{ Department of Mathematics \\
Indian Institute of Space Science and Technology (IIST) \\
Trivandrum- 695 547, INDIA}
\email{sidharthpatnaik96@gmail.com, sakthivel@iist.ac.in}
\curraddr{}
\begin{document}
\thanks{The work of the second author is supported by the National Board for Higher Mathematics, Govt. of India through the research grant No:02011/13/2022/R\&D-II/10206.}
\maketitle	

\begin{abstract}
	The magnetization control problem for the Landau-Lifshitz-Gilbert (LLG) equation $m_t= m \times (\Delta m +u)-  m \times (m \times (\Delta m +u)),\ (x,t) \in \Omega\times (0,T] $ with zero Neumann boundary data on a two-dimensional bounded domain $\Omega$ is studied when the control energy $u$ is applied on the effective field. First, we show the existence of a weak solution, and the magnetization vector field $m$ satisfies an energy inequality. If a weak solution $m$ obeys the condition that $\nabla m\in L^4(0,T;L^4(\Omega)),$ then we show that it is a regular solution. The classical cost functional is modified by incorporating $L^4(0,T;L^4(\Omega))$-norm of $\nabla m$  so that a rigorous study of the optimal control problem is established. Then, we justified the existence of an optimal control and derived first-order necessary optimality conditions using an adjoint problem approach. We have established the continuous dependency and Fr\'echet differentiability of the control-to-state and control-to-costate operators and shown the Lipschitz continuity of their Fr\'echet derivatives. Using these postulates,  we derived a local second-order sufficient optimality condition when a control belongs to a critical cone. Finally, we also obtain another remarkable global optimality condition posed only in terms of the adjoint state associated with the control problem.   
\end{abstract}

\section{Introduction}

The study of magnetization dynamics in ferromagnetic media is of great interest because of its wide range of applications, from magnetic sensors to data storage devices. In 1935, L.D. Landau and E.M. Lifshitz obtained  the first dynamical model for micromagnetization phenomena occurring inside ferromagnetic materials (\cite{LLEL}). Later, in 1955, T.L. Gilbert modified this equation by introducing dissipation in a phenomenological way (\cite{TLG}). This paper considers the LLG equation containing energy interactions from exchange energy and external magnetic field.

In the context of our study, we consider a smooth, bounded region $\Omega$ within two-dimensional space, representing a material with ferromagnetic properties. The magnetization of this material is described by the vector field $m$, which varies over both space and time, denoted as $m: \Omega \times [0, T] \rightarrow \mathbb{R}^3$. It's important to note that below the Curie temperature, the magnitude of the magnetization remains constant. To account for this, we normalize the magnetization vector field by dividing it by the saturated magnetization, ensuring that our net magnetization is a unit vector field, i.e., $m \in \mathbb{S}^2$, which is the unit sphere in $\mathbb{R}^3$.

The evolution of this magnetization field is governed by the Landau-Lifshitz-Gilbert (LLG) equation, given by:
$$m_t= \gamma m \times \mathcal{E}_{eff}(m)- \alpha\gamma m \times (m \times \mathcal{E}_{eff}(m)),$$
where $\times$ represents the cross product in $\mathbb{R}^3,$  $\mathcal{E}_{eff}$ represents the effective field acting on the magnetization, $\alpha$ is a positive parameter known as the Gilbert damping constant, and $\gamma$ represents the gyromagnetic factor.  The effective field $\mathcal{E}_{eff}(m)$ is determined as the negative gradient of the micromagnetism energy $\mathcal{E}(m)$, that is, $\mathcal{E}_{eff}(m)=-\nabla_m \mathcal{E}(m)$. The micromagnetism energy $\mathcal{E}(m)$ encapsulates various energy interactions within the ferromagnetic material.  Specifically, $\mathcal{E}=\mathcal{E}_{ex}+\mathcal{E}_{an}+\mathcal{E}_{me}+\mathcal{E}_{d}+\mathcal{E}_{a}$, that is, $\mathcal{E}$ comprises the following energy components: exchange energy $\mathcal{E}_{ex}$, anisotropy energy $\mathcal{E}_{an}$, magnetoelastic energy $\mathcal{E}_{me}$, demagnetization field energy $\mathcal{E}_{d}$, and external magnetic field $\mathcal{E}_{a}$. For more comprehensive details on these energy interactions we refer to \cite{WFB}.

In ferromagnetic materials, atomic magnetic moments strive to align themselves with neighboring moments due to exchange interactions, leading to an increase in exchange energy when they deviate from their equilibrium orientation. Assuming an external magnetic field described by a function $u:\Omega\times[0,T]\to\mathbb R^3$ and a purely isotropic energy field, the micromagnetism energy can be expressed as:
$$\mathcal{E}=\frac{1}{2} \int_\Omega |\nabla m|^2 dx- \int_\Omega u\cdot m\ dx.$$

Considering the associated magnetic fields for these energy components, the effective field is given by $\mathcal{E}_{eff}(m)=\Delta m+u$. For an extensive overview of the model and the physical interpretation of these energy terms, we refer to \cite{MKAP}. 

By taking the effective field concerning only the exchange energy and external magnetic field into account, and fixing the damping constant and gyromagnetic factor to be of unit values, that is, $\alpha=\gamma=1$, the Landau-Lifshitz-Gilbert equation becomes
\begin{equation}\label{NLP}
	\begin{cases}
		m_t= m \times (\Delta m +u)-  m \times (m \times (\Delta m +u)),\ \ \ (x,t) \in \Omega_T:=\Omega\times (0,T], \\
		
		\frac{\partial m}{\partial \eta}=0, \ \ \ \ \ \ \ \ (x,t) \in \partial\Omega_T:= \partial \Omega\times [0,T],\\
		m(\cdot,0)=m_0 \ \ \text{in} \ \Omega,
	\end{cases}	
\end{equation}
%u is the external magnetic field(already mentioned)       
where $\eta$ is the outward unit  normal vector to the boundary $\partial \Omega$ and $u$ is the external magnetic field.
Further, while going for the regular solution of system \eqref{NLP}, we have assumed that the initial magnetic field  $m_0:\Omega\to\mathbb R^3$ satisfies the following conditions
\begin{equation}\label{IC}
	m_0 \in H^2(\Omega),\ \ \frac{\partial m_0}{\partial \eta}=0 \ \ \text{on} \ \partial \Omega,\ \ |m_0|=1.	
\end{equation}
We have considered the cost functional $\mathcal{J}:\mathcal M \times \mathcal{U}_{ad} \to \mathbb{R}^+$ defined as
%\begin{equation}\label{CF}
%	\mathcal J(m,u):= \frac{1}{4} \int_0^T \|\nabla  m- \nabla  m_d\|^4_{L^4(\Omega)} dt +  \frac{1}{2} \int_{\Omega} |m(x,T) -m_\Omega(x)|^2\ dx +\frac{1}{2} \int_0^T \int_{\Omega} |u(x,t)|^2 \ dx\ dt	
%\end{equation} 
\begin{eqnarray}\label{CF}
	\mathcal J(m,u):= \frac{1}{4} \int_0^T \int_{\Omega} |\nabla  m- \nabla  m_d|^4\ dx\  dt +  \frac{1}{2} \int_{\Omega} |m(x,T) -m_\Omega(x)|^2\ dx\nonumber\\
	+\frac{1}{2} \int_0^T \int_{\Omega} |u(x,t)|^2 \ dx\ dt +\frac{1}{2} \int_0^T \int_{\Omega} |\nabla u(x,t)|^2 \ dx\ dt,				
\end{eqnarray}
where the desired evolutionary magnetic moment $m_d$ is a map from $\Omega\times[0,T]$ to $\mathbb R^3$ such that $\nabla m_d \in L^4(0,T;L^4(\Omega))$ and final time target moment $m_{\Omega}:\Omega\to\mathbb R^3$ belongs to $L^2(\Omega)$. The optimal control problem can be interpreted as the search for an optimal strategy to magnetize a ferromagnetic material such that\vspace{-0.3cm}
\begin{enumerate}[(i)]
	\item the gradient profile of desired magnetization vector field $m_d$ can be realized by the first term of the cost functional. This first term of $\mathcal J(m,u),$ which is defined with $L^4(0,T;L^4(\Omega))$ norm of $\nabla(m-m_d)$ instead of usual $L^2(0,T;L^2(\Omega)),$ also plays a crucial role in showing the strong solvability of system \eqref{NLP} in Section \ref{SECSS}.
	\item The magnetization in the material should reach the target moment $m_{\Omega}$ at the final time, represented by the second term of the cost functional.
	\item The last two terms ensures that we achieve desired gradient profile of evolutionary moment $m_d$ and final target $m_{\Omega}$ by applying the least amount of external magnetic field in $L^2(0,T;H^1(\Omega))$ space.
\end{enumerate}\vspace{-0.3cm}
% The optimal control problem given by the Landau-Lifshitz-Gilbert equation has numerous real world applications. The control problem can be applied to magnetic hard drives and other magnetic storage devices. By optimizing the external magnetic field, it is possible to enhance the writing and reading processes, improving the efficiency and reliability of data storage systems. 
%The optimal control problem governed by the Landau-Lifshitz-Gilbert equation has numerous real world applications. 

The control problem finds application in optimizing the performance of magnetic hard drives,  magnetic random access memories (MRAM) and various magnetic storage devices \cite{BGSA,TDMKAP,THL}.  By strategically optimizing the external magnetic field, the writing process can be significantly improved, leading to enhanced efficiency and reliability of data storage systems. The control problem can also be relevant in designing and optimizing spintronic devices, such as spin valves and magnetic tunnel junctions, to achieve desired spin configurations and improve device performance \cite{NMOHK}. 
%Further, the utilization of the LLG equation for modeling the magnetization behavior of atomic nuclei within a magnetic field has proven instrumental in  Magnetic Resonance Imaging (MRI). The comprehension of the intricate dynamics governed by the LLG equation empowers researchers and engineers to finely optimize the magnetic field gradients  and radiofrequency (RF) pulses deployed in MRI \cite{KGBM}. Such optimization endeavors yield tangible benefits, manifesting as enhanced image quality and expedited acquisition speed.
 %These are the few applications concerned with the optimal control of the LLG equation. 
 This paper explores the potential of such control optimization techniques by studying the second-order sufficient local optimality condition and adjoint system based global optimality condition, which is very much essential for numerical methods.

%[Omit this paragraph.] Magnetic sensors play a vital role in various applications, ranging from compasses and position sensors to magnetic field measurement devices, where optimizing the control of the external magnetic field is crucial for maximizing sensitivity and accuracy. One such device is magnetoresistive sensor, in which changes in the magnetic field induce variations in the resistance of a magnetic material. The LLG equation is utilized to model the behavior of the magnetization in the material and understand the underlying physics governing the resistance changes. By incorporating the LLG equation into the design and optimization process, the sensitivity and accuracy of magnetoresistive sensors can be improved. Additionally, the LLG equation is also applicable in other magnetic sensor devices such as magnetic tunnel junctions (MTJs) and magnetic field sensors based on giant magnetoresistance (GMR) effects. In these devices, the LLG equation helps in understanding the magnetization dynamics and optimizing the performance of the sensors for precise and reliable magnetic field detection.

% The control problem also finds valuable applications in micromanipulation tasks involving magnetic particles or objects. By leveraging precise control over the external magnetic field, it becomes feasible to manipulate and position magnetic objects with exceptional precision, thereby facilitating advancements in fields such as biomedicine and microfluidics.

In \cite{FAAS}, the authors established the weak solutions for the Landau-Lifshitz (LL) equation in  both bounded regular domain of $\mathbb R^3$ and on the entire domain  $\mathbb{R}^3$  without an external magnetic field $u.$ They utilized a penalized form, ensuring $|m|=1$, and demonstrated its convergence to the weak formulation. Further, they obtained a critical result that the weak solution of the LL equation exists globally but, in general, is not unique.  A visible amount of work has already been done for the strong solution of  the LLG equation but without the nonlinear control term in the effective field. The local existence and uniqueness of regular solution over a bounded domain in $\mathbb{R}^n,n=2,3$ and the whole domain $\mathbb{R}^3$ was investigated in \cite{GCPF}, \cite{GCPF2}.   The global existence under a smallness assumption on the initial data over 2D and 3D bounded domains have been analyzed in \cite{GCPF} and \cite{MFTT}, respectively.  Apart from the solvability of the LLG equation, as explained above, minimal articles are available for the optimal control problem.  The paper \cite{TDMKAP} discusses the solvability of the control problem for the 1D LLG equation, the existence of an optimum, and gives the first-order necessary optimality condition. The article \cite{FAKB} investigates the optimal control type problems associated with the LLG equation, and a necessary optimality system is derived when the magnetization is constant in space, which essentially leads to an optimization problem constraint by an ordinary differential equation.  Further, for more results related to controllability of the Landau-Lifshitz equation, one may refer to \cite{SAGC},\cite{GCSL2},\cite{AC}.

\noindent {\bf The main contributions of this paper are summarized as follows:}

\begin{itemize}
%	\item  We have established the existence of a weak solution $m\in L^2(0,T;H^1(\Omega))$ for the LLG equation \eqref{NLP} on a bounded domain $\Omega\subset \mathbb R^2$, following the strategy outlined in \cite{FAAS}. Notably, we addressed the nonlinearity arising from the external magnetic field/control $u$ in \eqref{NLP} by introducing a suitable penalized problem. This approach allowed us to demonstrate the existence of a weak solution when the control is in a general space $L^2(0,T;L^2(\Omega))$.
%It is worth noting that as the external magnetic field/control $u$ arises nonlinearly in \eqref{NLP}, we introduced an appropriate penalized problem and showed the existence of a weak solution when the control $u\in L^2(0,T;L^2(\Omega)).$
%\item We have proven the existence of a weak solution for the LLG equation \eqref{NLP} on a bounded domain $\Omega \subset \mathbb{R}^2$, following the methodology in \cite{FAAS}. Addressing the complications arising due to the control $u$, we introduce a suitable penalized problem. This approach enabled us to establish the existence of a weak solution even when the control belongs to the general space $L^2(0,T;L^2(\Omega))$. 
\item We have proven the existence of a weak solution for the LLG equation \eqref{NLP} on a bounded domain $\Omega \subset \mathbb{R}^2$, following the methodology in \cite{FAAS}. Addressing complications arising from the nonlinear control $u$, we introduce a suitable penalized problem, enabling the existence of a weak solution  when the control is in $L^2(0,T;L^2(\Omega))$.
	\item In the context of the 1D LLG equation, as demonstrated in \cite{TDMKAP}, the authors established the existence of a unique strong solution in $L^2(0,T;H^2(\Omega))$ when the control belongs to $L^2(0,T;L^2(\Omega))$. This result was achieved without imposing any restrictions based on the size of the initial data or control. However, when extending this analysis to the 2D LLG system \eqref{NLP}, the severe non-linearity of the state equation made it challenging to directly show the strong solution for controls in $L^2(0,T;L^2(\Omega))$ using the classical Galerkin method. Infact, the existence of regular solution is proven for uncontrolled 2D LLG equation only when the initial data is sufficiently small (see, \cite{GCPF}, \cite{AP}).  In our previous work \cite{SPSK}, we addressed the local solvability of the 2D control problem \eqref{NLP} for all controls $u\in L^2(0,T;H^1(\Omega))$ with initial data satisfying condition \eqref{IC}. Moreover, we substantiated the existence and uniqueness of a global regular solution (see, Theorem 2.2, \cite{SPSK}) under a smallness condition on control and initial data.
	
	In this paper, we have made a significant contribution by demonstrating that if the magnetization field $m$ is a weak solution of \eqref{NLP} and satisfies $\nabla m\in L^4(0,T;L^4(\Omega))$, it qualifies as a regular solution of the LLG equation \eqref{NLP}, all without the need for smallness conditions imposed on the control and initial data. This result paves the way for a more comprehensive and realistic exploration of the control problem, offering broader applicability and relevance in practical scenarios.
%	\item In the 1D LLG equation, as shown in \cite{TDMKAP}, a unique strong solution in $L^2(0,T;H^2(\Omega))$ was established for controls in $L^2(0,T;L^2(\Omega))$, without size restrictions on initial data or control. Extending this to the 2D LLG system \eqref{NLP} was challenging due to its severe non-linearity. 
	%Directly proving strong solutions for controls in $L^2(0,T;L^2(\Omega))$ using the classical Galerkin method was difficult. 
%	Regular solutions for the uncontrolled 2D LLG equation are proven only with sufficiently small initial data (see \cite{GCPF}, \cite{AP}).
	
%	In our prior work \cite{SPSK}, we addressed local solvability of the 2D control problem \eqref{NLP} for all controls $u\in L^2(0,T;H^1(\Omega))$ with initial data satisfying \eqref{IC}. We substantiated the existence and uniqueness of a global regular solution (Theorem 2.2, \cite{SPSK}) under a smallness condition on control and initial data.

	\item When the  state equation doesn’t admit a (unique) strong solution for a general class of controls and data,  in Chapter 4 of \cite{SSS}, the author established and validated the argument that any weak solution of 3D Navier-Stokes equation belonging to the space $L^8(0,T;L^4(\Omega))$ become a strong solution. Subsequently, in \cite{ECKC}, the authors considered an appropriate  norm in the cost functional, which confirmed the strong solvability of the problem and allowed for the study of the associated optimal control problem.  Moreover, for a different method of solving the optimal control of 3D Navier-Stokes equations with pointwise control constraints and MHD equations with state constraints, one may refer to \cite{BTK} and \cite{LW}, respectively.

We adopted a similar concept of \cite{ECKC} in the current paper but with a different approach. Instead of directly employing the method of that paper, which would involve a specific norm condition on $\Delta m$ by which weak solutions of the LLG equation become regular, we utilized the regularity property of the local solution and derived estimate \eqref{SSEE2}. This approach led us to identify a condition on the gradient of the weak solution under which it qualifies as a regular solution. This insight prompted the introduction of the term $\|\nabla m-\nabla m_d\|^4_{L^4(0,T;L^4(\Omega))}$ in the cost functional. This term ensures that any weak solution $m$ of \eqref{NLP} satisfying $\mathcal J(m,u)<+\infty$ is a regular solution of \eqref{NLP}. Consequently, we proceeded to study the existence of an optimal control for \eqref{NLP}-\eqref{CF} and derived the first-order optimality conditions.

\item Finally, we address two more crucial results of this paper: a local second-order sufficient optimality condition and a global optimality condition. The local optimality criterion, is defined on a certain cone of critical directions (see, \cite{FT}). It requires a rigorous study of the control-to-state operator that maps from the control space $\mathcal  U_R$ into the regular solution space $\mathcal M$  and control-to-costate operator that maps from control set $\mathcal{U}_R$ to the weak adjoint solution space $\mathcal{Z}$.  We establish the important properties that both of these operators are Lipschitz continuous and Fr\'echet differentiable.
		
	Furthermore, we prove a remarkable theorem stating a global optimality condition, relying solely on the $L^2(0, T; L^2(\Omega))$ norm of the weak solution of the adjoint system \eqref{AS}. These optimality conditions are particularly valuable in non-convex or nonlinear optimal control problems and play a crucial role in demonstrating the convergence of error estimates for the numerical discretization of corresponding computational optimization problems (see, \cite{ADH, CRT, FTW}). 

\end{itemize} 

The paper is organized as follows: In subsection \ref{MR}, we present the main results of this paper. Subsection \ref{FS} establishes standard inequalities and essential cross product estimates that are utilized throughout the manuscript. Section \ref{S-WSRS} addresses the existence of both weak and regular solutions for system \eqref{NLP}. Section \ref{S-FOOC} investigates the existence of optimal control and presents the first-order optimality condition. This section also discusses the Fr\'echet differentiability and Lipschitz continuity of the control-to-state operator, as well as the weak solvability of the adjoint system. Subsection \ref{S-CTCO} focuses on the Fr\'echet differentiability and Lipschitz continuity of the control-to-costate operator. Lastly, subsection \ref{S-SOOC} is dedicated to the second-order optimality condition, and section \ref{S-GOC} concludes the paper with the global optimality condition.

\section{Main Results and Function Spaces}

\subsection{Main Results}\label{MR}

In this section we state the main results of our paper. For any regular solution of system \eqref{NLP}, by taking cross product with $m$, we can find that \eqref{NLP} is equivalent to  
$$\frac{\partial m}{\partial t} + m \times \frac{\partial m}{\partial t}=2 \sum_{i=1}^{3} \frac{\partial }{\partial x_i} \left(m\times \frac{\partial m}{\partial x_i}\right)+2\  m \times u.$$
This leads us to introduce the following notion of weak solution for system \eqref{NLP} (see, \cite{FAAS}).
\begin{Def}\label{WSD}
	Suppose $m_0 \in H^1(\Omega)$ with $|m_0|=1$ a.e. $x \in \Omega$ and $u\in L^2(0,T;L^2(\Omega))$. A function $m\in L^2(0,T;H^1(\Omega))$ with $m_t\in L^2(0,T;L^2(\Omega))$ and $|m|=1$ a.e. in $\Omega_T$ is said to be a weak solution of \eqref{NLP} if the following hold: 
	\begin{enumerate}
		\item 			 $\displaystyle{ \int_{\Omega_T} m_t \cdot \psi \ dx\ dt + \int_{\Omega_T}\left(m \times m_t \right)\cdot \psi \ dx\ dt}$\\ 
		$\displaystyle{\hspace{0.4in}=-2\sum_{i=1}^{3}\int_{\Omega_T}\left(m \times\frac{\partial m}{\partial x_i}\right)\cdot \frac{\partial \psi}{\partial x_i}\ dx\ dt+2 \int_{\Omega_T}(m \times u)\cdot \psi \ dx\  dt\ \ \ \ \forall \ \psi \in L^2(0,T;H^1(\Omega)),}$
		\item $m(x,0)=m_0(x)$,\\
		\item $\displaystyle{ \frac{1}{2}\left\|m_t \right\|^2_{L^2(0,T;L^2(\Omega)}+\sup_{t \in [0,T]} \|\nabla m(t)\|^2_{L^2(\Omega)}  \ \leq   \|\nabla m_0\|^2_{L^2(\Omega)}+2\ \|u\|^2_{L^2(0,T;L^2(\Omega))}.}$
	\end{enumerate}
\end{Def}
The following theorem proves the existence of weak solution of system \eqref{NLP} in the above sense. 
\begin{Thm}[Weak Solution]\label{EOWS}
	For any control $u\in L^2(0,T;L^2(\Omega))$ and initial data $m_0\in H^1(\Omega)$ with $|m_0|=1$, there exists a weak solution of system \eqref{NLP} in the sense of Definition \ref{WSD}.
\end{Thm}

The proof of this theorem resemble the arguments put forth in \cite{FAAS}. However, in our problem the control term appears non-linearly in the effective field. As a result, we provide a concise overview of the proof in our paper to account this modification. 

Considering the dot product of \eqref{NLP} with $m$ and applying the cross product properties stated in Lemma \ref{FS}, we can see that $m$ and $m_t$ are orthogonal in space and time, that is, $(d/dt)|m(\cdot,t)|^2=0$. Since $m_0\in \mathbb{S}^2$, we find that $|m(x,t)|^2=1$, pointwise in space and time. Consequently, for a regular solution $m$, expanding the cross product $m\times (m\times\Delta m)$ in \eqref{NLP} and  using the identity 
$\Delta |m|^2=2(m\cdot \Delta m)+ 2|\nabla m|^2$, we find an equivalent system of \eqref{NLP}:
\begin{equation}\label{EP}
	\begin{cases}
		m_t- \Delta m= |\nabla m|^2m +  m \times \Delta m  + m \times u- m \times (m \times u), \ \ \ (x,t)\in \Omega_T,\\
		
		\frac{\partial m}{\partial \eta}=0, \ \ \ \ \ \ \ (x,t) \in  \partial\Omega_T,\\
		
		m(\cdot,0)=m_0 \ \ \text{in} \ \Omega.
	\end{cases}	
\end{equation}
Due to the absence of explicit elliptic operator in equation \eqref{NLP}, we can't directly apply the classical techniques of Galerkin approximation. So, we will show the solvability of this equivalent system \eqref{EP} instead of \eqref{NLP}. If a function $m$ satisfies system \eqref{EP} almost everywhere and has an absolute value of $|m|=1$, it will also satisfy system \eqref{NLP} almost everywhere.

%   \begin{Def}[Strong Solution]
	%		Suppose the control $u\in L^2(0,T;H^1(\Omega))$ and the initial data $m_0$ satisfies \eqref{IC}. Then we say that $m$ is a regular solution of the system \eqref{NLP} if it is a weak solution and $\nabla   m \in L^4(0,T;L^4(\Omega))$. 
	%	\end{Def}

Now, define the set of admissible solution space as 
$$\mathcal{M}:= W^{1,2}(0,T;H^3(\Omega),H^1(\Omega))=\{m \in L^2(0,T;H^3(\Omega)) \ |\ m_t \in L^2(0,T;H^1(\Omega))\}.$$ 
A function $m_u \in \mathcal{M}$ that satisfies system \eqref{NLP} almost everywhere  while being associated with the control function $u\in L^2(0,T;H^1(\Omega))$ is referred to as a "regular solution."

\begin{Thm}(Global Existence of Regular Solution)\label{T-SS}
	Suppose the control $u\in L^2(0,T;H^1(\Omega))$ and the initial data $m_0$ satisfies condition \eqref{IC}. If $m$ is any weak solution of system \eqref{NLP} and satisfies the condition $\nabla   m \in L^4(0,T;L^4(\Omega))$, then $m$ is a unique regular solution.
	
	\noindent Moreover, there exists a constant $C(\Omega,T)>0$ such that the following estimate holds:
	\begin{align}\label{SSEE2}
		&\|m\|^2_{L^\infty(0,T;H^2(\Omega))} + \|m\|^2_{L^2(0,T;H^3(\Omega))}+ \|m_t\|^2_{L^2(0,T;H^1(\Omega))}  \nonumber\\
		&\ \ \ \ \ \ \ \  \leq   \exp\left(C\ \left[ 1 + \|\Delta  m_0\|^4_{L^2(\Omega)} +\ \|\nabla m\|^8_{L^4(0,T;L^4(\Omega))}+ \ \|u\|^4_{L^2(0,T;H^1(\Omega))} \right] \right).
	\end{align}	
\end{Thm}
\begin{Rem}
	For any control $u\in L^2(0,T;H^1(\Omega))$ and initial data $m_0$ satisfying condition \eqref{IC}, if a regular solution exists, then it must be unique (see, Theorem 2.1, \cite{SPSK}).
\end{Rem}
\noindent As boundedness of $\mathcal{J}(m,u)$ implies $\|\nabla m\|_{L^4(0,T;L^4(\Omega))}< \infty$, so we can get the following remark.
\begin{Rem}
	If $m$ is a weak solution of \eqref{NLP} verifying $\mathcal{J}(m,u)<+\infty$, then $m$ is a regular solution.	
\end{Rem}
Let  $\mathcal{U}$ be the set of all controls in $L^2(0,T;H^1(\Omega))$ for which there exists a regular solution in $\mathcal{M}$. As a consequence of Proposition 5.1 of \cite{SPSK},  we can find that the set of controls $\mathcal{U}$ is an open set. We aim to prove the Fréchet differentiability of the control-to-state and control-to-costate operators, as well as the Lipschitz continuity of their derivatives over the open set $\mathcal{U}_R = \left\{ u \in \mathcal{U} : \mathcal{J}(m_u,u) < R \right\}$. Defining such a control set has the advantage that, according to estimate \eqref{SSEE2}, the admissible solution set will be uniformly bounded. In other words, there exists a constant $C$ depending on $R$ and initial data $m_0$, but independent of the control variable such that the following holds:
\begin{equation}\label{SSUB}
	\|m_u\|_{\mathcal{M}} \leq C, \ \ \ \ \forall \  u \in \mathcal{U}_R.
\end{equation}

In this work, we also consider the admissible control variables subject to box-type constraints. Given any  functions $a,b \in L^\infty (\Omega_T)$, we define
$$\mathcal{U}_{a,b}:=\left\{u \in L^2(0,T;L^2(\Omega)): a(x,t)\leq u(x,t)\leq b(x,t)\ \ \text{for almost every}\ (x,t) \in \Omega_T\right\}.$$
Furthermore, we introduce the set $\overline{\mathcal{U}_{R/2}} := \left\{ u\in \mathcal{U}  \ :\ \mathcal{J}(m_u,u) \leq  R/2 \right\}$. The admissible control set is then defined as $\mathcal{U}_{ad}:= \mathcal{U}_{a,b} \cap \overline{\mathcal{U}_{R/2}}$.

% For any functions $a,b \in L^\infty (\Omega_T)$, we define the set of admissible controls as $\mathcal{U}_{ad}:= \mathcal{U}_{a,b} \cap \overline{\mathcal{U}_{R/2}}$, where $\overline{\mathcal{U}_{R/2}} := \left\{ u\in \mathcal{U}  \ :\ \mathcal{J}(m_u,u) \leq  R/2 \right\}$ and 
%$$\mathcal{U}_{a,b}=\left\{u \in L^2(0,T;L^2(\Omega)): a(x,t)\leq u(x,t)\leq b(x,t)\ \ \text{for almost every}\ (x,t) \in \Omega_T\right\}.$$
In  order to give a possible detailed analysis of the control problem, we have made the following non-emptiness assumption on the set of admissible control:
\begin{equation}\label{AOAC}
	\mathcal{U}_{ad}\neq  \emptyset.
\end{equation}
 A pair $(m,u) \in \mathcal{M} \times \mathcal{U}_{ad}$ is called an admissible pair if $(m,u)$ is a regular solution of system \eqref{NLP}. Let us denote the set of all admissible pair as $\mathcal{A}$.
The optimal control problem is stated as follows:
\begin{equation*}\label{OCP}
	\text{(OCP)}\begin{cases}
		\text{minimize}\ \mathcal J(m,u),\\
		(m,u) \in \mathcal{A}.
	\end{cases}	
\end{equation*}
\begin{Thm}[Existence of Optimal Control]\label{T-EOOC}
	Suppose the initial data $m_0$ satisfies condition \eqref{IC}, final time target $m_\Omega\in L^2(\Omega)$ and $\nabla m_d\in L^4(0,T;L^4(\Omega))$. Then under the assumption \eqref{AOAC}, the OCP has at least one solution. 
\end{Thm}

%Next, we proceed to establish the first-order necessary optimality condition for the problem OCP in the form of a variational inequality by employing the classical adjoint problem approach. 

In the subsequent steps, we will establish the first-order necessary optimality condition for the OCP by utilizing the classical adjoint problem approach. This approach allows us to express the optimality condition in a compact form, providing a concise framework for analysis. First, by employing the formal Lagrange method \cite{FT}, we derive the adjoint system associated with OCP as follows:
\begin{equation}\label{AS}
	\begin{cases}
		\phi_t + \Delta \phi  + |\nabla \widetilde{m}|^2\phi -2 \nabla \cdot \big((\widetilde{m}\cdot \phi)\nabla \widetilde{m}\big) + \Delta (\phi \times \widetilde{m})+(\Delta  \widetilde{m}\times \phi)-(\phi \times \widetilde{u})\\
		\hspace{1.6cm}+\  \big((\phi \times \widetilde{m})\times \widetilde{u}\big) + \big(\phi \times (\widetilde{m}\times \widetilde{u})\big)=  \nabla \cdot \left(|\nabla \widetilde{m}-\nabla m_d|^2 (\nabla \widetilde{m}-\nabla m_d)\right) \ \ \ \text{in}\ \Omega_T,\vspace{0.2cm}\\
		\frac{\partial \phi}{\partial \eta}=0 \ \ \ \text{in}\ \partial \Omega_T,\vspace{0.1cm}\\
		\phi(T)=\widetilde{m}(x,T)-m_{\Omega}(x)\ \ \ \text{in} \ \Omega.
	\end{cases}
\end{equation}
%Weak formulation of system \eqref{AS} is given by \eqref{WAF} with $g=\nabla \cdot \left(|\nabla \widetilde{m}-\nabla m_d|^2 (\nabla \widetilde{m}-\nabla m_d)\right)$.
Define the weak adjoint solution space as 
$$\mathcal{Z}:=W^{1,2}(0,T;H^1(\Omega),H^1(\Omega)^*)=\{z \in L^2(0,T;H^1(\Omega))\ | \ z_t \in L^2(0,T;H^1(\Omega)^*)\},$$
where $H^1(\Omega)^*$ is the dual of the space $H^1(\Omega)$. 	Let $\langle\cdot ,\cdot \rangle$ denote the duality pairing between $H^1(\Omega)$ and $H^1(\Omega)^*$. We have the following notion of weak solution for the system \eqref{AS}.

%For the definition of weak solution of system \eqref{AS} refer to \eqref{WAF}.

%\begin{Def}[Weak formulation of the adjoint system]\label{WSD}
%	A function $\phi \in \mathcal{Z}$ is said to be a weak solution of the adjoint system \eqref{AS} if for every $v\in L^2(0,T;H^1(\Omega))$, the following holds:
%	\begin{enumerate}[(i)]
	%		\item $\begin{aligned}[t]
		%			&\int_0^T \langle \phi'(t),\vartheta (t)\rangle_{H^1(\Omega)^*\times H^1(\Omega)}dt - \int_{\Omega_T} \nabla \phi \cdot \nabla  \vartheta\ dx \ dt+ \int_{\Omega_T} |\nabla \widetilde{m}|^2\phi \cdot \vartheta \ dx\ dt\nonumber\\
		%			&\hspace{1cm}+2 \int_{\Omega_T} (\widetilde{m}\cdot \phi)\nabla \widetilde{m}\cdot \nabla \vartheta\ dx\ dt- \int_{\Omega_T}\nabla (\phi \times \widetilde{m})\cdot \nabla  \vartheta\ dx \ dt+ \int_{\Omega_T}(\Delta \widetilde{m}\times \phi)\cdot \vartheta\ dx\ dt\nonumber\hspace{2cm}\\
		%			&\hspace{1cm}-\int_{\Omega_T} (\phi \times \widetilde{u})\cdot \vartheta\ dx\ dt			+ \int_{\Omega_T} \big((\phi \times \widetilde{m})\times \widetilde{u}\big)\cdot \vartheta\ dx\ dt+ \int_{\Omega_T} \big(\phi \times (\widetilde{m}\times \widetilde{u})\big)\cdot \vartheta\ dx\ dt\nonumber\\
		%			&\hspace{1cm}= -\int_{\Omega_T} |\nabla \widetilde{m}-\nabla m_d|^2 (\nabla \widetilde{m}-\nabla m_d) \cdot \nabla  \vartheta\ dx\ dt,  
		%		\end{aligned}$\vspace{0.1cm}
	%		\item $\phi(T)=\widetilde{m}(x,T)-m_{\Omega}(x).$
	%	\end{enumerate}
%\end{Def}

	\begin{Def}[Weak Formulation]\label{AWSD}
	A function $\phi \in \mathcal{Z}$ is said to be a weak solution of system \eqref{CLAS} if for every $\vartheta \in L^2(0,T;H^1(\Omega))$, the following holds:
%	\begin{enumerate}[(i)]
%		\item $\begin{aligned}[t]
%			&\int_0^T \langle \phi'(t),\vartheta (t)\rangle_{H^1(\Omega)^*\times H^1(\Omega)}\ dt - \int_{\Omega_T} \nabla \phi \cdot \nabla  \vartheta\ dx \ dt+ \int_{\Omega_T} |\nabla \widetilde{m}|^2\phi \cdot \vartheta \ dx\ dt\nonumber\\
%			&\hspace{1cm}+2 \int_{\Omega_T} (\widetilde{m}\cdot \phi)\nabla \widetilde{m}\cdot \nabla \vartheta\ dx\ dt- \int_{\Omega_T}\nabla (\phi \times \widetilde{m})\cdot \nabla  \vartheta\ dx \ dt+ \int_{\Omega_T}(\Delta \widetilde{m}\times \phi)\cdot \vartheta\ dx\ dt\nonumber\hspace{2cm}\\
%			&\hspace{1cm}-\int_{\Omega_T} (\phi \times \widetilde{u})\cdot \vartheta\ dx\ dt			+ \int_{\Omega_T} \big((\phi \times \widetilde{m})\times \widetilde{u}\big)\cdot \vartheta\ dx\ dt+ \int_{\Omega_T} \big(\phi \times (\widetilde{m}\times \widetilde{u})\big)\cdot \vartheta\ dx\ dt\nonumber\\
%			&\hspace{1cm}= -\int_{\Omega_T} |\nabla \widetilde{m}-\nabla m_d|^2 (\nabla \widetilde{m}-\nabla m_d) \cdot \nabla  \vartheta\ dx\ dt,  
%		\end{aligned}$\vspace{0.1cm}
%		\item $\phi(T)=\phi_T.$
%	\end{enumerate}
%	\begin{equation}\label{WAF}
%	\end{equation}
\begin{align}\label{WAF}
	&(i)\ \ \ \int_0^T \langle \phi'(t),\vartheta (t)\rangle_{H^1(\Omega)^*\times H^1(\Omega)}\ dt - \int_{\Omega_T} \nabla \phi \cdot \nabla  \vartheta\ dx \ dt+ \int_{\Omega_T} |\nabla \widetilde{m}|^2\phi \cdot \vartheta \ dx\ dt\nonumber\\
	&\hspace{2cm}+2 \int_{\Omega_T} (\widetilde{m}\cdot \phi)\nabla \widetilde{m}\cdot \nabla \vartheta\ dx\ dt- \int_{\Omega_T}\nabla (\phi \times \widetilde{m})\cdot \nabla  \vartheta\ dx \ dt+ \int_{\Omega_T}(\Delta \widetilde{m}\times \phi)\cdot \vartheta\ dx\ dt\nonumber\\
	&\hspace{2cm}-\int_{\Omega_T} (\phi \times \widetilde{u})\cdot \vartheta\ dx\ dt			+ \int_{\Omega_T} \big((\phi \times \widetilde{m})\times \widetilde{u}\big)\cdot \vartheta\ dx\ dt+ \int_{\Omega_T} \big(\phi \times (\widetilde{m}\times \widetilde{u})\big)\cdot \vartheta\ dx\ dt\nonumber\\
	&\hspace{2cm}= -\int_{\Omega_T} |\nabla \widetilde{m}-\nabla m_d|^2 (\nabla \widetilde{m}-\nabla m_d) \cdot \nabla  \vartheta\ dx\ dt,\\  
    &(ii)\ \ \phi(T)=\phi_T.\nonumber
\end{align}
\end{Def}
\begin{Thm}[Weak Solution of Adjoint System]\label{T-AWS}
	Suppose $(\widetilde{m},\widetilde{u})$ is an admissible pair, that is, $(\widetilde{m},\widetilde{u}) \in \mathcal{A}$, $m_{\Omega}\in L^2(\Omega)$ and $\nabla m_d\in L^6(0,T;L^6(\Omega))$. Then, there exists a unique weak solution $\phi \in \mathcal{Z}$ of the adjoint system \eqref{AS} in the sense of Definition \ref{AWSD}.  Moreover, the following estimate holds: 
	%		\begin{align}
		%			&\int_0^T \langle \phi'(t),v(t)\rangle_{H^1(\Omega)^*\times H^1(\Omega)}dt = -\int_0^T(\phi(t),v'(t)) \ dt\nonumber\\
		%			&\hspace{1in}+\big(\phi(T),v(T)\big)-\big(\phi(0),v(0)\big).\label{INBY}
		%		\end{align}
	\begin{eqnarray}\label{AEEE}
		\lefteqn{\|\phi\|^2_{L^{\infty}(0,T;L^2(\Omega))}+ \|\phi\|^2_{L^2(0,T;H^1(\Omega))} \leq \left(\|\widetilde{m}(T)-m_\Omega\|^2_{L^2(\Omega)} +  \|\nabla \widetilde{m}-\nabla m_d\|^6_{L^6(0,T;L^6(\Omega))}\right) }\nonumber\\
		&&\hspace{0.5cm}\times \ \exp\bigg\{C  \left(\|\widetilde{m}\|^2_{L^\infty(0,T;H^2(\Omega))}\|\widetilde{m}\|^2_{L^2(0,T;H^3(\Omega))}+\|\widetilde{m}\|^2_{L^\infty(0,T;H^2(\Omega))} \|\widetilde{u}\|^2_{L^2(0,T;H^1(\Omega))}\right) \bigg\}.
	\end{eqnarray}
\end{Thm}

By virtue of the global existence of regular solution, we can define a \textit{reduced cost functional} $\mathcal{I}  : \mathcal{U}_{ad} \to \mathbb{R}$ by $\mathcal{I}(u)=\mathcal J(G(u),u)$. Therefore, the optimal control problem (OCP) can be redefined as follows:
\begin{equation*}\text{(MOCP)}
	\left\{	\begin{array}{lclclc}
		\text{minimize} \ \ \mathcal{I}(u)	\\
		u\in \mathcal{U}_{ad}.
	\end{array}
	\right.
\end{equation*}

Now we are ready to state the first-order optimality condition satisfied by the optimal control $\widetilde{u}\in \mathcal{U}_{ad}$.

\begin{Thm}[First Order Optimality Condition]\label{FOOCT}
	Suppose $\widetilde{u} \in \mathcal{U}_{ad}$ be an optimal control of MOCP with associated state $\widetilde{m}.$ Then there exists a unique element $\phi \in \mathcal{Z}$ corresponding to the admissible pair $(\widetilde{m},\widetilde{u})$ such that the triplet $(\phi,\widetilde{m},\widetilde{u})$ satisfies the weak adjoint formulation \eqref{WAF} and the following variational inequality holds:
%	\begin{equation}\label{FOOC}	
%		\int_{\Omega_T} \widetilde{u}\cdot (u-\widetilde{u})\ dx\  dt+\int_{\Omega_T} \nabla \widetilde{u}\cdot \nabla (u-\widetilde{u})\ dx\  dt+ \int_{\Omega_T} \Big((\phi \times \widetilde{m})+  \widetilde{m} \times (\phi \times \widetilde{m})\Big)\cdot (u-\widetilde{u})\ dx\ dt \geq 0, \ \ \forall \ u \in \mathcal{U}.
%	\int_0^T  \big(\widetilde{u}, (u-\widetilde{u}) \big)\  dt+\int_0^T  \big(\nabla \widetilde{u}, \nabla (u-\widetilde{u})\big)\  dt + \int_0^T  \Big((\phi \times \widetilde{m})+  \widetilde{m} \times (\phi \times \widetilde{m}), (u-\widetilde{u})\Big)\ dt \geq 0, \ \ \forall \ u \in \mathcal{U}_{ad}.	\end{equation}
\begin{eqnarray}\label{FOOC}
	\lefteqn{\int_0^T  \big(\widetilde{u}, (u-\widetilde{u}) \big)\  dt+\int_0^T  \big(\nabla \widetilde{u}, \nabla (u-\widetilde{u})\big)\  dt} \nonumber \\
	&&+ \int_0^T  \Big((\phi \times \widetilde{m})+  \widetilde{m} \times (\phi \times \widetilde{m}), (u-\widetilde{u})\Big)\ dt \geq 0, \ \ \ \ \  \forall \ u \in \mathcal{U}_{ad}.	
\end{eqnarray}
\end{Thm}

\noindent Next, we will discuss a local second-order optimality condition satisfied by an optimal control.

Let $\widetilde{u} \in \mathcal{U}_{ad}$ satisfy the first order necessary optimality condition given by the variational inequality \eqref{FOOC} and suppose $\mathcal{J}''(\widetilde{u})[h,h]>0$ for all directions $h\in L^2(0,T;H^1(\Omega)) \ \backslash \ \{0\}$. Then $\widetilde{u}$ is a strict local minimum of the functional $\mathcal{J}$ on the set $\mathcal{U}_{ad}$. But such second order conditions are too restrictive as it requires $h$ to be all controls in $L^2(0,T;H^1(\Omega))$ excluding $\{0\}$. However, we can avoid this restriction by defining a cone of critical directions which will suffices the needs for local optimality. For more details on the cone of critical direction, one may refer to Chapter 4 of \cite{FT}.

\begin{Def}{(Critical Cone)}
	For any control $\widetilde{u}\in \mathcal{U}_{ad}$, let $\Lambda(\widetilde{u})$ denotes the set of all $u\in L^2(0,T;H^1(\Omega))$ such that for almost all $(x,t)\in \Omega_T$, 
	$$u(x,t)=
	\begin{cases}
		\geq 0 \ \ \ \text{if}\ \ \widetilde{u}(x,t)=	a(x,t),\\
		\leq 0 \ \ \ \text{if}\ \ \widetilde{u}(x,t)= b(x,t).
		%	=0     \ \ \ \text{if}\  \ \widetilde{u}+ (\phi \times \widetilde{m}) + \big(\widetilde{m} \times (\phi \times \widetilde{m})\big) \neq 0.
	\end{cases}
	$$
	The set of controls $\Lambda(\widetilde{u})$ is called the cone of critical directions.
\end{Def}

In order to show the second order optimality condition, we need the  Fr\'echet differentiability of both the control-to-state and control-to-costate operators(defined later), and that of their Lipschitz continuity. With all these results, we established the second-order conditions for local optimality as follows.

\begin{Thm}[Local Optimality Condition]\label{T-SOLO}
	Let $\widetilde{u}\in \mathcal{U}_{ad}$ be any control satisfying the variational inequality \eqref{FOOC}. Suppose there exists a constant $\delta>0$ such that $\mathcal{J}''(\widetilde{u})[h,h]\geq \delta \ \|h\|^2_{L^2(0,T;H^1(\Omega))}$ for all $h\in \Lambda(\widetilde{u})$, that is,
	\begin{flalign}\label{AOSOD}
		&\int_{\Omega_T} h^2\ dx\ dt+\int_{\Omega_T} |\nabla h|^2\ dx\ dt+ \int_{\Omega_T} \big(\phi'_{\widetilde{u}}[h]\times \widetilde{m}\big) \cdot h\ dx\ dt + \int_{\Omega_T} \big(\phi_{\widetilde{u}} \times m'_{\widetilde{u}}[h]\big)\cdot h\ dx\ dt\nonumber\\
		&\hspace{0.8cm} + \int_{\Omega_T} \big(m'_{\widetilde{u}}[h] \times (\phi_{\widetilde{u}}\times m_{\widetilde{u}})\big)\cdot h \ dx\ dt
		+ \int_{\Omega_T} \big(m_{\widetilde{u}} \times (\phi'_{\widetilde{u}}[h] \times m_{\widetilde{u}})\big)\cdot h\ dx\ dt \nonumber\\
		&\hspace{0.8cm}+ \int_{\Omega_T} \big(m_{\widetilde{u}} \times (\phi_{\widetilde{u}} \times m'_{\widetilde{u}}[h])\big)\cdot h\ dx\ dt \geq \delta \int_{\Omega_T} h^2 \ dx\ dt +\delta  \int_{\Omega_T} |\nabla h|^2 \ dx\ dt\ \ \ \ \ \ \ \forall \ h\in \Lambda(\widetilde{u}).
	\end{flalign}
	Then there exist $\epsilon>0$ and $\sigma>0$ such that for every $u\in\mathcal{U}_{ad}$ with $\|u-\widetilde{u}\|_{L^2(0,T;H^1(\Omega))} \leq \epsilon$, we have
	\begin{equation*}\label{QGC}
		\mathcal{J}(u) \geq \mathcal{J}(\widetilde{u}) + \sigma \ \|u-\widetilde{u}\|^2_{L^2(0,T;L^2(\Omega))}.
	\end{equation*} 
	In particular, this implies that $\widetilde{u}$ is a local minimum of the functional $\mathcal{J}$ on the set of admissible control $\mathcal{U}_{ad}$.
\end{Thm}

Theorem \ref{T-SOLO} establishes that a control input $u\in \mathcal{U}_{ad}$, which satisfies the variational inequality along with a second-order sufficient condition \eqref{AOSOD} defined on a cone of critical directions, acts as a local minimizer for the functional $\mathcal{J}$. However, the global optimality and uniqueness of such a control input for MOCP are still uncertain. To address these questions, we employ a methodology inspired by the work presented in \cite{ADH} for a semilinear elliptic control problem, and for related results refer to \cite{MEPK}, \cite{SK}. This approach allows us to establish that an admissible control input that satisfies the variational inequality \eqref{FOOC} and a condition involving the adjoint solution, yields a global optimal control for the MOCP, which is another major contribution of this paper.

\begin{Thm}[Global Optimality Condition]\label{T-GO}
	Let $\widetilde{u} \in \mathcal{U}_{ad}$ be a control with the associated state $\widetilde{m}\in \mathcal{M}$ and the adjoint state $\phi\in \mathcal{Z}$. Suppose that the triplet $(\widetilde{u},\widetilde{m},\phi)$ satisfies the variational inequality \eqref{FOOC}, and the adjoint system fulfills the following condition:
	\begin{equation}\label{GO-C}
		C(\Omega,T)\ \left\{ 1+ \|\widetilde{m}\|_{L^\infty(0,T;H^2(\Omega))} +\|\widetilde{u}\|_{L^2(0,T;H^1(\Omega))} \right\} \ \|\phi\|_{L^2(0,T;L^2(\Omega))} \leq \frac{1}{2}.
	\end{equation}
	Then $\widetilde{u}\in \mathcal{U}_{ad}$ is a global optimal control of MOCP. Moreover, if the inequality in condition \eqref{GO-C} is strict, then the global optimum $\widetilde{u}$ is unique.
\end{Thm}

\subsection{Function Spaces and Inequalities}\label{FS}

In this subsection, we give some basic cross-product properties in Lemma \ref{CPP}, the equality of norms in Lemma \ref{EN}, and some norm estimates in Lemma \ref{PROP2} and \ref{L-CP}, that we have used throughout this paper. 

\begin{Lem}\label{CPP}
	Let $a,b$ and $c$ be three vectors of $\mathbb{R}^3$, then the following vector identities hold: $a\cdot(b \times c)=-(b \times a)\cdot c$,  $a \cdot (a \times b)=0$,  $a \times (b \times c)=(a \cdot c) b - (a \cdot b) c.$ Moreover, assume that $1 \leq r,s \leq \infty, \ (1/r)+(1/s)=1$ and $p\geq 1$, then if $f\in L^{pr}(\Omega)$ and $g\in L^{ps}(\Omega),$ we have
	\begin{equation}\label{ES0}
	\|f \times g\|_{L^p(\Omega)} \leq \|f\|_{L^{pr}(\Omega)} \|g\|_{L^{ps}(\Omega)}.	
	\end{equation}	
\end{Lem}
The proof of estimate \eqref{ES0} can be readily derived by applying H\"older's inequality.

The $L^2$ theory of Laplace operator with Neumann  boundary condition leads to the following inequality of norms that will be quite useful. 
\begin{Lem}[see, \cite{KW}]\label{EN}
	Let $\Omega$ be a bounded smooth domain in $\mathbb{R}^n$ and $k \in \mathbb{N}$. There exists a constant $C_{k,n}>0$ such that for all $m \in H^{k+2}(\Omega)$ and $\frac{\partial m}{\partial \eta}\big|_{\partial \Omega}=0,$ it holds that
	\begin{eqnarray*}\label{ES1}
		\|m\|_{H^{k+2}(\Omega)} \leq C_{k,n} \left(\|m\|_{L^2(\Omega)}+ \|\Delta m\|_{H^k(\Omega)}\right).	
	\end{eqnarray*}
\end{Lem}
\noindent As a consequence of Lemma \ref{EN}, we can define an equivalent norm on $H^{k+2}(\Omega)$ as follows 
$$\|m\|_{H^{k+2}(\Omega)}:=\|m\|_{L^2(\Omega)}+\|\Delta m\|_{H^k(\Omega)}.$$
While showing the existence of a regular solution, we have used the following estimates.
\begin{Pro}\label{P-GNI}
	Let $\Omega$ be a regular bounded subset of  $\mathbb{R}^2$. There exists a constant $C>0$ depending on $\Omega$ such that for all $m \in H^2(\Omega)$ with $\frac{\partial m}{\partial \eta}\big|_{\partial\Omega}=0,$ we have
	\begin{eqnarray} 
		\|\nabla m\|_{L^s(\Omega)} &\leq& C\  \|\Delta m\|_{L^2(\Omega)}, \ \ \forall \ s \in [1,\infty),\label{ES3}\\
		\|D^2m\|_{L^2(\Omega)} &\leq& C\ \|\Delta m\|_{L^2(\Omega)},\label{ES4}	\\
		\|\nabla m\|_{L^6(\Omega)} &\leq& C\ \|\nabla m\|^{\frac{1}{3}}_{L^2(\Omega)} \|\Delta m\|^{\frac{2}{3}}_{L^2(\Omega)},\label{ES6}\\
		\|\nabla m\|_{L^\infty(\Omega)} &\leq& C\ \|\nabla m\|^{\frac{1}{2}}_{L^2(\Omega)} \|\nabla \Delta m\|^{\frac{1}{2}}_{L^2(\Omega)}.\label{ES8}	
	\end{eqnarray}
\end{Pro}
\noindent The proof of Proposition \ref{P-GNI} can be found in \cite{SPSK}. 

 In order to calculate the $L^2(0,T;H^1(\Omega))$ norm of different cross product and non-linear terms, we will use the following inequalities throughout the paper frequently. The constant $C>0$ may differ from one estimate to another estimate in Lemma \ref{PROP2} and \ref{L-CP}.
\begin{Lem}\label{PROP2}
	Let $\Omega$ be a regular bounded domain of $\mathbb{R}^2$. Then there exists a constant $C>0$ depending on $\Omega$ and $T$ such that 
	\begin{enumerate}[(\roman*)]
	%	\item for $p\in L^\infty(0,T;H^2(\Omega))$ and $q\in L^2(0,T;H^2(\Omega))$,
	%	\begin{equation}\label{EE-1}
	%			\|p \times q\|^2_{L^2(0,T;H^1(\Omega))} \leq  C\ \|p\|^2_{L^\infty(0,T;H^2(\Omega))} \|q\|^2_{L^2(0,T;H^2(\Omega))},
	%	\end{equation}
	%	\begin{equation}\label{EE-2}
	%			\|p \times \Delta q\|^2_{L^2(0,T;H^1(\Omega))}\leq  C\ \|p\|^2_{L^\infty(0,T;H^2(\Omega))} \|q\|^2_{L^2(0,T;H^3(\Omega))},
	%	\end{equation}
		\item  for $\xi\in L^\infty(0,T;H^2(\Omega))$ and $\zeta\in L^2(0,T;H^3(\Omega))$,
		\begin{equation}\label{EE-2}
		\|\xi \times \Delta \zeta\|^2_{L^2(0,T;H^1(\Omega))}\leq  C\ \|\xi\|^2_{L^\infty(0,T;H^2(\Omega))}\  \|\zeta\|^2_{L^2(0,T;H^3(\Omega))},
		\end{equation}
		\item for $\xi\in L^\infty(0,T;H^2(\Omega))$ and $\zeta \in L^2(0,T;H^1(\Omega))$,
		\begin{equation}\label{EE-3}
				\|\xi \times \zeta\|^2_{L^2(0,T;H^1(\Omega))} \leq  C\ \|\xi\|^2_{L^\infty(0,T;H^2(\Omega))}\  \|\zeta\|^2_{L^2(0,T;H^1(\Omega))},
		\end{equation}
		\item  for $\xi\in L^2(0,T;H^3(\Omega))$, $\zeta\in L^\infty(0,T;H^2(\Omega))$ and $\omega\in L^\infty(0,T;H^2(\Omega))$,
		\begin{equation}\label{EE-4}
				\|(\nabla \xi\cdot \nabla \zeta)\ \omega\|^2_{L^2(0,T;H^1(\Omega))}\leq C \  \|\xi\|^2_{L^2(0,T;H^3(\Omega))} \ \|\zeta\|^2_{L^\infty(0,T;H^2(\Omega))}\  \|\omega\|^2_{L^\infty(0,T;H^2(\Omega))},
		\end{equation} 
		\item  for $\xi\in L^\infty(0,T;H^2(\Omega))$, $\zeta\in L^\infty(0,T;H^2(\Omega))$ and $\omega \in L^2(0,T;H^1(\Omega))$,
		\begin{equation}\label{EE-5}
				\|\xi \times (\zeta \times \omega)\|^2_{L^2(0,T;H^1(\Omega))} \leq C \ \|\xi\|^2_{L^\infty(0,T;H^2(\Omega))} \ \|\zeta\|^2_{L^\infty(0,T;H^2(\Omega))}\  \|\omega\|^2_{L^2(0,T;H^1(\Omega))} .
		\end{equation}
	\end{enumerate}
\end{Lem}
\begin{proof}
%	\noindent (i) \ For the first estimate, an application of H\"older's inequality and the embedding $H^1(\Omega) \hookrightarrow L^4(\Omega)$ gives
%	\begin{flalign*}
%	 \|p \times q\|^2_{L^2(0,T;H^1(\Omega))} & \leq  \int_0^T \int_{\Omega} |p|^2 | q|^2 dx\ dt+ \int_0^T  \int_{\Omega} |\nabla p|^2 |q|^2 dx\ dt+ \int_0^T \int_{\Omega} |p|^2 |\nabla q|^2 dx\ dt&\\
%	&\leq  \int_0^T \|p\|^2_{L^4(\Omega)} \|q\|^2_{L^4(\Omega)} dt+ \int_0^T \|\nabla p\|^2_{L^4(\Omega)}\|q\|^2_{L^4(\Omega)} dt+ \int_0^T \|p\|^2_{L^4(\Omega)} \|\nabla  q\|^2_{L^4(\Omega)} dt\\
%	&\leq  \int_0^T C \|p\|^2_{H^2(\Omega)}\|q\|^2_{H^3(\Omega)}\\
%	&\leq  C\ \|p\|^2_{L^\infty(0,T;H^2(\Omega))} \|q\|^2_{L^2(0,T;H^2(\Omega))}.
%\end{flalign*}
	 For the first and second estimates, applying H\"older's inequality with embeddings $H^1(\Omega)\hookrightarrow L^4(\Omega)$ and $H^2(\Omega)\hookrightarrow L^{\infty}(\Omega)$, we find
\begin{flalign*}
 	\|\xi \times \Delta \zeta\|&^2_{L^2(0,T;H^1(\Omega))} 
 	% \leq  \int_0^T \int_{\Omega} |\xi|^2 |\Delta \zeta|^2 dx\ dt+2  \int_0^T  \int_{\Omega} |\nabla \xi|^2 |\Delta \zeta|^2 dx\ dt+2 \int_0^T  \int_{\Omega} |\xi|^2 |\nabla \Delta \zeta|^2 dx\ dt&\\
	\leq  \int_0^T \|\xi\|^2_{L^\infty(\Omega)} \|\Delta \zeta\|^2_{L^2(\Omega)} \ dt+ 2\int_0^T \|\nabla \xi\|^2_{L^4(\Omega)}\|\Delta \zeta\|^2_{L^4(\Omega)} dt&\\
	&\hspace{2cm}+ 2\int_0^T \|\xi\|^2_{L^\infty(\Omega)} \|\nabla \Delta \zeta\|^2_{L^2(\Omega)}\ dt
	%	&\leq  \int_0^T C \|\xi\|^2_{H^2(\Omega)}\|\zeta\|^2_{H^3(\Omega)}\\
	\leq  C\ \|\xi\|^2_{L^\infty(0,T;H^2(\Omega))}\  \|\zeta\|^2_{L^2(0,T;H^3(\Omega))}.
\end{flalign*}
%and	\begin{flalign*}
%	 \|\xi \times \zeta\|^2_{L^2(0,T;H^1(\Omega))} & \leq  \int_0^T \int_{\Omega} |\xi|^2 |\zeta|^2 dx\ dt+2\int_0^T  \int_{\Omega} |\nabla \xi|^2 |\zeta|^2 dx\ dt+ 2\int_0^T \int_{\Omega} |\xi|^2 |\nabla \zeta|^2 dx\ dt&\\
%	&\leq  \int_0^T \|\xi\|^2_{L^4(\Omega)} \|\zeta\|^2_{L^4(\Omega)} dt+ 2\int_0^T \|\nabla \xi\|^2_{L^4(\Omega)}\|\zeta\|^2_{L^4(\Omega)} dt+ 2\int_0^T \|\xi\|^2_{L^\infty(\Omega)} \|\nabla  \zeta\|^2_{L^2(\Omega)} dt\\
	%	&\leq  \int_0^T C \|\xi\|^2_{H^2(\Omega)}\|\zeta\|^2_{H^3(\Omega)}\\
%	&\leq  C\ \|\xi\|^2_{L^\infty(0,T;H^2(\Omega))} \ \|\zeta\|^2_{L^2(0,T;H^1(\Omega))}.
%\end{flalign*}
This proves the inequality (i). The proof for (ii) can be constructed using a comparable line of reasoning as the proof for (i). Finally, for the last two estimates, implementing H\"older's inequality followed by the embeddings $H^1(\Omega)\hookrightarrow L^p(\Omega)$ for $p\in [1,\infty)$ and $H^2(\Omega)\hookrightarrow L^{\infty}(\Omega)$, we derive
\begin{flalign*}
	\|(\nabla \xi\cdot \nabla \zeta)\ \omega\|&^2_{L^2(0,T;H^1(\Omega))} 
	%\leq  \int_0^T \int_{\Omega} |\nabla \xi|^2 |\nabla \zeta|^2 |\omega|^2 dx\  dt + 3 \int_0^T \int_{\Omega} |D^2 \xi|^2 |\nabla \zeta|^2 |\omega|^2 dx\  dt&\\
	%&\ \ \ \ +3\int_0^T \int_{\Omega} |\nabla \xi|^2 |D^2 \zeta|^2 |\omega|^2 dx \ dt + 3\int_0^T \int_{\Omega} |\nabla \xi|^2 |\nabla \zeta |^2 |\nabla \omega|^2 dx \ dt\\
	\leq \int_0^T  \|\nabla \xi\|^2_{L^4(\Omega)}  \|\nabla \zeta\|^2_{L^8(\Omega)} \|\omega\|^2_{L^8(\Omega)} dt + 3\int_0^T \|D^2 \xi\|^2_{L^4(\Omega)}  \|\nabla \zeta \|^2_{L^8(\Omega)} \|\omega\|^2_{L^8(\Omega)} dt\\
	&\ \ \ +3\int_0^T \|\nabla \xi\|^2_{L^\infty(\Omega)}  \|D^2 \zeta\|^2_{L^2(\Omega)} \|\omega\|^2_{L^\infty(\Omega)} dt +3 \int_0^T   \|\nabla \xi\|^2_{L^4(\Omega)}  \|\nabla \zeta\|^2_{L^8(\Omega)} \|\nabla \omega\|^2_{L^8(\Omega)} dt\\
%	&\leq C \int_0^T  \|\xi\|^2_{H^3(\Omega)}  \|\zeta\|^2_{H^2(\Omega)} \|\omega\|^2_{H^2(\Omega)}  dt\\
	&\leq C \  \|\xi\|^2_{L^2(0,T;H^3(\Omega))} \ \|\zeta\|^2_{L^\infty(0,T;H^2(\Omega))} \ \|\omega\|^2_{L^\infty(0,T;H^2(\Omega))},
\end{flalign*}
\begin{flalign*}
	\text{and}\ \|\xi \times (\zeta \times \omega)\|&^2_{L^2(0,T;H^1(\Omega))} 
	%\leq \int_0^T \int_{\Omega} |\xi|^2 |\zeta|^2 |\omega|^2 dx\ dt+3 \int_0^T \int_{\Omega} |\nabla \xi|^2 |\zeta|^2 |\omega|^2 dx\ dt&\\
	%&\ \ \ + 3\int_0^T \int_{\Omega} |\xi|^2 |\nabla\zeta|^2 |\omega|^2 dx\ dt+3 \int_0^T \int_{\Omega} |\xi|^2 |\zeta|^2 |\nabla \omega|^2 dx\ dt \\
	\leq \int_0^T \|\xi\|^2_{L^4(\Omega)} \|\zeta\|^2_{L^8(\Omega)} \|\omega\|^2_{L^8(\Omega)} dt +3 \int_0^T \|\nabla \xi\|^2_{L^4(\Omega)} \|\zeta\|^2_{L^8(\Omega)} \|\omega\|^2_{L^8(\Omega)} dt &\\
	&\ \ \ +3\int_0^T \|\xi\|^2_{L^4(\Omega)} \|\nabla\zeta\|^2_{L^8(\Omega)} \|\omega\|^2_{L^8(\Omega)} dt + 3\int_0^T \|\xi\|^2_{L^\infty(\Omega)} \|\zeta\|^2_{L^\infty (\Omega)} \|\nabla \omega\|^2_{L^2(\Omega)} dt\\
	%&\leq C \int_0^T \|\xi\|^2_{H^2(\Omega)} \|\zeta\|^2_{H^2(\Omega)} \|\omega\|^2_{H^1(\Omega)} dt\\
	&\leq C \ \|\xi\|^2_{L^\infty(0,T;H^2(\Omega))} \  \|\zeta\|^2_{L^\infty(0,T;H^2(\Omega))}\  \|\omega\|^2_{L^2(0,T;H^1(\Omega))} .
\end{flalign*}
Hence the proof of (iii) and (iv).
\end{proof}
Moreover, to estimate the $L^2(0,T;H^1(\Omega)^*)$ norm of various terms, we require the following lemma.
	\begin{Lem}\label{L-CP}	
		Let $\Omega$ be a bounded subset of $\mathbb{R}^2$ with smooth boundary. Then there exists a constant $C>0$ such that 
	\begin{enumerate}[(\roman*)]
		\item for $\xi \in L^\infty(0,T;H^2(\Omega))$, $\zeta\in L^\infty(0,T;H^2(\Omega))$ and $\omega\in L^2(0,T;L^2(\Omega))$, 
		\begin{equation}\label{AEE-1}
			\|\big( \nabla \xi \cdot \nabla \zeta \big) \omega\|_{L^2(0,T;H^1(\Omega)^*)} \leq  C\ \|\xi\|_{L^\infty(0,T;H^2(\Omega))} \|\zeta\|_{L^\infty(0,T;H^2(\Omega))} \|\omega\|_{L^2(0,T;L^2(\Omega))},
		\end{equation}
		\item  for $\xi \in L^\infty(0,T;H^2(\Omega))$, $\zeta\in L^\infty(0,T;L^2(\Omega))$ and $\omega \in L^2(0,T;H^3(\Omega))$ with $\frac{\partial \omega}{\partial \eta}=0$,
		\begin{equation}\label{AEE-2}
			\|\nabla \cdot \big((\xi\cdot \zeta)\nabla \omega\big)\|_{L^2(0,T;H^1(\Omega)^*)}\leq C\ \|\xi\|_{L^\infty(0,T;H^2(\Omega))} \|\zeta\|_{L^\infty(0,T;L^2(\Omega))} \|\omega\|_{L^2(0,T;H^3(\Omega))},
		\end{equation}
		\item for $\xi \in L^2(0,T;H^1(\Omega))$ and $\zeta \in L^\infty(0,T;H^2(\Omega))$ with $\frac{\partial \xi}{\partial \eta}=\frac{\partial \zeta}{\partial \eta}=0$,
		\begin{equation}\label{AEE-3}
			\|\Delta (\xi\times \zeta)\|_{L^2(0,T;H^1(\Omega)^*)} \leq  C\ \|\xi\|_{L^2(0,T;H^1(\Omega))} \|\zeta\|_{L^\infty(0,T;H^2(\Omega))},
		\end{equation}
		\item  for $\xi \in L^\infty(0,T;H^2(\Omega))$ and $\zeta\in L^2(0,T;H^1(\Omega))$,
		\begin{equation}\label{AEE-4}
			\|\Delta \xi\times \zeta \|_{L^2(0,T;H^1(\Omega)^*)}\leq C\ \|\xi\|_{L^\infty(0,T;H^2(\Omega))} \|\zeta\|_{L^2(0,T;H^1(\Omega))},
		\end{equation} 
	    \item for $\xi\in L^\infty(0,T;L^2(\Omega))$ and $\zeta\in L^2(0,T;H^1(\Omega))$,
	    \begin{equation}\label{AEE-8}
	    	\|\xi\times \zeta\|_{L^2(0,T;H^1(\Omega)^*)}\leq C\ \|\xi\|_{L^\infty(0,T;L^2(\Omega))}\ \|\zeta\|_{L^2(0,T;H^1(\Omega))},
	    \end{equation}
		\item  for $\xi\in L^\infty(0,T;L^2(\Omega))$, $\zeta \in L^\infty(0,T;H^2(\Omega))$ and $\omega \in L^2(0,T;H^1(\Omega))$,
		\begin{equation}\label{AEE-5}
			\|\big(\xi \times \zeta\big) \times \omega\|_{L^2(0,T;H^1(\Omega)^*)} \leq C\ \|\xi\|_{L^\infty(0,T;L^2(\Omega))} \|\zeta\|_{L^\infty(0,T;H^2(\Omega))} \|\omega\|_{L^2(0,T;H^1(\Omega))},
		\end{equation}
		\item  for $\xi\in L^\infty(0,T;L^2(\Omega))$, $\zeta \in L^\infty(0,T;H^2(\Omega))$ and $\omega \in L^2(0,T;H^1(\Omega))$,
		\begin{equation}\label{AEE-6}
			\| \xi \times\big(\zeta \times \omega\big)\|_{L^2(0,T;H^1(\Omega)^*)} \leq  C\ \|\xi\|_{L^\infty(0,T;L^2(\Omega))} \|\zeta\|_{L^\infty(0,T;H^2(\Omega))} \|\omega\|_{L^2(0,T;H^1(\Omega))},
		\end{equation}
		\item  for $\xi \in L^6(0,T;L^6(\Omega))$, $\zeta \in L^6(0,T;L^6(\Omega))$ and $\omega \in L^6(0,T;L^6(\Omega))$ with $\frac{\partial \omega}{\partial \eta}=0$,
		\begin{equation}\label{AEE-7}
			\|\nabla \cdot \big((\nabla \xi\cdot \nabla \zeta)\nabla \omega\big)\|_{L^2(0,T;H^1(\Omega)^*)} \leq C\ \|\nabla \xi\|_{L^6(0,T;L^6(\Omega))} \|\nabla \zeta\|_{L^6(0,T;L^6(\Omega))} \|\nabla \omega\|_{L^6(0,T;L^6(\Omega))},
		\end{equation}	
	\end{enumerate}
\end{Lem}

\begin{proof}
	Let $\upsilon$ be an arbitrary element in $H^1(\Omega)$. For the estimates (i)-(v), we will apply H\"older's inequality followed by embeddings $H^1(\Omega) \hookrightarrow L^4(\Omega)$, $H^1(\Omega)\hookrightarrow L^8(\Omega)$ and $H^2(\Omega)\hookrightarrow L^{\infty}(\Omega)$ to obtain the following estimates:
			\begin{flalign*}
			(i)\ \ \  &\langle ( \nabla \xi\cdot \nabla \zeta )\ \omega ,\upsilon \rangle  =  \int_{\Omega} ( \nabla \xi\cdot \nabla \zeta )\omega \cdot \upsilon\ dx
			\leq  \|\nabla \xi\|_{L^8(\Omega)} \|\nabla \zeta\|_{L^8(\Omega)} \|\omega\|_{L^2(\Omega)} \|\upsilon\|_{L^4(\Omega)} \hspace{1.5in} \\  
			&\hspace{2.6in} \leq C\  \| \xi\|_{H^2(\Omega)} \|\zeta\|_{H^2(\Omega)} \|\omega\|_{L^2(\Omega)} \|\upsilon\|_{H^1(\Omega)},  
			\end{flalign*}
			which leads to the estimate
			$$\int_0^T \|\big(\nabla \xi\cdot \nabla \zeta\big)\ w\|^2_{H^1(\Omega)^*} \ dt \leq C\ \|\xi\|^2_{L^\infty(0,T;H^2(\Omega))} \ \|\zeta\|^2_{L^\infty(0,T;H^2(\Omega))} \ \|\omega\|^2_{L^2(0,T;L^2(\Omega))}.$$
%		Now, by taking the time integration of the square of $\|\big(\nabla \xi\cdot \nabla \zeta\big)\ w\|_{H^1(\Omega)^*}$, we find
\begin{flalign*}
	(ii) \ \ &\langle \nabla \cdot \big((\xi\cdot \zeta)\nabla \omega\big) ,\upsilon \rangle  = -  \int_{\Omega} \big((\xi\cdot \zeta)\nabla \omega\big)\cdot \nabla \upsilon\ dx \leq  \| \xi\|_{L^\infty(\Omega)} \|\zeta\|_{L^2(\Omega)} \|\nabla \omega\|_{L^\infty(\Omega)} \|\nabla \upsilon\|_{L^2(\Omega)},\hspace{1.5in}\\
	 &\int_0^T \|\nabla \cdot \big((\xi\cdot \zeta)\nabla \omega\big)\|^2_{H^1(\Omega)^*}\ dt \leq C\ \|\xi\|^2_{L^\infty(0,T;H^2(\Omega))}  \ \|\zeta\|^2_{L^\infty(0,T;L^2(\Omega))} \  \|\omega\|^2_{L^2(0,T;H^3(\Omega))}.
	 \end{flalign*} 
%and hence we find
	\begin{flalign*}
		(iii)\ \ \ &\langle \Delta (\xi\times \zeta), \upsilon \rangle  = - \int_{\Omega} \nabla (\xi\times \zeta)\cdot \nabla \upsilon \ dx \leq  \|\nabla \xi\|_{L^2(\Omega)} \|\zeta\|_{L^\infty(\Omega)} \|\nabla \upsilon\|_{L^2(\Omega)}  +  \|\xi\|_{L^4(\Omega)} \|\nabla \zeta\|_{L^4(\Omega)} \|\nabla \upsilon\|_{L^2(\Omega)}&\\
		&\hspace{2.5in}\leq C\  \|\xi\|_{H^1(\Omega)} \ \|\zeta \|_{H^2(\Omega)}\ \|\upsilon\|_{H^1(\Omega)},\\ 
		&\int_0^T \|\Delta (\xi \times \zeta)\|^2_{H^1(\Omega)^*} \ dt\leq C\ \|\xi\|^2_{L^2(0,T;H^1(\Omega))} \ \|\zeta\|^2_{L^\infty(0,T;H^2(\Omega))}.
	\end{flalign*}
	\begin{flalign*}
	(iv)\ \ \ \	&\langle \Delta \xi\times \zeta, \upsilon \rangle\  =  \int_{\Omega} \big(\Delta \xi\times \zeta\big) \cdot \upsilon\ dx\leq  \|\Delta \xi\|_{L^2(\Omega)} \|\zeta\|_{L^4(\Omega)} \|\upsilon\|_{L^4(\Omega)} \leq C\  \| \xi\|_{H^2(\Omega)} \|\zeta\|_{H^1(\Omega)} \|\upsilon\|_{H^1(\Omega)},&\\
	& \int_0^T \|\Delta \xi \times \zeta\|^2_{H^1(\Omega)^*}\ dt\leq C\ \|\xi\|^2_{L^\infty(0,T;H^2(\Omega))}\ \|\zeta\|^2_{L^2(0,T;H^1(\Omega))}.
	\end{flalign*}
%	\begin{flalign*}
%		(v)\ \ \ &\langle  \xi\times \zeta, \upsilon \rangle =  \int_{\Omega} \big( \xi\times \zeta\big) \cdot \upsilon\ dx \leq   \| \xi\|_{L^2(\Omega)} \|\zeta\|_{L^4(\Omega)} \|\upsilon\|_{L^4(\Omega)} \leq C\ \|\xi\|_{L^2(\Omega)}\ \|\zeta\|_{H^1(\Omega)} \|\upsilon\|_{H^1(\Omega)}&\\
%		&\implies \int_0^T \|\xi \times \zeta \|^2_{H^1(\Omega)^*}\ dt \leq C\ \|\xi\|^2_{L^\infty(0,T;L^2(\Omega))}\ \|\zeta\|^2_{L^2(0,T;H^1(\Omega))}.
%	\end{flalign*}
%With this we have completed the proof of (i)-(v). 

The proof of (v) can be established using a similar reasoning as applied in the proof of (iv). Now, for the last three estimates (vi)-(vii), again appealing to H\"older's inequality and implementing the embeddings $H^1(\Omega)\hookrightarrow L^p(\Omega)$ for $p\in [1,\infty)$ and $H^2(\Omega) \hookrightarrow L^{\infty}(\Omega)$, followed by time integration as before, we derive
	\begin{flalign*}
		(vi)\ \ \ &\langle (\xi\times \zeta)\times \omega, \upsilon\rangle   =  \int_{\Omega} \big((\xi\times \zeta)\times \omega\big)\cdot \upsilon\ dx \leq  \|\xi\|_{L^2(\Omega)} \|\zeta\|_{L^\infty(\Omega)} \|\omega\|_{L^4(\Omega)} \|\upsilon\|_{L^4(\Omega)} &\\
		&\hspace{2.7in}\leq C\  \|\xi\|_{L^2(\Omega)} \|\zeta\|_{H^2(\Omega)} \|\omega\|_{H^1(\Omega)} \|\upsilon\|_{H^1(\Omega)},\\
		& \int_0^T \|(\xi \times \zeta)\times \omega\|^2_{H^1(\Omega)^*}\ dt \leq C\ \|\xi\|^2_{L^\infty(0,T;L^2(\Omega))} \  \|\zeta\|^2_{L^\infty(0,T;H^2(\Omega))} \  \|\omega\|^2_{L^2(0,T;H^1(\Omega))}.
	\end{flalign*}
%	\begin{flalign*}
%		(vii)\ \ \ &\langle \xi\times (\zeta\times \omega), \upsilon\rangle   =  \int_{\Omega} \big(\xi\times( \zeta \times \omega)\big)\cdot \upsilon\ dx \leq  \|\xi\|_{L^2(\Omega)} \|\zeta\|_{L^\infty(\Omega)} \|\omega\|_{L^4(\Omega)} \|\upsilon\|_{L^4(\Omega)} &\\
%		&\hspace{2.6cm}\leq C\  \|\xi\|_{L^2(\Omega)} \|\zeta\|_{H^2(\Omega)} \|\omega\|_{H^1(\Omega)} \|\upsilon\|_{H^1(\Omega)}\\
%		&\implies \int_0^T \|\xi \times (\zeta \times \omega)\|^2_{H^1(\Omega)^*}\ dt \leq C\ \|\xi\|^2_{L^\infty(0,T;L^2(\Omega))}\  \|\zeta\|^2_{L^\infty(0,T;H^2(\Omega))} \ \|\omega\|^2_{L^2(0,T;H^1(\Omega))},
%	\end{flalign*}
	\begin{flalign*}
		(viii)\ \ \ &\langle \nabla \cdot \big((\nabla \xi\cdot \nabla \zeta)\nabla \omega\big), \upsilon \rangle = - \int_{\Omega} \big(\nabla \xi\cdot \nabla \zeta\big)\nabla \omega\cdot \nabla \upsilon \ dx\leq  \|\nabla \xi\|_{L^6(\Omega)} \|\nabla \zeta\|_{L^6(\Omega)} \|\nabla \omega\|_{L^6(\Omega)} \|\nabla \upsilon\|_{L^2(\Omega)}&\\
		&\hspace{3.4in}\leq C\ \|\xi\|_{H^2(\Omega)} \ \|\zeta\|_{H^2(\Omega)}\ \|\omega\|_{H^2(\Omega)} \ \|\upsilon\|_{H^1(\Omega)},\\
		&\int_0^T \|\nabla \cdot \big((\nabla \xi\cdot \nabla \zeta)\nabla \omega\big)\|^2_{H^1(\Omega)^*}\ dt \leq C\ \|\nabla \xi\|^2_{L^6(0,T;L^6(\Omega))} \ \|\nabla \zeta\|^2_{L^6(0,T;L^6(\Omega))}\  \|\nabla \omega\|^2_{L^6(0,T;L^6(\Omega))}.
	\end{flalign*}
The proof for (vii) can be derived using an argument analogous to the one used for (vi). Thus, we have proved all the estimates.
\end{proof}

\section{Weak Solution and Regular Solution}\label{S-WSRS}
\subsection{Weak Solution}
In this section, we establish the existence of a weak solution for \eqref{NLP} in the sense of Definition \ref{WSD}, proving Theorem \ref{EOWS}. We base our arguments on the approaches presented in \cite{FAAS} and \cite{YC}. For adopting the control $u$ that appears non-linearly in \eqref{NLP}, we define an appropriate penalized form and give a brief proof of Theorem \ref{EOWS}. This penalized form allows us to verify the hypothesis $|m|=1$.

%While the proof follows similar lines of reasoning, we introduce a penalized form \eqref{WPP} for completeness and to adapt the control that appears semi-linearly in the state equation. 

%The proof follows the similar ideas developed in 5,15. Since the control appears non-linearly in the state equation \eqref{NLP} we appropriately define a penalized problem and give a short proof on the existence of weak solution.

\begin{proof}[Proof of Theorem \ref{EOWS}]
	Consider the following penalized problem:
	\begin{equation}\label{WPP}
		\begin{cases}
			(m^k)_t  -m^k \times (m^k)_t = 2 \Delta m^k -2k\left(|m^k|^2-1\right)m^k+2u\ \ \ \ \text{in}\ \Omega_T,\\
			\frac{\partial m^k}{\partial \eta}=0 \ \ \ \ \ \ \ \ \ \ \text{on} \ \partial \Omega_T,\\
			m^k(\cdot,0)=m_0 \ \ \text{in} \ \Omega.
		\end{cases}	
	\end{equation}
	Let $\{\xi_i\}_{i=1}^{\infty}$ be an orthonormal basis of $L^2(\Omega)$ consisting of eigenvectors for  $-\Delta$ operator with vanishing Neumann boundary condition. Suppose $W_n=\text{span}\{\xi_1,\xi_2,...,\xi_n\}$ and $\mathbb{P}_n:L^2\to W_n$ be the orthogonal projection.  Then, consider the Galerkin system of \eqref{WPP}
	\begin{equation}\label{WGA}
		\begin{cases}
			\big((m^k_n)_t , \xi_i\big)-\big(m^k_n\times (m^k_n)_t, \xi_i\big) = 2 \left(\Delta m^k_n,\xi_i\right) -2 k \left(\big(|m^k_n|^2-1\big)m^k_n,\xi_i\right) +2\ (u,\xi_i),\\
			(m^k_n(0),\xi_i)=(m_0,\xi_i),
		\end{cases}
	\end{equation}
	where $m^k_n=\sum_{i=1}^{n} a_{ni}(t)\ \xi_i\in W_n$ and $m^k_n(0)=\mathbb{P}_n(m_0)$. 
	% Also, for $i=1,2...,n$, the time-dependent variable $a_{ni}(t)\in \mathbb{R}^3$ for each t. If we define $a(t):=(a^1_{n1}(t),a^2_{n1}(t),a^3_{n1}(t),a^1_{n2}(t),......,a^3_{nn}(t))$, then system \eqref{WGA} is equivalent to the following system of $3n$ ordinary differential equations:
	%\begin{equation}\label{ODE-WF}
	%	\begin{cases}
	%		\frac{da}{dt}-\mathbb{M}(a)\frac{da}{dt}=f(t,a(t),u(t)),\\
	%		a(0)=a_0,
	%	\end{cases}
	%\end{equation}
	%where $a_0$ is the projection of $m_0$ on $(\xi_1,\xi_2,...,\xi_n)$. 
	%Note that, the matrix $\mathbb{M}(a)$ is  skew-symmetric. As a result $\mathbb{I}-\mathbb{M}(a)$ is invertible and hence the solvability of system  \eqref{ODE-WF} directly follows from classical theories of ordinary differential equation.
	%The matrix $\mathbb{M}(a)$ is skew-symmetric, and as a result, the matrix $\mathbb{I}-\mathbb{M}(a)$ is invertible, where $\mathbb{I}$ is the identity matrix. 

    The local solvability of system \eqref{WGA} can be directly inferred from classical theories of ordinary differential equations. For more details on the solvability of such ODE one can refer to \cite{FAAS} and \cite{SPSK}. 	Moreover, by an appropriate \textit{a priori} estimate, we can find that the solution to system \eqref{WGA} exists on the entire time interval $[0,T]$. 
	
	\noindent Now, multiplying system \eqref{WGA} by $(a_{ni})_t$, summing over $0$ to $n$ and then taking an integration over $0$ to $t$, we obtain
%	$$\frac{1}{2}\int_{\Omega}\left| (m^k_n)_t \right|^2dx+  \frac{d}{dt} \left[\int_{\Omega}|\nabla m^k_n|^2 dx+ \frac{k}{2}\int_{\Omega} \left(|m^k_n|^2-1\right)^2 dx\right] \leq 2\  \|u\|^2_{L^2(\Omega)}.$$
%	Then, an integration over $0$ to $t$ leads to
	\begin{eqnarray}
		\lefteqn{\frac{1}{2}\int_0^t\int_{\Omega}\left| (m^k_n)_t \right|^2dx\ dt+  \int_{\Omega}|\nabla m^k_n|^2 dx+ \frac{k}{2}\int_{\Omega} \left(|m^k_n|^2-1\right)^2 dx}\nonumber\\
		&& \leq 2 \int_0^t \|u(s)\|^2_{L^2(\Omega)} ds+ \int_{\Omega}|\nabla m^k_n(0)|^2 dx+ \frac{k}{2} \int_{\Omega} \left(|m^k_n(0)|^2-1\right)^2 dx.\label{WSGAEE}
	\end{eqnarray}
	Using the property $\int_{\Omega}|m^k_n|^2 dx \leq \frac{1}{2} \int_{\Omega}\left(|m^k_n|^2-1\right)^2 dx+ \frac{3}{2}|\Omega|$ and $\|\mathbb{P}_n(m_0)\|_{H^1(\Omega)} \leq  \|m_0\|_{H^1(\Omega)}$(see, Proposition 1, \cite{GCRJ}), we find that $\{m^k_n\}$ and $\{|m^k_n|^2-1\}$ are uniformly bounded in $L^{\infty}(0,T;H^1(\Omega))$ and $L^\infty(0,T;L^2(\Omega))$ respectively. Also, $\left\{(m^k_n)_t \right\}$ is uniformly bounded in $L^2(0,T;L^2(\Omega))$. Now, using reflexive weak compactness and Aubin-Lions-Simon lemma for finding a sub-sequence and following the ideas used in \cite{FAAS}, we can find a weak solution $m^k$ of the penalized form \eqref{WPP} satisfying the energy estimate
	\begin{equation*}%\label{WSE1}
		\frac{1}{2}\int_{\Omega_T}\left| (m^k)_t \right|^2dx\ dt+ \int_{\Omega}|\nabla m^k|^2 dx+ \frac{k}{2}\int_{\Omega} \left(|m^k|^2-1\right)^2 dx \leq  2 \ \|u\|^2_{L^2(0,T;L^2(\Omega))} +\ \|\nabla m_0\|^2_{L^2(\Omega)}.
	\end{equation*} 
	Again using the uniform bound for $m^k$ for finding a sub-sequence that converges to $m$, we can show that $m$  weakly satisfies the system \eqref{NLP} in the sense of Definition \ref{WSD}. For more details of the proof, one may refer to \cite{FAAS}.
\end{proof}	

\subsection{Existence of Regular Solution}\label{SECSS}
  Since, solvability of the linearized system and the adjoint system demands more regularity then that of weak solutions, in this subsection we will prove a condition under which any weak solution of system \eqref{NLP} will be a regular solution. 
\begin{proof}[Proof of Theorem \ref{T-SS}]
	In the previous work, the authors (see, Theorem 2.1, \cite{SPSK}), proved the existence of a unique local regular solution for any control in $L^2(0,T;H^1(\Omega))$ with initial data satisfying condition \eqref{IC}.
	Let $u$ be any control in $L^2(0,T;H^1(\Omega))$ such that the weak solution $m$ satisfies $\|\nabla m\|_{L^4(0,T;L^4(\Omega))}<\infty$. For such a control, suppose $T^*<T$ be the maximum time upto which the regular solution exist. Therefore, taking $L^2$ inner product of \eqref{EP} with $-\Delta m$ and using the fact that $m \cdot \Delta m=-|\nabla m|^2$, we get
	\begin{eqnarray*}
	\lefteqn{\frac{1}{2} \frac{d}{dt} \|\nabla m(t)\|^2_{L^2(\Omega)} + \int_\Omega |\Delta m(t)|^2\ dx }\\
	&&=- \int_\Omega | \nabla m|^2 m\cdot \Delta m\ dx -  \int_\Omega (m \times u)\cdot \Delta m\ dx +  \int_\Omega (m \times (m \times u))\cdot \Delta m\ dx\\
	%		&\leq \ \|\nabla m(t)\|^4_{L^4(\Omega)} + \ \|u(t)\|_{L^2(\Omega)} \|\Delta m(t)\|_{L^2(\Omega)} +\ \|u(t)\|_{L^2(\Omega)} \|\Delta m(t)\|_{L^2(\Omega)}\\
	&&\leq \|\nabla m(t)\|^4_{L^4(\Omega)} + \frac{1}{2} \|\Delta m(t)\|^2_{L^2(\Omega)} + 2\ \|u(t)\|^2_{L^2(\Omega)}, \ \ \ \ \ \ \ \ \ \ \ \ \forall \ t\in [0,T^*).		
	\end{eqnarray*}
By considering the time integration over $0$ to $t$, we find the following estimate for all $t\in[0,T^*)$
	\begin{equation}\label{E2}
	\|\nabla m(t)\|^2_{L^2(\Omega)} + \int_0^t \|\Delta m(\tau)\|^2_{L^2(\Omega)} d\tau \leq  \|\nabla m_0\|^2_{L^2(\Omega)}+ 2 \int_0^t \|\nabla m(t)\|^4_{L^4(\Omega)} dt + 4 \int_0^t \|u(t)\|^2_{L^2(\Omega)} dt.	
\end{equation}
Now, taking `$\nabla$' in both sides of system \eqref{EP}, hitting with $-\nabla \Delta m$, we have
\begin{align}\label{EQ3}
&\frac{1}{2}\frac{d}{dt} \| \Delta m(t)\|^2_{L^2(\Omega)} + \int_{\Omega} |\nabla  \Delta m(t)|^2\ dx 
=- \int_\Omega  \nabla \left(|\nabla m|^2m\right) \cdot \nabla \Delta m
- \int_\Omega  \nabla \big(m \times \Delta m\big) \cdot \nabla \Delta m\ dx\ \ \ \ \ \ \nonumber\\
&\hspace{2cm}- \int_\Omega  \nabla\big(m \times u(t)\big) \cdot \nabla\Delta m\ dx
+ \int_\Omega  \nabla\big(m \times (m \times u(t))\big) \cdot \nabla\Delta m\ dx 
:= \sum_{i=1}^4 E_i.
\end{align}
For the first term $E_1$, employing H\"older's inequlaity, followed by the equality $|m|=1$, and the estimates \eqref{ES4}, \eqref{ES6} and \eqref{ES8}, we derive
\begin{flalign*}
	E_1 &= - \int_\Omega \left[ 2  \nabla m (D^2m)m \cdot \nabla \Delta m \  - |\nabla m|^2 \nabla m \cdot \nabla \Delta m\right] \ dx&\\
	&\leq 2 \ \|\nabla m(t)\|_{L^\infty(\Omega)} \|D^2 m(t)\|_{L^2(\Omega)} \| \nabla \Delta m(t)\|_{L^2(\Omega)} +    \| \nabla m(t)\|^3_{L^6(\Omega)} \| \nabla \Delta m(t)\|_{L^2(\Omega)}\\ 
	&\leq C\ \|\nabla m(t)\|^{\frac{1}{2}}_{L^2(\Omega)} \|\Delta m(t)\|_{L^2(\Omega)} \|\nabla \Delta m(t)\|^{\frac{3}{2}}_{L^2(\Omega)} + C\ \|\nabla m(t)\|_{L^2(\Omega)} \|\Delta m(t)\|^2_{L^2(\Omega)} \|\nabla \Delta m(t)\|_{L^2(\Omega)}\\   	
	&\leq  \epsilon    \int_{\Omega}|\nabla \Delta m(t)|^2dx + C(\epsilon) \ \| \nabla m(t)\|_{L^2(\Omega)}^2\ \|\Delta m(t)\|^4_{L^2(\Omega)}.
\end{flalign*}
For the second term $E_2$,  applying H\"older's inequlaity, the corss-product property $(m \times \nabla\Delta m)\cdot \nabla \Delta m=0$ and estimate \eqref{ES8}, we find 
\begin{flalign*}
	E_2 &= - \int_\Omega (\nabla m \times \Delta m) \cdot\nabla \Delta m \ dx \leq   \|\nabla m(t)\|_{L^\infty(\Omega)} \|\Delta m(t)\|_{L^2(\Omega)} \|\nabla \Delta m(t)\|_{L^2(\Omega)}&\\
	%	&\leq C\ \|\nabla m(t)\|^{\frac{1}{2}}_{L^2(\Omega)} \|\Delta m(t)\|_{L^2(\Omega)} \|\nabla \Delta m(t)\|^{\frac{3}{2}}_{L^2(\Omega)}\\
	&\leq \epsilon \int_{\Omega}|\nabla \Delta m(t)|^2dx + C(\epsilon) \ 
	\|\nabla m(t)\|^2_{L^2(\Omega)} \|\Delta m(t)\|^4_{L^2(\Omega)}.
\end{flalign*}
Similarly, we can obtain the estimates for $E_3$ and $E_4$, by implementing H\"{o}lder's inequality and making use of the fact $|m|=1$. So applying these facts with the estimate \eqref{ES3} and the embeddings $H^1(\Omega) \hookrightarrow L^4(\Omega)$, we deduce
\begin{flalign*}
	E_3 &= - \int_\Omega (\nabla m \times u)\cdot \nabla \Delta m \ dx- \int_\Omega (m \times \nabla u)\cdot\nabla \Delta m \ dx&\\
	& \leq  \ \|\nabla m(t)\|_{L^4(\Omega)} \|u(t)\|_{L^4(\Omega)} \|\nabla \Delta m(t)\|_{L^2(\Omega)} + \|\nabla u(t)\|_{L^2(\Omega)} \|\nabla \Delta m(t)\|_{L^2(\Omega)}\\
	&\leq  \epsilon \int_{\Omega}|\nabla \Delta m(t) |^2dx + C(\epsilon) \ \left(1+\|\Delta m(t)\|^2_{L^2(\Omega)}\right)\ \|u(t)\|^2_{H^1(\Omega)},
\end{flalign*}
and
\begin{flalign*}
	E_4 &= \int_\Omega (\nabla m \times (m \times u))\cdot \nabla \Delta m \ dx+\int_\Omega ( m \times (\nabla m \times u))\cdot \nabla \Delta m \ dx + \int_\Omega (m \times (m \times \nabla u))\cdot \nabla \Delta m \ dx\\
	& \leq  2\ \|\nabla m(t)\|_{L^4(\Omega)} \|u(t)\|_{L^4(\Omega)} \|\nabla \Delta m(t)\|_{L^2(\Omega)} +\  \|\nabla u(t)\|_{L^2(\Omega)} \|\nabla \Delta m(t)\|_{L^2(\Omega)}\\
	&\leq  \epsilon\ \int_{\Omega}|\nabla \Delta m(t) |^2dx + C(\epsilon)\ \left(1+\|\Delta m(t)\|^2_{L^2(\Omega)}\right)\ \|u(t)\|^2_{H^1(\Omega)}.
\end{flalign*}
Now, substituting the estimates for $E_i$ for $i=1$ to $4$ in \eqref{EQ3}, choosing $\epsilon=1/8$ and doing a time integration, we obtain
	\begin{multline*}
		\|\Delta m(t)\|^2_{L^2(\Omega)} +  \int_0^t \| \nabla \Delta m(\tau)\|^2_{L^2(\Omega)}\ d\tau \leq \|\Delta m_0\|^2_{L^2(\Omega)}  
		+C\ \bigg[  \int_0^t \|\nabla m\|^2_{L^2(\Omega)} \ \left( \|\Delta m(\tau)\|^2_{L^2(\Omega)}\right)^2 \ d\tau\\
			+  \int_0^t  \ \left(1 + \|\Delta m(\tau)\|^2_{L^2(\Omega)}\right) \  \|u(\tau)\|^2_{H^1(\Omega)} \ d\tau \bigg],\ \ \forall\ t\in [0,T^*).
	\end{multline*}	
	Applying Gronwall's inequality and using the estimate \eqref{E2}, we derive
	\begin{align}
		\|\Delta m(t)\|^2_{L^2(\Omega)} &+  \int_0^t \| \nabla \Delta m(\tau)\|^2_{L^2(\Omega)}\ d\tau  \leq \|\Delta m_0\|^2_{L^2(\Omega)}\nonumber\\
		& \times  \exp\left(C\ \|\nabla m\|^2_{L^\infty(0,t;L^2(\Omega))} \|\Delta m\|^2_{L^2(0,t;L^2(\Omega))}+ C\ \|u\|^2_{L^2(0,T;H^1(\Omega))}\right),\ \ \forall\ t\in [0,T^*).\label{E3}	
	\end{align}
	Substituting the bounds for $\|\nabla m\|^2_{L^\infty(0,t;L^2(\Omega))}$ and $\|\Delta m\|^2_{L^2(0,t;L^2(\Omega))}$ from estimate \eqref{E2} in estimate \eqref{E3}, we obtain
	\begin{align}\label{E4}
		\|\Delta m(t)\|^2_{L^2(\Omega)} &+  \int_0^t \| \nabla \Delta m(\tau)\|^2_{L^2(\Omega)}\ d\tau  \leq \|\Delta m_0\|^2_{L^2(\Omega)}\\
		& \times  \exp\left(C\ \left[ 1 + \|\nabla m_0\|^4_{L^2(\Omega)} +\ \|\nabla m\|^8_{L^4(0,T^*;L^4(\Omega))}+ \ \|u\|^4_{L^2(0,T;H^1(\Omega))} \right] \right),\ \ \forall\ t\in [0,T^*).	\nonumber
	\end{align}
Therefore, the solution doesn't blow up at $T^*$, which contradicts out assumption that $T^*<T$ is the maximal time of existence. Hence the regular solution of system \eqref{NLP} exists on the entire time domain $[0,T]$.

	Next, we take square of the $H^1(\Omega)$ norm of $m_t$ in equation \eqref{NLP} and then integrate over time to obtain
	\begin{align*}
		\int_0^T \|m_t(t)\|^2_{H^1(\Omega)}\ dt \leq & \ 4\left(\int_0^T \|m\times \Delta m\|^2_{H^1(\Omega)}\ dt +\int_0^T \|m \times u\|^2_{H^1(\Omega)} \ dt\right.\\
		&\left. +\int_0^T \|m\times \big(m\times \Delta m\big)\|^2_{H^1(\Omega)}\ dt + \int_0^T \|m\times \big(m\times u\big)\|^2_{H^1(\Omega)}\ dt\right).
	\end{align*}
Now, estimating the terms on the right hand side using Lemma \ref{PROP2} and using the equality $|m|=1$, we derive
\begin{equation*}
		\int_0^T \|m_t(t)\|^2_{H^1(\Omega)}\ dt \leq  C\ \|m\|^2_{L^\infty(0,T;H^2(\Omega))} \ \|m\|^2_{L^2(0,T;H^3(\Omega))} +C\ \|m\|^2_{L^\infty(0,T;H^2(\Omega))}\ \|u\|^2_{L^2(0,T;H^1(\Omega))}.
\end{equation*}
Substituting the bounds of $\|m\|^2_{L^\infty(0,T;H^2(\Omega))}$ and $\|m\|^2_{L^2(0,T;H^3(\Omega))}$ from estimate \eqref{E4}, we find
	 \begin{equation}\label{E5}
	 	\|m_t\|^2_{L^2(0,T;H^1(\Omega))}  \leq  \exp\left(C\ \left[ 1 + \|\Delta m_0\|^4_{L^2(\Omega)} +\ \|\nabla m\|^8_{L^4(0,T;L^4(\Omega))}+ \ \|u\|^4_{L^2(0,T;H^1(\Omega))} \right] \right).		 	
	 \end{equation}
Combining estimates \eqref{E4} and \eqref{E5}, we find the required result \eqref{SSEE2}. Hence the proof.
\end{proof}

%	\begin{Cor}\label{SSCO}
%		Assume that $m$ is a regular solution of system \eqref{NLP} corresponding to the control $u\in L^2(0,T;H^1(\Omega))$ and initial data $m_0$ satisfying condition \eqref{IC}. Then $m\in W^{1,2}(0,T;H^3(\Omega),H^1(\Omega))$ and also satisfies the following estimate:
%			\begin{align}\label{SSEE2}
%			\|m\|^2_{L^\infty(0,T;H^2(\Omega))} &+ \|m\|^2_{L^2(0,T;H^3(\Omega))}+ \|m_t\|^2_{L^2(0,T;H^1(\Omega))}  \leq \|\Delta m_0\|^2_{L^2(\Omega)}\nonumber\\
%			&\ \ \ \  \times  \exp\left(C\ \left[ 1 + \|\nabla m_0\|^4_{L^2(\Omega)} +\ \|\nabla m\|^8_{L^4(\Omega)}+ \ \|u\|^4_{L^2(0,T;H^1(\Omega))} \right] \right),\ \ \forall\ t\in [0,T].	
%		\end{align}
%		\end{Cor}
%The proof of Corollary \ref{SSCO} is deduced from Theorem \ref{T-SS}, by taking $a=b=f=h=m$.

\section{Existence of Optimum and First Order Optimality Condition}\label{S-FOOC}
\subsection{Existence of Optimal Control}
%The existence of an optimal control in optimal control problems is not always assured and can be a challenging aspect to determine depending upon the convexity of the problem. While convex problems offer specific assumptions that ensure the existence of an optimal control, supported by the well-behaved properties of convex functions and sets, non-convex problems present greater challenges. In non-convex scenarios, additional analysis and specialized techniques are often required to establish the existence of an optimal control. The presence of multiple local optima or infeasible solutions further complicates the determination of an optimal control in non-convex settings. Therefore, the existence of an optimal control is contingent upon factors such as problem formulation, necessitating further investigation and exploration in order to ascertain its presence.

%In this subsection we have proved the existence of an optimum of problem (OCP) under assumption \eqref{AOAC}.

For the optimal control problem to have practical relevance, it is essential to confirm the existence of at least one globally optimal solution. This assurance is provided by the following theorem.

\begin{proof}[Proof of Theorem \ref{T-EOOC}]
	
	Since the functional $\mathcal J(\cdot, \cdot)$ is bounded below and $\mathcal{A}\neq \emptyset$, there exists a minimizing sequence $\big\{(m_n,u_n)\big\} \subset \mathcal{A}$ such that
	\begin{equation*}
		\inf_{(m,u) \in \mathcal{A}} \mathcal J(m,u)=\lim_{n \to \infty} \mathcal J(m_n,u_n)=\alpha.	
	\end{equation*}
	Since $(m_n,u_n)\in \mathcal{A}$, it follows that for each $n$, the pair $(m_n,u_n)$ serves as a regular solution of the following system:	
%	Every admissible pair $(m_n,u_n)$ is a strong solution of the system 
	\begin{equation}
		\begin{cases}
			(m_n)_t= m_n \times (\Delta m_n +u_n)-  m_n \times (m_n \times (\Delta m_n +u_n)) \ \ \ \ \  \text{in}\ \Omega_T, \\
			\frac{\partial m_n}{\partial \eta}=0\ \ \ \ \ \ \  \ \ \ \ \ \ \ \text{in}\ \partial\Omega_T,\label{EOC1}\\
			m_n(\cdot,0)=m_0 \ \ \ \ \ \text{in} \ \Omega.
		\end{cases}	 
	\end{equation}
	As the controls $u_n\in \mathcal{U}_{ad}$ for each $n$, so $\{u_n\}$ is uniformly bounded in $L^2(0,T;H^1(\Omega))$. Then there exists a sub-sequence again denoted as $\{u_n\}$ such that $u_n \rightharpoonup \widetilde{u}$ weakly in $L^2(0,T;H^1(\Omega))$ for some element $\widetilde{u} \in L^2(0,T;H^1(\Omega))$. Now, by using estimate \eqref{SSUB}, we can find a uniform bound for $\{m_n\}$  in $L^2(0,T;H^3(\Omega))$ $\cap\ C([0,T];H^2(\Omega))$ and for $\left\{ (m_n)_t\right\}$ is in $L^2(0,T;H^1(\Omega))$.	
	
	Then, by Aubin–Lions–Simon compactness theorem, $\{m_n\}$ is relatively compact in $C([0,T];H^1(\Omega))$ $\cap L^2(0,T;H^2(\Omega))$. Therefore, there exists sub-sequence (again represented as $\left\{ (m_n,u_n)\right\}$) such that
	\begin{eqnarray*} \left\{\begin{array}{ccccl}
			u_n &\overset{w}{\rightharpoonup} & \widetilde{u} \ &\mbox{weakly in}&  L^2(0,T;H^1(\Omega)),\\
			m_n &\overset{w}{\rightharpoonup} & \widetilde{m} \  &\mbox{weakly in}&  L^2(0,T;H^3(\Omega)),\\
			(m_n)_t &\overset{w}{\rightharpoonup} & \widetilde{m}_t \ &\mbox{weakly in}&  L^2(0,T;H^1(\Omega)),\\
			m_n &\overset{s}{\to} & \widetilde{m} \  &\mbox{strongly in}& C([0,T];H^1(\Omega))\cap L^2(0,T;H^2(\Omega)),	 \ \mbox{as} \  \ n\to \infty. \label{P2}
		\end{array}\right.	
	\end{eqnarray*}
	Taking limit $n \to \infty$ in \eqref{EOC1} and using the above convergences, we obtain that $(\widetilde{m},\widetilde{u})$ is a regular solution of system \eqref{EP}, that is, $\widetilde{u}\in \mathcal{U}$. Indeed, by weak sequential lower semi-continuity $\mathcal{J}(\widetilde{m},\widetilde{u})\leq 
	\liminf_{n \to \infty} \mathcal{J}(m_{u_n},u_n)\leq R/2$, which implies $\widetilde{u}\in \overline{\mathcal{U}_{R/2}}$. Also, as the set $\mathcal{U}_{a,b}$ is closed and convex, therefore it is weakly closed. Consequently, we have $\widetilde{u}\in \mathcal{U}_{ad}$.
%Since, $\{u_n\}$ belongs to this set and $u_n \overset{w}{\rightharpoonup} \widetilde{u}$ weakly in $L^2(0,T;H^1(\Omega))$, so $\widetilde{u} \in \mathcal{U}_{a,b}\cap \mathcal{U}_R$[CORRECT]. 
%Indeed, $\widetilde{u}\in \mathcal{U}_{ad}$.

	Also, since $m_n \to \widetilde{m}$ strongly in $L^4(0,T;H^1(\Omega)),$ $m_n(\cdot,T)\rightharpoonup \widetilde{m}(\cdot,T)$ weakly in $L^2(\Omega)$ and $u_n \rightharpoonup \widetilde{u}$ weakly in $L^2(0,T;H^1(\Omega))$, the functional $\mathcal J(\cdot,\cdot)$ is weakly lower semi-continuous, that is,
	%$$\mathcal J(\widetilde{m},\widetilde{u}) \leq \liminf_{n \to \infty} \mathcal J(m_n,u_n) < +\infty,$$
	%whence $(\widetilde{m},\widetilde{u})$ is an admissible pair. As  $\left\{(m_n,u_n)\right\}$ is a minimizing sequence, we have
	\begin{equation}\label{I2}
		\mathcal J(\widetilde{m},\widetilde{u}) \leq \liminf_{n \to \infty}  \mathcal J(m_n,u_n)=\lim_{n \to \infty} \mathcal J(m_n,u_n)=\alpha.
	\end{equation}
	Since $\alpha$ is the infimum of the functional $\mathcal J$ over $\mathcal{A}$, so that 
	$$\alpha \leq \mathcal J(\widetilde{m},\widetilde{u}).$$
	Hence combining with \eqref{I2}, we get 
	$\displaystyle{\mathcal J(\widetilde{m},\widetilde{u}) = \alpha = \inf_{(m,u) \in \mathcal{A}} \mathcal J(m,u).}$
	This completes the proof.
\end{proof}

\subsection{Control-to-state operator}

Suppose $\widetilde{m}$ be the unique regular solution of the system \eqref{NLP} corresponding to the control $\widetilde{u}\in \mathcal{U}_R$ and initial data $m_0$ satisfying condition \eqref{IC}. Then consider the following linearized system 
\begin{equation}\label{CLE}
	(L-LLG)\begin{cases}
		\begin{array}{l}
			\mathcal{L}_{\widetilde{u}}z=f \ \ \ \ \text{in}\ \Omega_T,\\
			\frac{\partial z}{\partial \eta}=0 \ \ \ \ \ \  \text{in}\ \partial \Omega_T, \ \ \ z(x,0)=z_0\ \ \text{in}\ \Omega,
		\end{array}
	\end{cases}	
\end{equation}
where the operator $\mathcal{L}_{\widetilde{u}}$ is defined as
\begin{equation}\label{CLO}
	\mathcal{L}_{\widetilde{u}}z:= z_t-\Delta z -2 (\nabla \widetilde{m}\cdot \nabla z)\widetilde{m}- |\nabla \widetilde{m}|^2z-z\times \Delta \widetilde{m} - \widetilde{m} \times \Delta z - z \times \widetilde{u} +z \times (\widetilde{m} \times \widetilde{u})+ \widetilde{m} \times (z \times \widetilde{u}).	
\end{equation}

\textbf{Note:} As $\widetilde{m}$ is a regular solution of the problem \eqref{NLP}, in the subsequent analysis we have used the inequality $1= \frac{1}{|\Omega|}|\Omega|=\frac{1}{|\Omega|}\|\widetilde{m}\|^2_{L^2(\Omega)} \leq C\ \|\widetilde{m}\|^2_{H^2(\Omega)}$ without mentioning it repeatedly.\\
\begin{Lem}\label{L-SLS}
	For any $f\in L^2(0,T;H^1(\Omega))$ there exist a unique regular solution $z\in L^2(0,T;H^3(\Omega))\cap L^{\infty}(0,T;H^2(\Omega))$ of the linearized system \eqref{CLE}. Moreover, the following estimate holds:
	\begin{align}\label{LSSE}
		&\|z\|^2_{L^{\infty}(0,T;H^2(\Omega))}+ \|z\|^2_{L^2(0,T;H^3(\Omega))} + \|z_t\|^2_{L^2(0,T;H^1(\Omega))} \leq \left(\|z_0\|^2_{L^2(\Omega)}+\|\Delta z_0\|^2_{L^2(\Omega)} +  \|f\|^2_{L^2(0,T;H^1(\Omega))}\right) \nonumber\\
		&\hspace{1cm}\times \ \exp\bigg\{C  \left(\|\widetilde{m}\|^2_{L^\infty(0,T;H^2(\Omega))}\|\widetilde{m}\|^2_{L^2(0,T;H^3(\Omega))}+\|\widetilde{m}\|^2_{L^\infty(0,T;H^2(\Omega))} \|\widetilde{u}\|^2_{L^2(0,T;H^1(\Omega))}\right) \bigg\}.
	\end{align}

\end{Lem}
\begin{proof}
By appealing to Faedo-Galerkin approximation technique, we can show the existence and uniqueness of system \eqref{CLE}. So, here we will only show a priori estimates for the solution. Considering the $L^2$ inner product of \eqref{CLE} with $z$ and employing cross product properties from Lemma \ref{CPP},  we have
\begin{equation}\label{LSE1}
	\frac{d}{dt} \|z(t)\|^2_{L^2(\Omega)} +  \|\nabla z(t)\|^2_{L^2(\Omega)} \leq C\  \left( \|\widetilde{m}(t)\|^2_{H^3(\Omega)}+\|\widetilde{u}(t)\|^2_{H^1(\Omega)}\right)\|z(t)\|^2_{L^2(\Omega)} + \ \|f(t)\|^2_{L^2(\Omega)}.	
\end{equation}
By taking gradient of system \eqref{CLE} and then considering inner product with $-\nabla \Delta z$, we find
\begin{eqnarray}\label{LSH2E}
	\lefteqn{\frac{1}{2} \frac{d}{dt} \|\Delta z(t)\|^2_{L^2(\Omega)} +  \int_\Omega |\nabla \Delta z(t)|^2 dx = -2 \int_\Omega \nabla \big( \widetilde{m}(\nabla \widetilde{m} \cdot \nabla z)\big) \cdot \nabla \Delta z\ dx-\ \int_\Omega \nabla \big( |\nabla \widetilde{m}|^2z\big) \cdot  \nabla  \Delta z\ dx}\nonumber\\
	&&- \int_\Omega \nabla \big(z \times \Delta \widetilde{m}\big) \cdot \nabla \Delta z\ dx - \int_\Omega \nabla \big( \widetilde{m} \times \Delta z\big)  \cdot \nabla \Delta z\ dx-  \int_\Omega \nabla \big(z \times \widetilde{u}\big)\cdot \nabla \Delta z\ dx\nonumber\\
	&&+\  \int_\Omega \nabla \big(z \times (\widetilde{m} \times \widetilde{u})\big)\cdot  \nabla \Delta z\ dx + \int_\Omega \nabla \big(\widetilde{m} \times (z \times \widetilde{u})\big)\cdot \nabla \Delta z\ dx - \int_\Omega \nabla f\cdot  \nabla \Delta z\ dx:= \sum_{i=1}^{8}\Gamma_i. \ \ \ \ \ \ \ 
\end{eqnarray}

	Let us estimate some of the terms on the right hand side. For the first term $\Gamma_1$, applying H\"older's inequality and the embeddings $H^1(\Omega) \hookrightarrow L^p(\Omega)\ \text{for} \ p \in [1,\infty)$ and $H^2(\Omega) \hookrightarrow L^\infty(\Omega)$, we get 
\begin{flalign*}
	\Gamma_1 &=-2 \int_\Omega \left[\nabla \widetilde{m}(\nabla \widetilde{m} \cdot \nabla z)+ \widetilde{m}(D^2\widetilde{m}\cdot\nabla z)+\widetilde{m}(\nabla \widetilde{m}\cdot D^2z)\right]\cdot \nabla \Delta z\ dx&\\
	&\leq \ 2\ \|\nabla \widetilde{m}(t)\|_{L^4(\Omega)} \|\nabla \widetilde{m}(t)\|_{L^8(\Omega)}  \|\nabla z(t)\|_{L^8(\Omega)} \|\nabla \Delta z(t)\|_{L^2(\Omega)}\\
	&+\ 2\ \|\widetilde{m}(t)\|_{L^\infty(\Omega)} \left(\|D^2\widetilde{m}(t)\|_{L^4(\Omega)} \|\nabla z(t)\|_{L^4(\Omega)}+\|\nabla \widetilde{m}(t)\|_{L^\infty(\Omega)} \|D^2z(t)\|_{L^2(\Omega)} \right) \|\nabla \Delta z(t)\|_{L^2(\Omega)} \\
%	&\leq C\ \|\nabla \widetilde{m}(t)\|_{H^1(\Omega)} \|\nabla \widetilde{m}(t)\|_{H^1(\Omega)}  \|\nabla z(t)\|_{H^1(\Omega)} \|\nabla \Delta z(t)\|_{L^2(\Omega)}\\
%	&+\ C\ \|\widetilde{m}(t)\|_{H^2(\Omega)}\left( \|D^2\widetilde{m}(t)\|_{H^1(\Omega)} \|\nabla z(t)\|_{H^1(\Omega)}+\|\nabla \widetilde{m}(t)\|_{H^2(\Omega)} \|D^2z(t)\|_{L^2(\Omega)}\right) \|\nabla \Delta z(t)\|_{L^2(\Omega)} \\
	&\leq \epsilon   \int_\Omega |\nabla \Delta z(t)|^2 dx +\ C(\epsilon)   \ \left(\|\widetilde{m}(t)\|^4_{H^2(\Omega)}+ \|\widetilde{m}(t)\|^2_{H^2(\Omega)} \|\widetilde{m}(t)\|^2_{H^3(\Omega)}\right) \|z(t)\|^2_{H^2(\Omega)}.	
\end{flalign*}
A similar estimate for $\Gamma_2$, $\Gamma_3$ and $\Gamma_4$ can be obtained. Now, let us estimate one of the terms containing the control 
\begin{flalign*}
	\Gamma_6 &= \int_\Omega \Big[ \nabla z \times (\widetilde{m} \times \widetilde{u})+z \times (\nabla \widetilde{m} \times \widetilde{u})+z \times (\widetilde{m} \times \nabla \widetilde{u})\Big]\cdot  \nabla \Delta z\ dx&\\
	&\leq  \|\nabla z(t)\|_{L^4(\Omega)} \|\widetilde{m}(t)\|_{L^8(\Omega)}  \|\widetilde{u}(t)\|_{L^8(\Omega)} \|\nabla \Delta z(t)\|_{L^2(\Omega)}\\
	&+\  \|z(t)\|_{L^\infty(\Omega)} \left(\|\nabla \widetilde{m}(t)\|_{L^4(\Omega)} \|\widetilde{u}(t)\|_{L^4(\Omega)}+\| \widetilde{m}(t)\|_{L^\infty(\Omega)} \|\nabla \widetilde{u}(t)\|_{L^2(\Omega)} \right) \|\nabla \Delta z(t)\|_{L^2(\Omega)} \\
	%	&\leq C\  \|\nabla z(t)\|_{H^1(\Omega)} \|\widetilde{m}(t)\|_{H^1(\Omega)}  \|\widetilde{u}(t)\|_{H^1(\Omega)} \|\nabla \Delta z(t)\|_{L^2(\Omega)}\\
	%	&+\ C\ \|z(t)\|_{H^2(\Omega)} \left( \|\nabla \widetilde{m}(t)\|_{H^1(\Omega)} \|\widetilde{u}(t)\|_{H^1(\Omega)}+\| \widetilde{m}(t)\|_{H^2(\Omega)} \|\nabla \widetilde{u}(t)\|_{L^2(\Omega)} \right) \|\nabla \Delta z(t)\|_{L^2(\Omega)} \\
	&\leq \epsilon \int_\Omega |\nabla \Delta z(t)|^2 \ dx  + C(\epsilon)   \ \|\widetilde{m}(t)\|^2_{H^2(\Omega)} \|\widetilde{u}(t)\|^2_{H^1(\Omega)} \|z(t)\|^2_{H^2(\Omega)}.
\end{flalign*}
By substituting all these estimates in equation \eqref{LSH2E} and choosing a suitable value for $\epsilon$ and adding with \eqref{LSE1}, we get
%	\begin{eqnarray}\label{LSH2}
	%		\lefteqn{\frac{1}{2} \frac{d}{dt} \left(\|\Delta v_n(t)\|^2_{L^2(\Omega)}\right)+ \frac{1}{2} \int_\Omega |\nabla \Delta v_n(t)|^2 dx \leq C \ \|m(t)\|^4_{H^2(\Omega)} \|v_n(t)\|^2_{H^2(\Omega)}}\nonumber\\
	%		&&+\ C\ \|m(t)\|^2_{H^3(\Omega)} \|v_n(t)\|^2_{H^2(\Omega)}	+ C \ \|m(t)\|^2_{H^2(\Omega)} \|u(t)\|^2_{H^1(\Omega)} \|v_n(t)\|^2_{H^2(\Omega)}+C\ \|g(t)\|^2_{H^1(\Omega)}.
	%	\end{eqnarray}
%	Combining \eqref{LSL2} and \eqref{LSH2}, we have
\begin{eqnarray}\label{LSE1}
	\lefteqn{\frac{d}{dt} \left(\|z(t)\|^2_{L^2(\Omega)}+\|\Delta z(t)\|^2_{L^2(\Omega)}\right)+  \|\nabla z(t)\|^2_{L^2(\Omega)}+\|\nabla \Delta z(t)\|^2_{L^2(\Omega)} } \nonumber\\
	&&\leq C\ \bigg[\|\widetilde{m}(t)\|^2_{H^2(\Omega)} \|\widetilde{m}(t)\|^2_{H^3(\Omega)}  + \|\widetilde{m}(t)\|^2_{H^2(\Omega)} \|\widetilde{u}(t)\|^2_{H^1(\Omega)}\bigg]\ \|z(t)\|^2_{H^2(\Omega)} + C\   \|f(t)\|^2_{H^1(\Omega)}.
\end{eqnarray}
By invoking the inequality $\|z\|_{H^2(\Omega)}\leq C\ \left(\|z\|_{L^2(\Omega)}+\|\Delta z\|_{L^2(\Omega)}\right)$ from Lemma \ref{EN} and applying Gronwall's inequality, we obtain 
\begin{equation}\label{LSE2}
	\sup_{t \in [0,T]} \left(\|z(t)\|^2_{L^2(\Omega)}+\|\Delta z(t)\|^2_{L^2(\Omega)}\right)
	\leq C_1(\Omega,T,z_0,f,\widetilde{m},\widetilde{u}),	
\end{equation}
\begin{flalign*}
\text{where}\ & C_1(\Omega,T,z_0,f,\widetilde{m},\widetilde{u}) = \left(\|z_0\|^2_{L^2(\Omega)}+\|\Delta z_0\|^2_{L^2(\Omega)} +  \|f\|^2_{L^2(0,T;H^1(\Omega))}\right)&\\
&\hspace{2cm}\times \ \exp\bigg\{C  \left(\|\widetilde{m}\|^2_{L^\infty(0,T;H^2(\Omega))}\|\widetilde{m}\|^2_{L^2(0,T;H^3(\Omega))}+\|\widetilde{m}\|^2_{L^\infty(0,T;H^2(\Omega))} \|\widetilde{u}\|^2_{L^2(0,T;H^1(\Omega))}\right) \bigg\}.	
\end{flalign*}

%	\begin{eqnarray}
	%		\lefteqn{\int_0^t \left(\|\nabla v_n(\tau)\|^2_{L^2(\Omega)}+\|\nabla \Delta v_n(\tau)\|^2_{L^2(\Omega)}\right)\ d\tau \leq \|v(0)\|^2_{L^2(\Omega)}+\|\Delta v(0)\|^2_{L^2(\Omega)}+ C\ \|g\|^2_{L^2(0,T;H^1(\Omega))} }\nonumber\\
	%		&&+ C \left[ 1+ \|m\|^2_{L^2(0,T;H^3(\Omega))} + \|u\|^2_{L^2(0,T;H^1(\Omega))} +T\  \|m\|^4_{L^\infty(0,T;H^2(\Omega))} + \|m\|^2_{L^2(0,T;H^2(\Omega))} \|u\|^2_{L^2(0,T;H^1(\Omega))}\right]\nonumber\\
	%		&&\hspace{3in}\times\ \|v_n\|^2_{L^\infty(0,T;H^2(\Omega))}.\nonumber
	%	\end{eqnarray}
By integrating \eqref{LSE1} over $(0,t)$ and using the bounds for $z$ in $L^\infty(0,T;H^2(\Omega))$ from inequality \eqref{LSE2}, we have
\begin{equation}\label{LSE3}	
	\int_0^T \left(\|\nabla z(\tau)\|^2_{L^2(\Omega)}+\|\nabla \Delta z(\tau)\|^2_{L^2(\Omega)}\right)\ d\tau \leq 	C_1(\Omega,T,z_0,f,\widetilde{m},\widetilde{u}). 
\end{equation}
Now, taking $L^2(0,T;H^1(\Omega))$ of $z_t$ in equation \eqref{CLE} and estimating the terms using Lemma \ref{PROP2}, we derive
\begin{align}\label{LSE4}
	\|z_t\|_{L^2(0,T;H^1(\Omega))} &\leq \|\widetilde{m}\|^2_{L^\infty(0,T;H^2(\Omega))}\ \|\widetilde{m}\|^2_{L^2(0,T;H^3(\Omega))}\ \left(\|z\|^2_{L^\infty(0,T;H^2(\Omega))}+\|z\|^2_{L^2(0,T;H^3(\Omega))}\right)\nonumber\\
	&\hspace{1cm} +\|z\|^2_{L^\infty(0,T;H^2(\Omega))}\ \|\widetilde{m}\|^2_{L^\infty(0,T;H^2(\Omega))}\ \|\widetilde{u}\|^2_{L^2(0,T;H^1(\Omega))}.
\end{align}
Substituting \eqref{LSE2} and \eqref{LSE3} in estimate \eqref{LSE4}, we can obtain the required result \eqref{LSSE}.
\end{proof}

\noindent By virtue of the definition of the control set $\mathcal{U}_R$, we are able to define an operator $G:\mathcal{U}_R \to \mathcal{M}$ called the control-to-state operator such that for every $u\in \mathcal{U}_R$ there exists a unique regular solution $m_u:=G(u)$ of system \eqref{NLP}. Now, we will prove the Lipschitz continuity and Fr\'echet differentiability of this operator.

\begin{Lem}{(Lipschitz Continuity of G)}\label{L-LCCTS}
	The control-to-state operator $G:\mathcal{U}_R \to \mathcal{M}$ is Lipschitz continuous, that is, there exists a constant $C_2>0$ depending on the parameters $\Omega,T,R,m_0$ such that
	\begin{equation}\label{LCCTSO}
		\|G(u_1)-G(u_2)\|_{\mathcal{M}} \leq C_2\ \|u_1-u_2\|_{L^2(0,T;H^1(\Omega))}, \ \ \ \ \ \ \forall \ u_1,u_2 \in \mathcal{U}_R.
	\end{equation}	
\end{Lem}

\begin{proof}
Suppose $m_1$ and $m_2$ are two regular solution of system \eqref{NLP} corresponding to the controls $u_1$ and $u_2$ respectively. Let us define $\shat{m} :=m_1-m_2$ and $\shat{u}:=u_1-u_2$. Then $(\shat{m},\shat{u})$ will satisfy the following system 
\begin{equation}\label{EPD}
	\begin{cases}
		\shat{m}_t- \Delta \shat{m} = |\nabla m_1|^2\shat{m} + \nabla \shat{m}\cdot (\nabla m_1+\nabla m_2) m_2 +\shat{m} \times \Delta m_1 +m_2 \times \Delta \shat{m} + m_1 \times \shat{u}  \\
		\hspace{0.7in} +\ \shat{m}\times u_2 - \shat{m} \times (m_1 \times u_1)- m_2 \times (\shat{m} \times u_1)-m_2\times (m_2 \times \shat{u})\ \ \ \ \ (x,t)\in \Omega_T,\\	
		\frac{\partial \shat{m}}{\partial \eta}=0 \ \ \ \ \ \ \ \ (x,t) \in  \partial\Omega_T,\\
		\shat{m}(\cdot,0)=0 \ \ \ \text{in} \ \Omega.
	\end{cases}	
\end{equation}
%\begin{equation}\label{EPD}
%	\begin{cases}
%		\displaystyle \shat{m}_t- \Delta \shat{m}= \sum_{k=1}^3 \mathcal{D}_k  \ \ \ (x,t)\in \Omega_T,\\
%		\frac{\partial \shat{m}}{\partial \eta}=0, \ \ \ (x,t) \in  \partial\Omega_T,\ \ \ \ \ \ \ 
%		\shat{m}(\cdot,0)=0 \ \ \text{in} \ \Omega,
%	\end{cases}	
%\end{equation}
%where
%\begin{flalign*}
%	\mathcal{D}_1 &= |\nabla m_1|^2\shat{m} + \nabla \shat{m}\cdot (\nabla m_1+\nabla m_2) m_2, &\\ 
%	\mathcal{D}_2 &=  \shat{m} \times \Delta m_1 +m_2 \times \Delta \shat{m} + m_1 \times \shat{u} +\shat{m}\times u_2,\\
%	\mathcal{D}_3 &=- \shat{m} \times (m_1 \times u_1)- m_2 \times (\shat{m} \times u_1)-m_2\times (m_2 \times \shat{u}).
%\end{flalign*}

%By proceeding in a similar way to Lemma \ref{L-SLS} and using the equality $\shat{m}(0)=0$, we can find the following estimate for $\shat{m}$,
%\begin{equation}\label{LC-E1}
%	\|\shat{m}\|^2_{L^{\infty}(0,T;H^2(\Omega))}+ \|\shat{m}\|^2_{L^2(0,T;H^3(\Omega))} \leq \sum_{k=1}^3 \|\mathcal{D}_k\|^2_{L^2(0,T;H^1(\Omega))}.
%\end{equation}
By estimating the $L^2(\Omega)$ inner product of equation \eqref{EPD} with $\shat{m}$, we obtain
\begin{align}\label{INEQ-1}
	&\frac{1}{2} \frac{d}{dt}\|\shat{m}\|^2_{L^2(\Omega)} + \|\nabla \shat{m}\|^2_{L^2(\Omega)} \leq  \|\shat{u}(t)\|^2_{L^2(\Omega)}\nonumber\\
	&\hspace{1cm}+ C\ \left(1+\|m_1(t)\|^2_{H^2(\Omega)}+\|m_2(t)\|^2_{H^2(\Omega)}+ \|u_1(t)\|^2_{L^2(\Omega)}\right)\ \|\shat{m}(t)\|^2_{H^1(\Omega)} .
\end{align}

Next, we apply the gradient operator to equation \eqref{EPD}, followed by taking the inner product with $\nabla \Delta \hat{m}$. This leads to the following expression
\begin{align*}
		\frac{1}{2} \frac{d}{dt}\|\Delta \shat{m}\|^2_{L^2(\Omega)} &+ \|\nabla \Delta \shat{m}\|^2_{L^2(\Omega)} = \int_{\Omega} \nabla \left( |\nabla m_1|^2\shat{m} \right)\cdot\nabla \Delta \shat{m}\ dx  + \int_{\Omega} \nabla \left(\nabla \shat{m}\cdot (\nabla m_1+\nabla m_2) m_2\right)\cdot \nabla \Delta \shat{m}\ dx\\
		&+\int_{\Omega} \nabla \left(\shat{m} \times \Delta m_1\right)\cdot \nabla \Delta \shat{m}\ dx + \int_{\Omega} \nabla \left(m_2 \times \Delta \shat{m}\right)\cdot \nabla \Delta \shat{m}\ dx +\int_{\Omega}\nabla \left( m_1 \times \shat{u}\right)\cdot \nabla \Delta \shat{m}\ dx\\
		&-\int_{\Omega}\nabla \left( \shat{m} \times (m_1 \times u_1)\right)\cdot \nabla \Delta \shat{m}\ dx-\int_{\Omega}\nabla \left( m_2 \times (\shat{m} \times u_1)\right) \cdot \nabla \Delta \shat{m}\ dx\\
		&+\int_{\Omega} \nabla \left( \shat{m}\times u_2 \right)\cdot \nabla \Delta \shat{m}\ dx -\int_{\Omega} \nabla \left( m_2\times (m_2 \times \shat{u})\right) \cdot \nabla \Delta \shat{m}\ dx.
\end{align*}
We can see that the terms in the above equation are similar to the ones in equation \eqref{LSH2E}. Therefore, estimating in the same way as we have done for the terms $\Gamma_i$ in equation \eqref{LSH2E} and combining with \eqref{INEQ-1}, we derive
\begin{align*}
	&\frac{d}{dt} \left(\| \shat{m}\|^2_{L^2(\Omega)} +\|\Delta \shat{m}\|^2_{L^2(\Omega)}\right)+\|\nabla \shat{m}\|^2_{L^2(\Omega)}+ \|\nabla \Delta \shat{m}\|^2_{L^2(\Omega)} \leq C\ \left(\|m_1\|^2_{H^2(\Omega)}+\|m_2\|^2_{H^2(\Omega)}\right)\ \|\shat{u}\|^2_{H^1(\Omega)}\\
	&\hspace{0.9cm}+ C\ \left(\|m_1\|^2_{H^2(\Omega)}+\|m_2\|^2_{H^2(\Omega)}\right)\ \left(\|u_1\|^2_{H^1(\Omega)}+\|u_2\|^2_{H^1(\Omega)}+\|m_1\|^2_{H^3(\Omega)}+\|m_2\|^2_{H^3(\Omega)}\right)\ \|
	\shat{m}\|^2_{H^2(\Omega)}.
\end{align*}

Applying Gr\"onwall's inequality and using the bounds $\|u_k\|_{L^2(0,T;H^1(\Omega))}\leq C(\Omega,T,R)$ and $\|m_k\|_{\mathcal{M}}\leq C(\Omega,T,R)$, for $k=1,2$ from estimate \eqref{SSUB}, we arrive at
$$	\|\shat{m}\|^2_{L^{\infty}(0,T;H^2(\Omega))}+ \|\shat{m}\|^2_{L^2(0,T;H^3(\Omega))} \leq C(\Omega,T,R) \ \|\shat{u}\|^2_{L^2(0,T;H^1(\Omega))}.$$
Next, taking the $L^2(0,T;H^1(\Omega))$ norm of $\shat{m}_t$ in \eqref{EPD} and estimating the terms using Lemma \ref{PROP2}, we reach at our required result \eqref{LCCTSO}. 
\end{proof}

\begin{Pro}\label{P-CTS}
	For any control $\bar{u}\in \mathcal{U}_R$, let $m_{\bar{u}}$ be the unique regular solution of system \eqref{NLP}. Then the following holds:
	\begin{enumerate}[label=(\roman*)]
		\item The control-to-state mapping $G$ is Fr\'echet differentiable on $\mathcal{U}_R$, that is for any $\bar{u}\in \mathcal{U}_R$, there exists a bounded linear operator $G'(\bar{u}):L^2(0,T;H^1(\Omega)) \to \mathcal{M}$ with $z:=G'(\bar{u})[u]$ such that
		$$\frac{\|G(\bar{u}+u)-G(\bar{u})-G'(\bar{u})[u]\|_{\mathcal{M}}}{\|u\|_{L^2(0,T;H^1(\Omega))}}\to 0 \ \ \ \ \ \text{as}  \ \|u\|_{L^2(0,T;H^1(\Omega))}\to 0,$$
		where $z$ is the unique regular solution of the following linearized system:
		\begin{equation}\label{LS2}
			\ \ \ \ \ \ \ \ \ \ \ \begin{cases}
				\mathcal{L}_{\bar{u}}z= m_{\bar{u}} \times u -  m_{\bar{u}} \times (m_{\bar{u}} \times u) \ \  \ \text{in} \ \ \Omega_T,\\
				\frac{\partial z}{\partial \eta}=0 \ \ \ \ \text{on} \ \partial \Omega_T,\\
				z(0)=0 \ \ \text{in} \ \Omega.	
			\end{cases}
		\end{equation}
		\item The Fr\'echet derivative $G'$ is Lipschitz continuous, that is, for any controls $u_1,u_2 \in \mathcal{U}_R$ and $u\in L^2(0,T;H^1(\Omega))$, there exists a constant $C_3>0$ depending on $\Omega,T,R,m_0$ such that 
		\begin{equation}\label{LC-CTS}
		\|G'(u_1)[u]-G'(u_2)[u] \|_{\mathcal{M}} \leq C_3\  \|u_1-u_2\|_{L^2(0,T;H^1(\Omega))} \|u\|_{L^2(0,T;H^1(\Omega))}. 	
		\end{equation}
\end{enumerate}
\end{Pro}
\begin{proof}
	Suppose $\bar{u}\in \mathcal{U}_R$ be any fixed control and $m_{\bar{u}}$ be the corresponding regular solution. As for any control $u\in L^2(0,T;H^1(\Omega))$, the terms $m_{\bar{u}}\times u,\ m_{\bar{u}}\times (m_{\bar{u}}\times u) \in L^2(0,T;H^1(\Omega))$, therefore as a consequence of Lemma \ref{L-SLS}, system \eqref{LS2} has a unique regular solution $z \in \mathcal{M}$. Note that at this stage we don't know whether $z=G'(\bar{u})[u]$ or even the control-to-state operator is differentiable or not. Since, $\mathcal{U}_R$ is an open set, so for small enough $u$, $\bar{u}+u$ will also belongs to $\mathcal{U}_R$. Let $m_{\bar{u}+u}$ be the admissible solution corresponding to the control $\bar{u}+u$, then $w:=m_{\bar{u}+u}-m_{\bar{u}}-z$ solves the following system
	\begin{equation}\label{LOCTS}
	\begin{cases}
		\displaystyle \mathcal{L}_{\bar{u}}w= \sum_{k=1}^{8} \Upsilon_k \ \ \ \ \ \text{in}\ \Omega_T,\\
		\frac{\partial w}{\partial \eta}=0 \ \ \ \ \text{on} \ \partial \Omega_T,  \ \ \ \ w(0)=0 \ \ \text{in} \ \Omega,	
	\end{cases}
\end{equation}
where $\mathcal{L}_{\bar{u}}$ is defined in \eqref{CLO}, $\shat{m}=m_{\bar{u}+u}-m_{\bar{u}}$ and the terms $\Upsilon_k$'s are given by
$$
\begin{array}{llll}
\Upsilon_1= \big(\nabla \shat{m}\cdot (\nabla m_{\bar{u}+u}+\nabla m_{\bar{u}})\big)\shat{m},& \Upsilon_2= |\nabla \shat{m}|^2 m_{\bar{u}}, & \Upsilon_3= \shat{m}\times \Delta \shat{m}, &
\Upsilon_4= \shat{m}\times u,\\
\Upsilon_5= -\shat{m}\times (\shat{m}\times u), &  \Upsilon_6= -\shat{m}\times (\shat{m}\times \bar{u}),& \Upsilon_7=-\shat{m}\times (m_{\bar{u}}\times u), &
\Upsilon_8=-m_{\bar{u}}\times (\shat{m}\times u).
\end{array}
$$
Invoking Lemma \ref{L-SLS}, we can directly find the following estimate:		
\begin{align}\label{DCTS-E1}
& \|w\|^2_{L^2(0,T;H^3(\Omega))}+ \|w_t\|^2_{L^2(0,T;H^1(\Omega))} \leq \left(\|w(0)\|^2_{L^2(\Omega)}+\|\Delta w(0)\|^2_{L^2(\Omega)} + \sum_{k=1}^8   \|\Upsilon_k\|^2_{L^2(0,T;H^1(\Omega))}\right) \nonumber\\
&\hspace{0.5cm}\times \ \exp\bigg\{C  \left(\|m_{\bar{u}}\|^2_{L^\infty(0,T;H^2(\Omega))} \|m_{\bar{u}}\|^2_{L^2(0,T;H^3(\Omega))}+\|m_{\bar{u}}\|^2_{L^\infty(0,T;H^2(\Omega))} \|\bar{u}\|^2_{L^2(0,T;H^1(\Omega))}\right) \bigg\}.	
\end{align}
%	\begin{eqnarray}\label{DCTS-E1}
%	\lefteqn{\|w\|^2_{L^2(0,T;H^3(\Omega))}+ \|w_t\|^2_{L^2(0,T;H^1(\Omega))} \leq \left(\|w(0)\|^2_{L^2(\Omega)}+\|\Delta w(0)\|^2_{L^2(\Omega)} + \sum_{k=1}^8   \|\Upsilon_k\|^2_{L^2(0,T;H^1(\Omega))}\right) }\nonumber\\
%	&&\hspace{0.5cm}\times \ \exp\bigg\{C  \left(\|m_{\bar{u}}\|^2_{L^\infty(0,T;H^2(\Omega))} \|m_{\bar{u}}\|^2_{L^2(0,T;H^3(\Omega))}+\|m_{\bar{u}}\|^2_{L^\infty(0,T;H^2(\Omega))} \|\bar{u}\|^2_{L^2(0,T;H^1(\Omega))}\right) \bigg\}.\ \ \ \ 
%\end{eqnarray}
	Let us estimate the terms containing $\Upsilon_k$'s on the right hand side. For the first term $\|\Upsilon_1\|^2_{L^2(0,T;H^1(\Omega))}$, applying inequality \eqref{EE-4} with $p=\shat{m}$, $q= m_{\bar{u}+u}+m_{\bar{u}}$ and $r=\shat{m}$, we find
	$$\|\Upsilon_1\|^2_{L^2(0,T;H^1(\Omega))} \leq C\ \left(\|m_{\bar{u}+u}\|^2_{L^\infty(0,T;H^2(\Omega))}+\|m_{\bar{u}}\|^2_{L^\infty(0,T;H^2(\Omega))}\right) \|\shat{m}\|^2_{L^\infty(0,T;H^2(\Omega))} \|\shat{m}\|^2_{L^2(0,T;H^3(\Omega))}.
	$$
	Similarly, applying estimates \eqref{EE-4}, \eqref{EE-2} and \eqref{EE-3} for $\Upsilon_2,\Upsilon_3$ and $\Upsilon_4$ respectively, we get
	\begin{flalign*}
		\|\Upsilon_2\|^2_{L^2(0,T;H^1(\Omega))} &\leq C\ \|m_{\bar{u}}\|^2_{L^\infty(0,T;H^2(\Omega))} \|\shat{m}\|^2_{L^\infty(0,T;H^2(\Omega))} \|\shat{m}\|^2_{L^2(0,T;H^3(\Omega))},&\\
		\|\Upsilon_3\|^2_{L^2(0,T;H^1(\Omega))} &\leq C\ \|\shat{m}\|^2_{L^\infty(0,T;H^2(\Omega))} \|\shat{m}\|^2_{L^2(0,T;H^3(\Omega))},\\
		\|\Upsilon_4\|^2_{L^2(0,T;H^1(\Omega))}&\leq C\ \|\shat{m}\|^2_{L^\infty(0,T;H^2(\Omega))} \|u\|^2_{L^2(0,T;H^1(\Omega))}.
	\end{flalign*}
 Now, applying estimate \eqref{EE-5} for the last four terms and combining it, we have
 \begin{align*}
 	\sum_{k=5}^8 \|\Upsilon_k\|^2_{L^2(0,T;H^1(\Omega))} &\leq C\ \|\shat{m}\|^4_{L^\infty(0,T;H^2(\Omega))} \left(\|u\|^2_{L^2(0,T;H^1(\Omega))}+\|\bar{u}\|^2_{L^2(0,T;H^1(\Omega))}\right)\\
 	& \ \ \ \ \ \ + C\ \|m_{\bar{u}}\|^2_{L^\infty(0,T;H^2(\Omega))} \|\shat{m}\|^2_{L^\infty(0,T;H^2(\Omega))}\|u\|^2_{L^2(0,T;H^1(\Omega))}.
 \end{align*}
Substituting all these estimates in \eqref{DCTS-E1} and using $w(0)=0$, we find
	\begin{eqnarray}\label{DCTS-E2}
	\lefteqn{ \|w\|^2_{L^2(0,T;H^3(\Omega))} + \|w_t\|^2_{L^2(0,T;H^1(\Omega))} \leq \left[ \|m_{\bar{u}}\|^2_{L^\infty(0,T;H^2(\Omega))} \|\shat{m}\|^2_{L^\infty(0,T;H^2(\Omega))}\|u\|^2_{L^2(0,T;H^1(\Omega))}  \right.  }\nonumber\\
	&&\ \ \ \ \ \ +\ \|\shat{m}\|^4_{L^\infty(0,T;H^2(\Omega))} \left(\|u\|^2_{L^2(0,T;H^1(\Omega))}+\|\bar{u}\|^2_{L^2(0,T;H^1(\Omega))}\right)  \nonumber\\
	&&\ \ \ \ \ \left. +  \left(\|m_{\bar{u}+u}\|^2_{L^\infty(0,T;H^2(\Omega))}+ \|m_{\bar{u}}\|^2_{L^\infty(0,T;H^2(\Omega))}\right)\|\shat{m}\|^2_{L^\infty(0,T;H^2(\Omega))}\|\shat{m}\|^2_{L^2(0,T;H^3(\Omega))}\right] \nonumber\\
	&&\hspace{0.5cm}\times \ \exp\bigg\{C  \left(\|m_{\bar{u}}\|^2_{L^\infty(0,T;H^2(\Omega))} \|m_{\bar{u}}\|^2_{L^2(0,T;H^3(\Omega))}+\|m_{\bar{u}}\|^2_{L^\infty(0,T;H^2(\Omega))} \|\bar{u}\|^2_{L^2(0,T;H^1(\Omega))}\right) \bigg\}.
\end{eqnarray}
Now, as $m_{\bar{u}}$, $m_{\bar{u}+u}$ are regular solutions of system \eqref{NLP}, so using the stability estimate \eqref{SSEE2}, we can find an uniform bound for $\|m_{\bar{u}}\|_{\mathcal{M}} $ and $\|m_{\bar{u}+u}\|_{\mathcal{M}}$. Then using the Lipschitz continuity of the control-to-state operator from Lemma \ref{L-LCCTS} in estimate \eqref{DCTS-E2}, we obtain
\begin{equation}\label{X1}
 \|w\|^2_{L^2(0,T;H^3(\Omega))}+\|w_t\|^2_{L^2(0,T;H^1(\Omega))} \leq C\ \|u\|^4_{L^2(0,T;H^1(\Omega))}+C\ \|u\|^6_{L^2(0,T;H^1(\Omega))}.	
\end{equation}
Therefore, $\frac{\|w\|_{\mathcal{M}}}{\|u\|_{L^2(0,T;H^1(\Omega))}}\to 0 $ as $\|u\|_{L^2(0,T;H^1(\Omega))} \to 0$, which gives the Fr\`echet differentiability of the control-to-state operator and establishes the equality $z=G'(\bar{u})[u]$. Thus, the proof of (i) is finished. 

For any controls $u_1,u_2 \in \mathcal{U}_R$ and $u\in L^2(0,T;H^1(\Omega))$, let $z_{u_1}:=G'(u_1)[u]$ and $z_{u_2}:=G'(u_2)[u]$ be two unique regular solutions of the linearized system \eqref{LS2}. If $\shat{z}=z_{u_1}-z_{u_2}$, $\shat{m}=m_{u_1}-m_{u_2}$ and $\shat{u}=u_1-u_2$, then $(\shat{z},\shat{u})$ satisfies the following system:
	\begin{equation}\label{DCTS-E3}
	\begin{cases}
		\displaystyle \mathcal{L}_{u_1}\shat{z}= \sum_{k=1}^{7} \Theta_k\ \ \ \ \ \text{in}\ \Omega_T,\\
		\frac{\partial \shat{z}}{\partial \eta}=0 \ \ \ \ \text{on} \ \partial \Omega_T,  \ \ \ \ \shat{z}(0)=0 \ \ \text{in} \ \Omega,	
	\end{cases}
\end{equation}
where $\mathcal{L}_{u_1}$ is defined in \eqref{CLO} and the terms $\Theta_k$'s are given by
$$
\begin{array}{llll}
	\Theta_1= 2\ \big(\nabla \shat{m}\cdot \nabla z_{u_2} \big) m_{u_1}+ 2\ \left(\nabla m_{u_2}\cdot \nabla z_{u_2}\right)\shat{m} ,& \Theta_2= \left(\nabla \shat{m} \cdot \nabla (m_{u_1}+m_{u_2})\right) z_{u_2},\\
	\Theta_3= z_{u_2} \times \Delta \shat{m} + \shat{m}\times \Delta z_{u_2}, &	\Theta_4= z_{u_2}\times \shat{u}+\shat{m}\times u,\\
	\Theta_5= -z_{u_2}\times (\shat{m}\times u_1) -z_{u_2}\times (m_{u_2} \times \shat{u}), &  \Theta_6= -\shat{m}\times (z_{u_2}\times u_1)-m_{u_2}\times (z_{u_2}\times \shat{u}),\\
	\Theta_7=-\shat{m}\times (m_{u_1} \times u) -m_{u_2}\times (\shat{m}\times u).
\end{array}
$$
Again invoking Lemma \ref{L-SLS} and using the equality $\shat{z}(0)=0$ for system \eqref{DCTS-E3}, we obtain		
\begin{eqnarray}\label{DCTS-E4}
	\lefteqn{\|\shat{z}\|^2_{L^{\infty}(0,T;H^2(\Omega))}+ \|\shat{z}\|^2_{L^2(0,T;H^3(\Omega))}+ \|\shat{z}_t\|^2_{L^2(0,T;H^1(\Omega))} \leq \sum_{k=1}^8   \|\Theta_k\|^2_{L^2(0,T;H^1(\Omega))} }\nonumber\\
	&&\hspace{0.1cm}\times \ \exp\bigg\{C  \left(\|m_{u_1}\|^2_{L^\infty(0,T;H^2(\Omega))} \|m_{u_1}\|^2_{L^2(0,T;H^3(\Omega))}+\|m_{u_1}\|^2_{L^\infty(0,T;H^2(\Omega))} \|u_1\|^2_{L^2(0,T;H^1(\Omega))}\right) \bigg\}.\ \ \ \ 
\end{eqnarray}
Now, let us estimate the terms containing $\Theta_k$'s on the right hand side. Applying estimate \eqref{EE-4} for the terms $\Theta_1$ and $\Theta_2$, we find
\begin{align*}
	&\|\Theta_1\|^2_{L^2(0,T;H^1(\Omega))} + \|\Theta_2\|^2_{L^2(0,T;H^1(\Omega))}   \\
	&\ \ \ \ \ \leq C \left(\|m_{u_1}\|^2_{L^\infty(0,T;H^2(\Omega))} + \|m_{u_2}\|^2_{L^\infty(0,T;H^2(\Omega))}\right)\|\shat{m}\|^2_{L^\infty(0,T;H^2(\Omega))} \|z_{u_2}\|^2_{L^2(0,T;H^3(\Omega))}\\
	&\ \ \ \ \ \ \ \ \  + C\  \left(\|m_{u_1}\|^2_{L^\infty(0,T;H^2(\Omega))} + \|m_{u_2}\|^2_{L^\infty(0,T;H^2(\Omega))}\right)\|\shat{m}\|^2_{L^2(0,T;H^3(\Omega))} \|z_{u_2}\|^2_{L^\infty(0,T;H^2(\Omega))}.
\end{align*}
Similarly, applying estimates \eqref{EE-2} and \eqref{EE-3} for the terms $\Theta_3$ and $\Theta_4$ respectively, we have
\begin{align*}
	&\|\Theta_3\|^2_{L^2(0,T;H^1(\Omega))} \leq C\ \|z_{u_2}\|^2_{L^\infty(0,T;H^2(\Omega))} \|\shat{m}\|^2_{L^2(0,T;H^3(\Omega))} +  C\ \|\shat{m}\|^2_{L^\infty(0,T;H^2(\Omega))} \|z_{u_2}\|^2_{L^2(0,T;H^3(\Omega))},  \\
	&\|\Theta_4\|^2_{L^2(0,T;H^1(\Omega))} \leq C\ \|z_{u_2}\|^2_{L^\infty(0,T;H^2(\Omega))} \|\shat{u}\|^2_{L^2(0,T;H^1(\Omega))} + C\ \|\shat{m}\|^2_{L^\infty(0,T;H^2(\Omega))}\|u\|^2_{L^2(0,T;H^1(\Omega))}.
\end{align*}
Finally, applying estimate \eqref{EE-5} for the terms $\Theta_5$,$\Theta_6$ and $\Theta_7$, and then combining them we deduce
\begin{align*}
	&\|\Theta_5\|^2_{L^2(0,T;H^1(\Omega))}+\|\Theta_6\|^2_{L^2(0,T;H^1(\Omega))} + \|\Theta_7\|^2_{L^2(0,T;H^1(\Omega))}\\
	& \ \ \ \ \ \leq C\ \|z_{u_2}\|^2_{L^\infty(0,T;H^2(\Omega))} \|\shat{m}\|^2_{L^\infty(0,T;H^2(\Omega))} \|u_1\|^2_{L^2(0,T;H^1(\Omega))}\\
	&\ \ \ \ \ \ \ \  + C\ \|z_{u_2}\|^2_{L^\infty(0,T;H^2(\Omega))} \|m_{u_2}\|^2_{L^\infty(0,T;H^2(\Omega))} \|\shat{u}\|^2_{L^2(0,T;H^1(\Omega))}\\
	& \ \ \ \ \ \ \ \ + C\ \left(\|m_{u_1}\|^2_{L^\infty(0,T;H^2(\Omega))} + \|m_{u_2}\|^2_{L^\infty(0,T;H^2(\Omega))}\right)  \|\shat{m}\|^2_{L^\infty(0,T;H^2(\Omega))} \|u\|^2_{L^2(0,T;H^1(\Omega))}.	
\end{align*}
Next, we combine all these estimates and substitute in \eqref{DCTS-E4}.
Since $m_{u_k}$ is a regular solution of system \eqref{NLP} and $u_k\in \mathcal{U}_R$ for $k=1,2$, therefore $\|m_{u_k}\|_{\mathcal{M}}\leq C(\Omega,T,R)$. Now, applying these bounds in the energy estimate for $z_{u_2}$, we find
$\|z_{u_2}\|_{\mathcal{M}}\leq C\ \|u\|_{L^2(0,T;H^1(\Omega))}$. 
Substituting these inequalities in \eqref{DCTS-E4} and using the Lipschitz continuity of control-to-state operator from estimate \eqref{LCCTSO}, we derive
$$\|\shat{z}\|^2_{L^{\infty}(0,T;H^2(\Omega))}+ \|\shat{z}\|^2_{L^2(0,T;H^3(\Omega))} +  \|\shat{z}_t\|^2_{L^2(0,T;H^1(\Omega))} \leq C(\Omega,T,R)\ \|\shat{u}\|^2_{L^2(0,T;H^1(\Omega))} \|u\|^2_{L^2(0,T;H^1(\Omega))}.$$
This leads to the proof of \eqref{LC-CTS}.
\end{proof}

\begin{Cor}\label{C-CTSD}
The Fr\'echet derivative of the control-to-state operator satisfies the following estimate:
\begin{equation}\label{FDCTS-S}
	\|G'(u)[v]\|_{\mathcal{M}} \leq C\ \|u\|_{L^2(0,T;H^1(\Omega))} \|v\|_{L^2(0,T;H^1(\Omega))}.
\end{equation} 	
\end{Cor}

	Next, we study the solvability of the adjoint problem. While deriving the optimality condition we will work with the weak solution of the adjoint equation instead of the strong solution, which allows to work with the target function $\nabla m_d\in L^6(0,T;L^6(\Omega))$.

Suppose $(\widetilde{m},\widetilde{u})$ be an admissible pair of the control problem. Then consider the following linear system: 
\begin{equation}\label{CLAS}
	(AL-LLG)\begin{cases}
		\begin{array}{l}
			\mathcal{E}_{\widetilde{u}}\phi=g \ \ \ \ \text{in}\ \Omega_T,\\
			\frac{\partial \phi}{\partial \eta}=0 \ \ \ \ \ \  \text{in}\ \partial \Omega_T, \ \ \ \phi(x,T)=\phi_T\ \ \text{in}\ \Omega,
		\end{array}
	\end{cases}	
\end{equation}
where the operator $\mathcal{E}_{\widetilde{u}}$ is defined as
\begin{equation*}\label{CLAO}
	\mathcal{E}_{\widetilde{u}}\phi:= \phi_t + \Delta \phi  + |\nabla \widetilde{m}|^2\phi -2 \nabla \cdot \big((\widetilde{m}\cdot \phi)\nabla \widetilde{m}\big) + \Delta (\phi \times \widetilde{m})+(\Delta  \widetilde{m}\times \phi)-(\phi \times \widetilde{u}) + \big((\phi \times \widetilde{m})\times \widetilde{u}\big) + \big(\phi \times (\widetilde{m}\times \widetilde{u})\big).	
\end{equation*}
The weak formulation of \eqref{CLAS} is same as \eqref{WAF} with $\int_0^T \langle g, \upsilon\rangle \ dt$ replacing the term on the right hand side of \eqref{WAF}.

\begin{Lem}\label{AL-SLS}
	Suppose $(\widetilde{m},\widetilde{u})$ is an admissible pair, that is, $(\widetilde{m},\widetilde{u}) \in \mathcal{A}$. Then, for any $g\in L^2(0,T;H^1(\Omega)^*)$ there exists a unique weak solution $\phi \in \mathcal{Z}$ of the linear system \eqref{CLAS} in the sense Definition \ref{AWSD}. Moreover, the following estimate holds:
	\begin{eqnarray}\label{ALSSE}
		\lefteqn{\|\phi\|^2_{L^{\infty}(0,T;L^2(\Omega))}+ \|\phi\|^2_{L^2(0,T;H^1(\Omega))} + \|\phi_t\|^2_{L^2(0,T;H^1(\Omega)^*)}\leq \left(\|\phi_T\|^2_{L^2(\Omega)}+  \|g\|^2_{L^2(0,T;H^1(\Omega)^*)}\right) }\nonumber\\
		&&\hspace{0.5cm}\times \ \exp\bigg\{C  \left(\|\widetilde{m}\|^2_{L^\infty(0,T;H^2(\Omega))}\|\widetilde{m}\|^2_{L^2(0,T;H^3(\Omega))}+\|\widetilde{m}\|^2_{L^\infty(0,T;H^2(\Omega))} \|\widetilde{u}\|^2_{L^2(0,T;H^1(\Omega))}\right) \bigg\}.\ \ \ \ 
	\end{eqnarray}
\end{Lem}

\begin{proof}
		We will prove this theorem using the Galerkin method. 
Let $\{w_j\}_{j=1}^{\infty}$ be an orthonormal basis of $L^2(\Omega)$ consisting of eigen vectors for the operator $-\Delta + I$ with vanishing Neumann boundary condition. Suppose $W_n=span\{w_1,w_2,...,w_n\}$ and $\mathbb{P}_n:L^2 \to W_n$ be the orthogonal projection. 
Using the solvability of ODEs, we can find a solution $\phi_n=\sum_{j=1}^{n}g_{jn}(t)w_j$ of the following approximated system for each $j=1,\cdots,n$ and $0\leq t\leq T$:
\begin{equation}\label{APGS}
	\begin{cases}
		-\big(\phi_n'(t),w_j\big)+\big(\nabla\phi_n(t),\nabla w_j\big)=\big(|\nabla \widetilde{m}(t)|^2\phi_n(t),w_j\big)+2\  \big((\widetilde{m}(t)\cdot \phi_n(t))\nabla \widetilde{m}(t),\nabla w_j\big)\\
		\hspace{0.5in}-\ \big(\nabla  (\phi_n(t) \times \widetilde{m}(t)),\nabla w_j\big)+\big((\Delta \widetilde{m}(t)\times \phi_n(t)),w_j\big)-\ \big((\phi_n(t) \times  \widetilde{u}(t)  ),w_j\big)\\
		\hspace{0.5in}+\  \big((\phi_n(t)\times \widetilde{m}(t))\times \widetilde{u}(t),w_j\big)+\big(\phi_n(t) \times (\widetilde{m}(t)\times \widetilde{u}(t)),w_j\big) -\ \langle g(t), w_j \rangle, \\
		\phi_n(T)=\mathbb{P}_n\big(\phi_T\big).
	\end{cases}
\end{equation}
Now, multiplying \eqref{APGS} by $g_{jn}(t)$, summing over $j=1,...,n,$ and using the property $a\cdot (a \times b)=0$, we have
\begin{eqnarray*}
	\lefteqn{-\frac{1}{2}\frac{d}{dt}\|\phi_n(t)\|^2_{L^2(\Omega)}+ \int_\Omega |\nabla \phi_n(t)|^2 dx = \int_\Omega |\nabla \widetilde{m}|^2 |\phi_n|^2\ dx+2 \int_\Omega \big( (\widetilde{m}\cdot \phi_n)\nabla \widetilde{m}\big)\cdot \nabla \phi_n \ dx}\\
	&&\hspace{1cm}-\int_\Omega \nabla(\phi_n \times \widetilde{m})\cdot \nabla\phi_n \ dx + \int_\Omega \big((\phi_n \times \widetilde{m}) \times \widetilde{u}\big)\cdot \phi_n \ dx - \langle g(t),\phi_n(t) \rangle,\hspace{2cm}
\end{eqnarray*}	 
Applying H\"older's inequality and the embeddings $H^1(\Omega)\hookrightarrow L^4(\Omega)$, $H^2(\Omega) \hookrightarrow L^\infty(\Omega)$, we find
\begin{eqnarray*}
	\lefteqn{-\frac{1}{2}\frac{d}{dt}\|\phi_n(t)\|^2_{L^2(\Omega)}+ \frac{1}{2}\int_\Omega |\nabla \phi_n(t)|^2 dx\leq  C\ \|g(t)\|^2_{H^1(\Omega)^*}  } \\
	&&+\ C\ \|\widetilde{m}(t)\|^2_{H^2(\Omega)} \left(\|\widetilde{m}(t)\|^2_{H^3(\Omega)} +\|\widetilde{u}(t)\|^2_{H^1(\Omega)}\right)  \|\phi_n(t)\|^2_{L^2(\Omega)}.\hspace{3cm} \nonumber
\end{eqnarray*}
%\begin{align*}
%	\lefteqn{-\frac{1}{2}\frac{d}{dt}\|\phi_n(t)\|^2_{L^2(\Omega)}+ \frac{1}{2}\int_\Omega |\nabla \phi_n(t)|^2 dx\leq  \|g(t)\|^2_{H^1(\Omega)^*}  }\\
%	&\ \ \ \ \ \ \ \ +\ C\ \|\widetilde{m}(t)\|^2_{H^2(\Omega)} \left(\|\widetilde{m}(t)\|^2_{H^3(\Omega)} +\|\widetilde{u}(t)\|^2_{H^1(\Omega)}\right)  \|\phi_n(t)\|^2_{L^2(\Omega)}.
%\end{align*}
%$$-\frac{1}{2}\frac{d}{dt}\|\phi_n(t)\|^2_{L^2(\Omega)}+ \frac{1}{2}\int_\Omega |\nabla \phi_n(t)|^2 dx\leq  \|g(t)\|^2_{H^1(\Omega)^*}  +\ C\ \|\widetilde{m}(t)\|^2_{H^2(\Omega)} \left(\|\widetilde{m}(t)\|^2_{H^3(\Omega)} +\|\widetilde{u}(t)\|^2_{H^1(\Omega)}\right)  \|\phi_n(t)\|^2_{L^2(\Omega)}.$$
By the application of Gronwall's inequality and using the inequality $\|\phi_n(T)\|_{L^2(\Omega)}\leq \|\phi(T)\|_{L^2(\Omega)}$, we derive
\begin{align}
	&\|\phi_n(t)\|^2_{L^2(\Omega)}+\int_t^T\int_\Omega |\nabla \phi_n(s)|^2 dx\ ds \leq \Big(\|\phi_T\|^2_{L^2(\Omega)} +  \|g\|^2_{L^2(0,T;H^1(\Omega)^*)}  \Big) \nonumber\\
	&\hspace{2.3cm}\times \exp{\left\{C \ \|\widetilde{m}\|^2_{L^\infty(0,T;H^2(\Omega))} \int_0^T \left(\|\widetilde{m}(s)\|^2_{H^3(\Omega)}+ \|\widetilde{u}(s)\|^2_{H^1(\Omega)}\right)ds\right\}}, \ \ \ \  \forall \  t\in [0,T).\label{AEE1}
\end{align}
%Since the right hand side of the above estimate is independent of $n$ and $\widetilde{m}\in \mathcal{M}$, $\{\phi_n\}$ is uniformly bounded in $L^\infty(0,T;L^2(\Omega)) \cap L^2(0,T;H^1(\Omega))$.
Given that the right-hand side of the above estimate is independent of $n$ and $\widetilde{m}\in \mathcal{M}$, we can conclude that the sequence $\{\phi_n\}$ is uniformly bounded in the space  $L^\infty(0,T;L^2(\Omega)) \cap L^2(0,T;H^1(\Omega))$.
Now, proceeding in a way similar to Theorem 5.2 of \cite{SPSK}, we find the following
%As the above estimate holds for every $v \in H^1(\Omega)$, we obtain upon integration on $[0,T]$ that
%$$\|\phi_n'\|_{H^{-1}(\Omega)} \leq C\left(\|\widetilde{m}(t)\|^2_{H^2(\Omega)}\|\phi_n(t)\|_{H^1(\Omega)}+\|\widetilde{u}(t)\|_{H^1(\Omega)}\|\phi_n(t)\|_{L^2(\Omega)}+\|\widetilde{m}(t)-m_d(t)\|_{L^2(\Omega)}\right).$$
%An integration over $0$ to $T$ gives
\begin{align}
	\int_0^T \|\phi_n'(t)\|^2_{H^1(\Omega)^*} dt &\leq C\  \left( \|g\|^2_{L^2(0,T;H^1(\Omega)^*)} + \|\widetilde{m}\|^4_{L^\infty(0,T;H^2(\Omega))} \|\phi_n\|^2_{L^2(0,T;H^1(\Omega))} \right.&\nonumber\\
	& \left.\ \ \ \ \  +\   \|\widetilde{m}\|^2_{L^\infty(0,T;H^2(\Omega))} \|\widetilde{u}\|^2_{L^2(0,T;H^1(\Omega))} \|\phi_n\|^2_{L^\infty(0,T;L^2(\Omega))}\right).\nonumber
\end{align}
%\begin{flalign*}
%\lefteqn{\int_0^T \big(\phi_n'(t),v(t)\big)\ dt \leq C\int_0^T \left( \|\widetilde{m}(t)\|^2_{H^2(\Omega)}\|\phi_n(t)\|_{H^1(\Omega)}\|v(t)\|_{H^1(\Omega)}+\|\widetilde{u}(t)\|_{H^1(\Omega)}\|\phi_n(t)\|_{L^2(\Omega)}\|v(t)\|_{H^1(\Omega)}\right.}\\
%&\hspace{2in}\left.+ \|\widetilde{m}(t)-m_d(t)\|_{L^2(\Omega)}\|v(t)\|_{L^2(\Omega)}\right) dt\\
% & \leq \|\widetilde{m}\|^2_{L^\infty(0,T;H^2(\Omega))}\|\phi_n\|_{L^2(0,T;H^1(\Omega))}\|v\|_{L^2(0,T;H^1(\Omega))} + \|\widetilde{u}\|_{L^2(0,T;H^1(\Omega))}\|\phi_n\|_{L^\infty(0,T;L^2(\Omega))}\|v\|_{L^2(0,T;H^1(\Omega))}&\\
% &\hspace{2in}+ \left(\|\widetilde{m}\|_{L^2(0,T;L^2(\Omega))}+ \|m_d\|_{L^2(0,T;L^2(\Omega))}\right)\|v\|_{L^2(0, T;L^2(\Omega))}.	
%\end{flalign*}
Therefore, using the uniform bound for $\phi_n$ in $L^\infty(0,T;L^2(\Omega))$ and $L^2(0,T;H^1(\Omega))$ from inequality \eqref{AEE1} in the above estimate, we find that $\{\phi_n^\prime\}$ is uniformly bounded in $L^2(0, T;H^1(\Omega)^*)$. Appealing to Aloglu weak* compactness and reflexive weak compactness theorems, we have
\begin{eqnarray} \left\{\begin{array}{cccll}
		\phi_n &\overset{w}{\rightharpoonup} & \phi \  &\mbox{weakly in}& \ L^2(0,T;H^1(\Omega)),\\
		\phi_n &\overset{w^\ast}{\rightharpoonup} & \phi \ &\mbox{weak$^*$ in}& \ L^{\infty}(0,T;L^2(\Omega)),\\
		\phi^\prime_n &\overset{w}{\rightharpoonup} & \phi^\prime \ &\mbox{weakly in}& \ L^2(0,T;H^1(\Omega)^*), \ \ \mbox{as} \  \ n\to \infty. \label{AEWC}
	\end{array}\right.	
\end{eqnarray}

Again the Aubin-Lions-Simon lemma (see, Corollary 4, \cite{JS}) establishes the existence of a sub-sequence of $\{\phi_n\}$ (again denoted as $\{\phi_n\}$) such that $\phi_n \overset{s}{\to} \phi$ strongly in $L^2(0,T;L^2(\Omega))$. Using this strong convergence along with \eqref{AEWC}, we can verify  that $\phi$ satisfies equality (i) of Definition \ref{AWSD} for every $\vartheta \in \text{span}(w_1,w_2,...)$. Further, as such functions are dense in $H^1(\Omega)$, it  holds true for every $\vartheta \in H^1(\Omega)$.  Since, for any  $\alpha\in L^2(0,T)$ and $\vartheta_j \in H^1(\Omega)$, the elements of the form $\sum_{j=1}^{n} \alpha_j(t)\ \vartheta_j$ are dense in $L^2(0,T;H^1(\Omega))$, so equality (i) holds for every $\vartheta \in L^2(0,T;H^1(\Omega))$. 
%Also, as $\phi\in  L^2(0,T;H^1(\Omega))$ and $\phi_t\in  L^2(0,T;H^1(\Omega)^*)$, it implies that $\phi\in  C([0,T];L^2(\Omega)).$ It is sufficient to verify the terminal condition $\phi(T)=\phi_T$ in a standard way. Finally, taking weak sequential lower semi-continuity for $\{\phi_n\}$ in \eqref{AEE1}, we obtain the estimate \eqref{ALSSE}. The uniqueness of weak solutions can be directly found by taking $\phi=\phi_1-\phi_2$ in the energy estimate \eqref{ALSSE}, where $\phi_1$ and $\phi_2$ are two weak solution of the linear adjoint system. The proof is thus completed.
% \begin{align}
	%	&\int_0^T \left(\|\phi(s)\|^2_{H^1(\Omega)} +\|\phi'(s)\|^2_{H^{-1}(\Omega)}\right) ds \leq \left(\|\widetilde{m}(T)-m_\Omega\|^2_{L^2(\Omega)} + \int_0^T \|\widetilde{m}(s)-m_d(s)\|^2_{L^2(\Omega)} ds \right)\nonumber\\
	%	& \hspace{1.9in}\times \|\widetilde{m}\|^4_{L^\infty(0,T;H^2(\Omega))}\exp{\left\{C \int_0^T \left(\|\widetilde{m}(s)\|^2_{H^3(\Omega)}+ \|\widetilde{u}(s)\|^2_{H^1(\Omega)}\right)ds\right\}}.\nonumber&	
	%\end{align}
	%Hence the proof. 
\end{proof}
Finally, we can conclude the proof of Theorem \ref{T-AWS}.
\begin{proof}[Proof of Theorem \ref{T-AWS}:] The result can be readily derived as a direct consequence of Lemma \ref{AL-SLS} by making the substitution $g=\nabla \cdot \left( |\nabla \widetilde{m}-\nabla m_d|^2 \big(\nabla \widetilde{m}-\nabla m_d\big)\right)$.
	
\noindent For any element $\upsilon \in L^2(0,T;H^1(\Omega))$, the following estimation for the dual product of $g$ holds:
%\begin{align}\label{ASEED}
%	&\int_0^T \langle g,\upsilon \rangle \ dt =-\int_0^T \int_{\Omega}\left( |\nabla \widetilde{m}-\nabla m_d|^2 \big(\nabla \widetilde{m}-\nabla m_d\big)\right)\cdot \nabla \upsilon \ dx\ dt\nonumber\\
%	&\leq \int_0^T \|\nabla \widetilde{m}-\nabla m_d\|^3_{L^6(\Omega)}\ \|\nabla \upsilon\|_{L^2(\Omega)}\ dt \leq C\ \|\nabla \widetilde{m}-\nabla m_d\|^3_{L^6(0,T;L^6(\Omega))}\ \|\upsilon\|_{L^2(0,T;H^1(\Omega))}. 
%\end{align}
\begin{align}\label{AEEED}
	&\int_0^T \langle g,\upsilon \rangle \ dt =-\int_0^T \int_{\Omega}\left( |\nabla \widetilde{m}-\nabla m_d|^2 \big(\nabla \widetilde{m}-\nabla m_d\big)\right)\cdot \nabla \upsilon \ dx\ dt\nonumber\\
	&\leq \int_0^T \|\nabla \widetilde{m}-\nabla m_d\|^3_{L^6(\Omega)}\ \|\nabla \upsilon\|_{L^2(\Omega)}\ dt \leq C\ \|\nabla \widetilde{m}-\nabla m_d\|^3_{L^6(0,T;L^6(\Omega))}\ \|\upsilon\|_{L^2(0,T;H^1(\Omega))}. 	
\end{align}
Therefore, in order to make this dual product meaningful we had to choose $m_d$ such that $\nabla m_d\in L^6(0,T;L^6(\Omega))$. Hence the estimate \eqref{AEEE} follow from \eqref{ALSSE} and \eqref{AEEED}.
\end{proof}

\subsection{First Order Optimality Condition}

The first-order optimality condition for nonconvex problems is a basic criterion used to analyze critical points in optimization. In contrast to convex problems, where this condition guarantees local minimality, it is not sufficient for determining the optimality of critical points in nonconvex scenarios. Since nonconvex problems can have multiple local minima or maxima, additional analysis is needed to assess their optimality. This may involve considering higher-order derivatives or specific problem structures.

\begin{proof}[Proof of Theorem \ref{FOOCT}]
	Clearly, $\mathcal{U}_{a,b}$ is a convex set. Since, the set of controls in $L^2(0,T;H^1(\Omega))$ for which there exists a regular solution in $\mathcal{M}$ is open, for any control $u\in \mathcal{U}_{ad}$ and $\epsilon>0$ sufficiently small, we have $\widetilde{u}+\epsilon (u-\widetilde{u}) \in \mathcal{U}$. Hence, $\widetilde{u}+\epsilon (u-\widetilde{u}) \in \mathcal{U}_{ad}$ for sufficiently small $\epsilon$.
	
	Now, as $\widetilde{u} \in \mathcal{U}_{ad}
	$ is an optimal control of MOCP and the control-to-state operator is Fr\'echet differentiable, the functional $\mathcal{I}(u)=\mathcal{J}(G(u),u)$ satisfies the following inequality:
	\begin{equation}\label{CFG0}
		D_u\mathcal{I}(\widetilde{u})(u-\widetilde{u})= \lim_{\epsilon \to 0^+} \frac{\mathcal{I}(\widetilde{u} + \epsilon (u - \widetilde{u}))-\mathcal{I}(\widetilde{u})}{\epsilon} \geq 0, \ \ \ \  \ \forall \ u \in \mathcal{U}_{ad}.
	\end{equation}
	Now, by setting $v=u-\widetilde{u}$ and applying chain rule, we have
	\begin{flalign}\label{FD1}
		D_u \mathcal{I} (\widetilde{u})\cdot v &= D_m\mathcal J(\widetilde{m},\widetilde{u}) \circ\left[ D_uG(\widetilde{u})\cdot v\right] + D_u\mathcal  J(\widetilde{m},\widetilde{u})\cdot v\nonumber\\
		&=\int_{\Omega_T} \widetilde{u}\cdot v\ dx\ dt + \int_{\Omega_T} \nabla \widetilde{u} \cdot \nabla v\ dx\ dt+ \int_\Omega \left(\widetilde{m}(x,T)-m_\Omega\right)\cdot z(x,T)\ dx\nonumber\\
		&\hspace{1.3cm} + \int_{\Omega_T} |\nabla \widetilde{m}-\nabla m_d|^2 (\nabla \widetilde{m}-\nabla m_d)\cdot \nabla z\ dx\ dt,
	\end{flalign}
	where $z_v:=D_uG(\widetilde{u})\cdot v\in \mathcal{M}$ \ is a unique regular solution of the linearized system \eqref{LS2}.
	
	Now, we take $\vartheta =z_v$ in weak adjoint formulation of \eqref{AS}. Then, doing a space integration by parts for the last term using $\frac{\partial z_v}{\partial \eta}=0$, we obtain
	\begin{eqnarray}
		\lefteqn{-\int_0^T(\phi(t),z_v'(t)) \ dt+\big(\widetilde{m}(x,T)-m_{\Omega}(x),z_v(T) \big)- \int_{\Omega_T} \nabla \phi \cdot \nabla z_v\ dx \ dt + \int_{\Omega_T} |\nabla \widetilde{m}|^2\phi \cdot z_v \ dx\ dt   }\nonumber\\
		&&+2 \int_{\Omega_T} (\widetilde{m}\cdot \phi)\nabla \widetilde{m}\cdot \nabla z_v\ dx\ dt - \int_{\Omega_T} \nabla (\phi \times \widetilde{m})\cdot \nabla z_v\ dx \ dt+ \int_{\Omega_T}(\Delta \widetilde{m}\times \phi)\cdot z_v\ dx\ dt\nonumber\hspace{2cm}\\
		&&-\int_{\Omega_T} (\phi \times \widetilde{u})\cdot z_v\ dx\ dt + \int_{\Omega_T} \big((\phi \times \widetilde{m})\times \widetilde{u}\big)\cdot z_v\ dx\ dt+ \int_{\Omega_T} \big(\phi \times (\widetilde{m}\times \widetilde{u})\big)\cdot z_v\ dx\ dt\nonumber\\
		&&=-\int_{\Omega_T}  |\nabla \widetilde{m}-\nabla m_d|^2 (\nabla \widetilde{m}-\nabla m_d)\cdot \nabla z_v \ dx\ dt,\label{FOE1}  \ 
	\end{eqnarray}
	where we also used the following identity: 
	\begin{eqnarray}\label{TIP}
		\int_{0}^{T} \big( z_t,\phi\big) \ dt+\int_0^T \big\langle\phi_t, z\big\rangle \ dt= \int_\Omega \left(\widetilde{m}(x,T)-m_\Omega\right)\cdot z(x,T)\ dx.
	\end{eqnarray}
	On the other hand, considering inner product of the linearized system \eqref{LS2} with $\phi$, we have 
	\begin{eqnarray}
		\lefteqn{\int_{0}^{T} \big( z_v',\phi\big) \ dt = -\int_{\Omega_T}\nabla  z_v \cdot \nabla  \phi \ dx\ dt+\int_{\Omega_T} |\nabla \widetilde{m}|^2 z_v \cdot \phi \ dx\ dt}\nonumber\\
		&&+ 2\int_{\Omega_T}   (\widetilde{m}\cdot \phi)\big(\nabla \widetilde{m}\cdot \nabla z_v\big) \ dx\ dt+\int_{\Omega_T}(z_v \times \Delta \widetilde{m})\cdot \phi\ dx\ dt-\int_{\Omega_T}  \nabla (\phi \times \widetilde{m})\cdot \nabla z_v \ dx\ dt\nonumber\\
		&&+\int_{\Omega_T}(z_v \times \widetilde{u})\cdot \phi \ dx\  dt - \int_{\Omega_T}\big( z_v \times (\widetilde{m}\times \widetilde{u})\big)\cdot \phi \ dx\ dt-\int_{\Omega_T} \big( \widetilde{m}\times (z_v \times \widetilde{u})\big)\cdot \phi \ dx\ dt\nonumber\\
		&&+\int_{\Omega_T}(\widetilde{m}\times v)\cdot \phi \ dx\ dt- \int_{\Omega_T}\big(\widetilde{m}\times (\widetilde{m}\times v)\big)\cdot \phi \ dx\ dt \label{FOE2}\ \ \ \ \ \ 
	\end{eqnarray}
	Combining equations \eqref{FOE1} and \eqref{FOE2}, then substituting the result in \eqref{FD1}, we obtain 
	\begin{equation*}
		D_u\mathcal{I}(\widetilde{u})\cdot v= \int_{\Omega_T} \widetilde{u}\cdot v\ dx\ dt+ \int_{\Omega_T} \nabla \widetilde{u} \cdot \nabla v\ dx\ dt +\int_{\Omega_T} \Big((\phi \times \widetilde{m})+  \widetilde{m} \times (\phi \times \widetilde{m})\Big)\cdot v\ dx\ dt.\nonumber
% \ \ \ \ \ \ \forall \ v \in \mathcal{U}_{ad}.\nonumber        v may not be in U_ad.
	\end{equation*}
	At last using \eqref{CFG0}, applying the cross product property $(a\times b)\cdot c=b \cdot (c\times a)$ for the last integral on the right hand side, and putting $v=u-\widetilde{u}$, we get the required optimality condition \eqref{FOOC}. 
\end{proof}

\section{Second Order Optimality Condition}

In convex optimal control problems, controls that satisfy the first-order necessary optimality conditions are globally optimal. However, for non-convex optimal control problems, additional analysis, including higher derivative analysis, is required to ensure local optimality. For optimal control problems governed by the Landau-Lifshitz-Gilbert equations, the second-order sufficient optimality conditions play a crucial role in the numerical analysis of these non-convex problems. These conditions are essential in guaranteeing the feasibility and local optimality of a control, considering the intricate dynamics and nonlinearity of the LLG  equations. 
%While convex problems have straightforward optimality conditions, non-convex problems involving the LLG equations demand advanced techniques and the incorporation of second-order conditions to ascertain optimality.

Due to the availability of a weak solution for the adjoint system, we are empowered to establish an operator that associates each control in $\mathcal{U}_R$ with the corresponding adjoint solution. Henceforth, we shall refer to this operator as the "control-to-costate" operator. This operator, denoted as $\Phi:\mathcal{U}_R \to \mathcal{Z}$, effectively maps controls from the set $\mathcal{U}_R$ to their respective adjoint solutions in $\mathcal{Z}$.

\subsection{Control-to-Costate Operator}\label{S-CTCO}

\begin{Lem}\label{LCCTC}
	The control-to-costate operator $\Phi:\mathcal{U}_R\to \mathcal{Z}$ is Lipschitz continuous, that is, there exists a constant $C_3>0$ depending on $\Omega,T,R,m_0$ such that
	\begin{equation}\label{LCI}
		\|\Phi(u_1)-\Phi(u_2)\|_{\mathcal{Z}} \leq C_3 \ \|u_1-u_2\|_{L^2(0,T;H^1(\Omega))}, \ \ \ \ \ \ \ \forall \ u_1,u_2\in\mathcal{U}_R.
	\end{equation}
\end{Lem}
\noindent The proof of Lemma \ref{LCCTC} can be done using a comparable approach to that employed in proving Lemma \ref{L-LCCTS}.

\begin{Pro}\label{P-CTC}
	Suppose $\bar{u}\in \mathcal{U}_R$ be the control and $m_{\bar{u}}$ be it's corresponding regular solution. Then the following two conclusions hold:
	\begin{enumerate}[label=(\roman*)]
		\item The control-to-costate mapping $\Phi$ is Fr\'echet differentiable on $\mathcal{U}_R$, that is for any $\bar{u}\in \mathcal{U}_R$, there exists a bounded linear operator $\Phi'(\bar{u}):L^2(0,T;H^1(\Omega)) \to \mathcal{Z}$ such that
		$$\frac{\|\Phi(\bar{u}+u)-\Phi(\bar{u})-\Phi'(\bar{u})[u]\|_{\mathcal{Z}}}{\|u\|_{L^2(0,T;H^1(\Omega))}}\to 0 \ \ \ \ \ \text{as}  \ \|u\|_{L^2(0,T;H^1(\Omega))}\to 0,$$
		where $\phi':=\Phi'(\bar{u})[u]$ is the unique weak solution of the following system:
		\begin{equation}\label{ASD}
			\begin{cases}
				\mathcal{E}_{\bar{u}}\phi' = -2\ (\nabla m_{\bar{u}}\cdot  \nabla z)\phi +2\ \nabla \cdot \big((z \cdot \phi)\nabla m_{\bar{u}}\big)+ 2\ \nabla \cdot \big((m_{\bar{u}}\cdot \phi)\nabla z\big)-\Delta (\phi\times z)-(\Delta z \times \phi) \vspace{0.1cm}\\
				\hspace{1.7cm}+\ (\phi \times u)- \big((\phi \times z)\times \bar{u}\big) - \big((\phi \times m_{\bar{u}} )\times u\big)-\big(\phi \times (z \times \bar{u})\big)-\big(\phi \times (m_{\bar{u}}\times u)\big)\\
				\hspace{1.7cm} +\ \nabla \cdot \big(2\ \big((\nabla m_{\bar{u}}-\nabla m_d)\cdot \nabla z\big)(\nabla m_{\bar{u}}-\nabla m_d)\big)+ \nabla \cdot \big(|\nabla m_{\bar{u}}-\nabla m_d|^2 \nabla z\big)\ \ \ \ \ \text{in}\ \Omega_T,\vspace{0.2cm}\\
				\frac{\partial \phi'}{\partial \eta}=0 \ \ \ \text{in}\ \partial \Omega_T,\vspace{0.1cm}\\
				\phi'(T)=z(T)\ \ \ \text{in} \ \Omega.
			\end{cases}
		\end{equation}
		\item The Fr\'echet derivative $\Phi'$ is Lipschitz continuous, that is, for any controls $u_1,u_2 \in \mathcal{U}_R$ and $u\in L^2(0,T;H^1(\Omega))$, there exists a constant $C_4>0$ depending on $\Omega,T,R,m_0$ such that 
		\begin{equation}\label{LC-AD}
		\|\Phi'(u_1)[u]-\Phi'(u_2)[u] \|_{\mathcal{Z}} \leq C_4\  \|u_1-u_2\|_{L^2(0,T;H^1(\Omega))} \|u\|_{L^2(0,T;H^1(\Omega))}.	
		\end{equation}
		\end{enumerate}
\begin{proof}
		Suppose $\bar{u}\in \mathcal{U}_R$ be any fixed control and $m_{\bar{u}}$ be the corresponding regular solution. Let $z$ be the unique regular solution of the linearized system \eqref{LS2} and $\phi_{\bar{u}}$ be the unique weak solution of the adjoint system \eqref{AS}. As a consequence of Lemma \ref{AL-SLS}, we can easily show that the system \eqref{ASD} has a unique weak solution $\phi'$. Note that, here $\phi'$ do not represents the derivative of the control-to-costate operator, but just a notation. Let $\shat{\phi}=\phi_{\bar{u}+u}-\phi_{\bar{u}}$, $\shat{m}=m_{\bar{u}+u}-m_{\bar{u}}$ and $w=m_{\bar{u}+u}-m_{\bar{u}}-z$. Then $\xi:=\phi_{\bar{u}+u}-\phi_{\bar{u}}-\phi'$ solves the following system weakly in the sense of Definition \ref{AWSD}: 
	\begin{equation}\label{LOCTS}
		\begin{cases}
			\displaystyle \mathcal{E}_{\bar{u}}\xi = \sum_{k=1}^{8} \Psi_k \ \ \ \ \text{in}\ \Omega_T,\\
			\frac{\partial \xi}{\partial \eta}=0 \ \ \ \ \text{on} \ \partial \Omega_T,  \ \ \ \ \xi(T)=m_{\bar{u}+u}(T)-m_{\bar{u}}(T)-z(T) \ \ \text{in} \ \Omega,	
		\end{cases}
	\end{equation}
	where $\mathcal{E}_{\bar{u}}$ is defined in \eqref{CLAS},  and the terms $\Psi_k$'s are given by
	\begin{flalign*}
		\Psi_{1} &=- \big(\nabla \shat{m} \cdot (\nabla m_{\bar{u}+u}+\nabla m_{\bar{u}}) \big) \ \shat{\phi} -|\nabla \shat{m}|^2 \phi_{\bar{u}} -2\ (\nabla m_{\bar{u}}\cdot \nabla w) \ \phi_{\bar{u}},&	\\
		\Psi_{2}&=2 \ \nabla \cdot \big((\shat{m}\cdot \shat{\phi})\nabla \shat{m}\big)+ 2\ \nabla \cdot \big((m_{\bar{u}}\cdot \shat{\phi})\nabla \shat{m}\big) + 2\ \nabla \cdot \big((\shat{m}\cdot \phi_{\bar{u}})\nabla \shat{m}\big)\\
		&\ \ \ \ \ \ \ \ \ \  + 2 \ \nabla \cdot \big((\shat{m}\cdot \shat{\phi})\nabla m_{\bar{u}}\big) + 2\  \nabla \big( (w\cdot \phi_{\bar{u}})\nabla m_{\bar{u}}\big) + 2\ \nabla \cdot \big((m_{\bar{u}}\cdot \phi_{\bar{u}})\nabla w\big),\\
		\Psi_{3}&= -\Delta (\shat{\phi}\times \shat{m}) -\Delta (\phi_{\bar{u}}\times w) ,\\
		\Psi_{4}&= - (\Delta \shat{m}\times \shat{\phi})-(\Delta w\times \phi_{\bar{u}}),\\ 
		\Psi_{5}&= \shat{\phi}\times u, \\
		\Psi_{6}&= - \big((\shat{\phi}\times \shat{m}) \times (\bar{u}+u)\big)-\big((\shat{\phi}\times m_{\bar{u}}) \times u\big)-\big((\phi_{\bar{u}}\times \shat{m})\times u\big) -\big((\phi_{\bar{u}}\times w)\times \bar{u}\big),\\
		\Psi_{7}&= - \big(\shat{\phi}\times (\shat{m} \times (\bar{u}+u))\big)-\big(\shat{\phi}\times (m_{\bar{u}} \times u)\big)-\big(\phi_{\bar{u}}\times (\shat{m} \times u)\big) -\big(\phi_{\bar{u}}\times (w\times \bar{u})\big),\\
		\Psi_{8}&= \nabla \cdot \big(\big(\nabla \shat{m}\cdot (\nabla m_{\bar{u}+u}+\nabla m_{\bar{u}}-2\nabla m_d)\big)\nabla \shat{m}\big)+ \nabla \cdot \big(|\nabla \shat{m}|^2(\nabla m_{\bar{u}}-\nabla m_d)\big)\\
		&\ \ \ \ \ \ \ \ \ \ \ +2\  \nabla \cdot \big(\big((\nabla m_{\bar{u}}-\nabla m_d)\cdot \nabla w\big)(\nabla m_{\bar{u}}-\nabla m_d)\big) + \nabla \cdot (|\nabla m_{\bar{u}}-\nabla m_d|^2\nabla w). 
	\end{flalign*}
Using Lemma \ref{AL-SLS}, we can find the following estimate on the solution of system \eqref{LOCTS}:
\begin{eqnarray}\label{DAO-E1}
	\lefteqn{ \|\xi\|^2_{L^2(0,T;H^1(\Omega))} + \|\xi_t\|^2_{L^2(0,T;H^1(\Omega)^*)}\leq \left(\|m_{\bar{u}+u}(T)-m_{\bar{u}}(T)-z(T)\|^2_{L^2(\Omega)}+ \sum_{k=1}^{8}  \|\Psi_k\|^2_{L^2(0,T;H^1(\Omega)^*)}\right) }\nonumber\\
	&&\hspace{1cm}\times \ \exp\bigg\{C  \left(\|m_{\bar{u}}\|^2_{L^\infty(0,T;H^2(\Omega))}\|m_{\bar{u}}\|^2_{L^2(0,T;H^3(\Omega))}+\|m_{\bar{u}}\|^2_{L^\infty(0,T;H^2(\Omega))} \|\bar{u}\|^2_{L^2(0,T;H^1(\Omega))}\right) \bigg\}.\ \ \ \ \ \ \ \ \ 
\end{eqnarray}
Applying estimate \eqref{AEE-1} for the $L^2(0,T;H^1(\Omega)^*)$ norm of $\Psi_1$, we obtain
\begin{flalign*}
	\|\Psi_1\|_{L^2(0,T;H^1(\Omega)^*)} &\leq C\  \big(\|m_{\bar{u}}\|_{L^\infty(0,T;H^2(\Omega))} + \|m_{\bar{u}+u}\|_{L^\infty(0,T;H^2(\Omega))}\big) \|\shat{m}\|_{L^\infty(0,T;H^2(\Omega))}\|\shat{\phi}\|_{L^2(0,T;L^2(\Omega))}&\\
	&+C\ \|\shat{m}\|^2_{L^\infty(0,T;H^2(\Omega))} \|\phi_{\bar{u}}\|_{L^2(0,T;L^2(\Omega))} + C\ \|m_{\bar{u}}\|_{L^\infty(0,T;H^2(\Omega))} \|w\|_{L^\infty(0,T;H^2(\Omega))} \|\phi_{\bar{u}}\|_{L^2(0,T;L^2(\Omega))}.
\end{flalign*}
Proceeding for $\Psi_2$ using estimate \eqref{AEE-2}, we have
%	\begin{flalign*}
%		\|\Psi_2\|_{L^2(0,T;H^1(\Omega)^*)} &\leq C\ \|\shat{m}\|_{L^\infty(0,T;H^2(\Omega))} \|\shat{\phi}\|_{L^\infty(0,T;L^2(\Omega))} \|\shat{m}\|_{L^2(0,T;H^3(\Omega))}&\\
%		& + C\ \|m_{\bar{u}}\|_{L^\infty(0,T;H^2(\Omega))} \|\shat{\phi}\|_{L^\infty(0,T;L^2(\Omega))} \|\shat{m}\|_{L^2(0,T;H^3(\Omega))}\\
%		&+ C\ \|\shat{m}\|_{L^\infty(0,T;H^2(\Omega))} \|\phi_{\bar{u}}\|_{L^\infty(0,T;L^2(\Omega))} \|\shat{m}\|_{L^2(0,T;H^3(\Omega))}\\
%		& + C\ \|\shat{m}\|_{L^\infty(0,T;H^2(\Omega))} \|\shat{\phi}\|_{L^\infty(0,T;L^2(\Omega))} \|m_{\bar{u}}\|_{L^2(0,T;H^3(\Omega))}\\
%		&+ C\ \|w\|_{L^\infty(0,T;H^2(\Omega))} \|\phi_{\bar{u}}\|_{L^\infty(0,T;L^2(\Omega))} \|m_{\bar{u}}\|_{L^2(0,T;H^3(\Omega))}\\
%		& + C\ \|m_{\bar{u}}\|_{L^\infty(0,T;H^2(\Omega))} \|\phi_{\bar{u}}\|_{L^\infty(0,T;L^2(\Omega))} \|w\|_{L^2(0,T;H^3(\Omega))}.		
%	\end{flalign*}
\begin{flalign*}
	\|\Psi_2\|_{L^2(0,T;H^1(\Omega)^*)} & \leq C\ \left(\|\shat{m}\|_{L^\infty(0,T;H^2(\Omega))}+\|m_{\bar{u}}\|_{L^\infty(0,T;H^2(\Omega))}\right) \|\shat{\phi}\|_{L^\infty(0,T;L^2(\Omega))} \|\shat{m}\|_{L^2(0,T;H^3(\Omega))}&\\
	%& + C\ \|m_{\bar{u}}\|_{L^\infty(0,T;H^2(\Omega))} \|\shat{\phi}\|_{L^\infty(0,T;L^2(\Omega))} \|\shat{m}\|_{L^2(0,T;H^3(\Omega))}\\
	&\hspace{0.5cm}+ C\ \left(\|\shat{\phi}\|_{L^\infty(0,T;L^2(\Omega))} +\|\phi_{\bar{u}}\|_{L^\infty(0,T;L^2(\Omega))}\right) \|\shat{m}\|_{L^\infty(0,T;H^2(\Omega))}  \|\shat{m}\|_{L^2(0,T;H^3(\Omega))}\\
%	& + C\ \|\shat{m}\|_{L^\infty(0,T;H^2(\Omega))} \|\shat{\phi}\|_{L^\infty(0,T;L^2(\Omega))} \|m_{\bar{u}}\|_{L^2(0,T;H^3(\Omega))}\\
	&\hspace{0.5cm}+ C\ \|w\|_{L^\infty(0,T;H^2(\Omega))} \|\phi_{\bar{u}}\|_{L^\infty(0,T;L^2(\Omega))} \|m_{\bar{u}}\|_{L^2(0,T;H^3(\Omega))}\\
	&\hspace{0.5cm} + C\ \|m_{\bar{u}}\|_{L^\infty(0,T;H^2(\Omega))} \|\phi_{\bar{u}}\|_{L^\infty(0,T;L^2(\Omega))} \|w\|_{L^2(0,T;H^3(\Omega))}.		
\end{flalign*}

Similarly using estimates \eqref{AEE-3}, \eqref{AEE-4} and \eqref{AEE-8} to evaluate the $L^2(0,T;H^1(\Omega)^*)$ norm of $\Psi_3$, $\Psi_4$ and $\Psi_5$, we find
	\begin{flalign*}
		\|\Psi_3\|_{L^2(0,T;H^1(\Omega)^*)} &\leq C\ \|\shat{\phi}\|_{L^2(0,T;H^1(\Omega))} \|\shat{m}\|_{L^\infty(0,T;H^2(\Omega))} + C\ \|\phi_{\bar{u}}\|_{L^2(0,T;H^1(\Omega))} \|w\|_{L^\infty(0,T;H^2(\Omega))}\\
		\|\Psi_4\|_{L^2(0,T;H^1(\Omega)^*)} &\leq C\ \|\shat{m}\|_{L^\infty(0,T;H^2(\Omega))}  \|\shat{\phi}\|_{L^2(0,T;H^1(\Omega))} + C\ \|w\|_{L^\infty(0,T;H^2(\Omega))} \|\phi_{\bar{u}}\|_{L^2(0,T;H^1(\Omega))}\\
		\|\Psi_5\|_{L^2(0,T;H^1(\Omega)^*)} &\leq C\ \|\shat{\phi}\|_{L^\infty(0,T;L^2(\Omega))} \|u\|_{L^2(0,T;H^1(\Omega))}.
	\end{flalign*}
	Now, for the terms $\Psi_6$ and $\Psi_7$ using estimates \eqref{AEE-5} and \eqref{AEE-6}, we have
	\begin{flalign*}
		\|\Psi_6\|_{L^2(0,T;H^1(\Omega)^*)}+ \|\Psi_7\|_{L^2(0,T;H^1(\Omega)^*)} &\leq C\ \|\shat{\phi}\|_{L^\infty(0,T;L^2(\Omega))} \|\shat{m}\|_{L^\infty(0,T;H^2(\Omega))} \big(\|\bar{u}\|_{L^2(0,T;H^1(\Omega))}+\|u\|_{L^2(0,T;H^1(\Omega))}\big)\\
		&\hspace{0.5cm}+ C\ \|\shat{\phi}\|_{L^\infty(0,T;L^2(\Omega))} \|m_{\bar{u}}\|_{L^\infty(0,T;H^2(\Omega))} \|u\|_{L^2(0,T;H^1(\Omega))}\\
		&\hspace{0.5cm}+ C\ \|\phi_{\bar{u}}\|_{L^\infty(0,T;L^2(\Omega))} \|\shat{m}\|_{L^\infty(0,T;H^2(\Omega))} \|u\|_{L^2(0,T;H^1(\Omega))}\\
		&\hspace{0.5cm} + C\ \|\phi_{\bar{u}}\|_{L^\infty(0,T;L^2(\Omega))} \|w\|_{L^\infty(0,T;H^2(\Omega))} \|\bar{u}\|_{L^2(0,T;H^1(\Omega))} .
	\end{flalign*}
	Finally, estimating for  $\Psi_8$ using \eqref{AEE-7}, and applying the embeddings $H^1(\Omega) \hookrightarrow L^6(\Omega)$ and the inequality $\|\cdot \|_{L^6(0,T)}\leq C\ \|\cdot \|_{L^\infty(0,T)}$, we find
	\begin{flalign*}
		\|\Psi_8\|_{L^2(0,T;H^1(\Omega)^*)} &\leq C\ \|\shat{m}\|^2_{L^\infty(0,T;H^2(\Omega))} \big(\|m_{\bar{u}+u}\|_{L^\infty(0,T;H^2(\Omega))}+\|m_{\bar{u}}\|_{L^\infty(0,T;H^2(\Omega))}+ \|\nabla m_d\|_{L^6(0,T;L^6(\Omega))}\big)&\\
		&+ C\ \|w\|_{L^\infty(0,T;H^2(\Omega))} \left(\|m_{\bar{u}}\|^2_{L^\infty(0,T;H^2(\Omega))}+\|\nabla m_d\|^2_{L^6(0,T;L^6(\Omega))}\right)
	\end{flalign*}
Then let us combine all the above estimates, substitute it in \eqref{DAO-E1} and divide by $\|u\|^2_{L^2(0,T;H^1(\Omega))}$. Since $\bar{u} \in \mathcal{U}_R$, so for small enough $u$ we can guarantee that $\bar{u}+u \in \mathcal{U}_R$. Hence we have the uniform bound for $\|m_{\bar{u}}\|_{\mathcal{M}}$ and $\|m_{\bar{u}+u}\|_{\mathcal{M}}$ for such small control functions $u$. Similarly, using estimate \eqref{AEEE} we can find a bound for $\|\phi_{\bar{u}}\|_{\mathcal{Z}}$. From the Lipschitz continuity of the control-to-state and control-to-costate operators given in estimate \eqref{LCCTSO} and \eqref{LCI} respectively, it is clear that $\|\shat{m}\|_{\mathcal{M}}\leq C\ \|u\|_{L^2(0,T;H^1(\Omega))}$ and $\|\shat{\phi}\|_{\mathcal{Z}}\leq C\ \|u\|_{L^2(0,T;H^1(\Omega))}$. Again, from estimate \eqref{X1}, we know that $\|w\|_{\mathcal{M}}\leq C\ \|u\|^2_{L^2(0,T;H^1(\Omega))}  + C\  \|u\|^3_{L^2(0,T;H^1(\Omega))}$. Therefore, using all these postulates we come to the conclusion that $\frac{\|\Psi_k\|_{L^2(0,T;H^1(\Omega)^*)}}{\|u\|_{L^2(0,T;H^1(\Omega))}}\to 0$ as $\|u\|_{L^2(0,T;H^1(\Omega))}\to 0$. Also, by Proposition \ref{P-CTS}, we have
\begin{align*}
&\frac{\|m_{\bar{u}+u}(T)-m_{\bar{u}}(T)-z(T)\|_{L^2(\Omega)}}{\|u\|_{L^2(0,T;H^1(\Omega))}}\leq  \frac{\|m_{\bar{u}+u}-m_{\bar{u}}-z\|_{L^\infty(0,T;L^2(\Omega))}}{\|u\|_{L^2(0,T;H^1(\Omega))}}\\
&\hspace{1.5cm} \leq \frac{\|m_{\bar{u}+u}-m_{\bar{u}}-z\|_{\mathcal{M}}}{\|u\|_{L^2(0,T;H^1(\Omega))}} \leq \frac{\|w\|_{\mathcal{M}}}{\|u\|_{L^2(0,T;H^1(\Omega))}} \to 0 \ \ \ \ \text{as}\ \|u\|_{L^2(0,T;H^1(\Omega))}\to 0 .	
\end{align*}
Hence, we conclude the Fr\`echet differentiability of the control-to-costate operator by showing $\frac{\|\xi\|_{\mathcal{Z}}}{\|u\|_{L^2(0,T;H^1(\Omega))}}\to 0$ as $\|u\|_{L^2(0,T;H^1(\Omega))}\to 0$.

Suppose $\phi'_{u_1}:= \Phi'(u_1)[u]$ and $\phi'_{u_2}:=\Phi'(u_2)[u]$ be two weak solution of system \eqref{ASD}. Let $\shat{\psi}=\phi'_{u_1}-\phi'_{u_2}$, $\shat{m}=m_{u_1}-m_{u_2}$ and $\shat{u}=u_1-u_2$.  Then $\shat{\psi}$ will weakly satisfy the following system:
	\begin{equation}\label{DOAO}
	\begin{cases}
		\displaystyle \mathcal{E}_{u_1} \shat{\psi} = \sum_{k=1}^{8} \mathcal{B}_k\ \ \ \ \text{in}\ \Omega_T,\\
		\frac{\partial \shat{\psi}}{\partial \eta}=0 \ \ \ \ \text{on} \ \partial \Omega_T,  \ \ \ \ \shat{\psi}(T)=z_{u_1}(T)-z_{u_2}(T) \ \ \text{in} \ \Omega,	
	\end{cases}
\end{equation}
where $\mathcal{E}_{u_1}$ is defined in \eqref{CLAS},  and the terms $\mathcal{B}_k$'s are given by
\begin{flalign*}
\mathcal{B}_1&= -\big(\nabla \shat{m} \cdot (\nabla m_{u_1}+\nabla m_{u_2})\big) \ \phi'_{u_2} -2\ (\nabla \shat{m}\cdot \nabla z_{u_1})\phi_{u_1}-2\ (\nabla m_{u_2}\cdot \nabla \shat{z})\phi_{u_1}-2\ (\nabla m_{u_2}\cdot \nabla z_{u_2})\shat{\phi}\\
\mathcal{B}_2&=2\ \nabla \cdot \big((\shat{m}\cdot \phi'_{u_2})\nabla m_{u_1}\big)+2\ \nabla \cdot \big((m_{u_2}\cdot \phi'_{u_2})\nabla \shat{m}\big)+2\ \nabla \cdot \big((\shat{z}\cdot \phi_{u_1})\nabla m_{u_1}\big)+2\ \nabla \cdot \big((z_{u_2}\cdot \shat{\phi})\nabla m_{u_1}\big) \\
&\ \ \ \ \ + 2\ \nabla \cdot \big((z_{u_2}\cdot \phi_{u_2})\nabla \shat{m}\big) +2\ \nabla \cdot \big((\shat{m}\cdot \phi_{u_1})\nabla z_{u_1}\big)+  2\ \nabla \cdot \big((m_{u_2}\cdot \shat{\phi})\nabla z_{u_1}\big)+ 2\ \nabla \cdot \big((m_{u_2}\cdot \phi_{u_2})\nabla \shat{z}\big)\\
\mathcal{B}_3&=-\Delta (\phi'_{u_2}\times \shat{m})  -\Delta (\shat{\phi}\times z_{u_1})-\Delta (\phi_{u_2}\times \shat{z}),\ \ \
\mathcal{B}_4=-\Delta \shat{m}\times \phi'_{u_2} -(\Delta \shat{z} \times \phi_{u_1}) -(\Delta z_{u_2}\times \shat{\phi})\\
\mathcal{B}_5&=\phi'_{u_2}\times \shat{u}+ \shat{\phi}\times u\\
\mathcal{B}_6&= -(\phi'_{u_2}\times \shat{m})\times u_1- (\phi'_{u_2}\times m_{u_2})\times \shat{u} -(\shat{\phi}\times z_{u_1})\times u_1 -(\phi_{u_2}\times \shat{z})\times u_1 -(\phi_{u_2}\times z_{u_2})\times \shat{u} \\
& \ \ \ \ \  -(\shat{\phi}\times m_{u_1})\times u -(\phi_{u_2}\times \shat{m})\times u\\
\mathcal{B}_7&=-\phi'_{u_2}\times (\shat{m}\times u_1)-\phi'_{u_2} \times (m_{u_2}\times \shat{u}) - \shat{\phi}\times (z_{u_1}\times u_1)  -\phi_{u_2} \times (\shat{z}\times u_1)-\phi_{u_2} \times (z_{u_2} \times \shat{u})\\
&\ \ \ \ \ \  -\shat{\phi}\times (m_{u_1}\times u)-\phi_{u_2} \times (\shat{m} \times u)\\
\mathcal{B}_{8}&=  2\ \nabla \cdot \big((\nabla \shat{m}\cdot \nabla z_{u_1})(\nabla m_{u_1}-\nabla m_d)\big) + 2\  \nabla \cdot\big(\big((\nabla m_{u_2}-\nabla m_d)\cdot \nabla \shat{z}\big)(\nabla m_{u_1}-\nabla m_d)\big)\\
&\ \ \ \ \ \ +  2\  \nabla \cdot \big(\big((\nabla m_{u_2}-\nabla m_d)\cdot \nabla z_{u_2}\big)\nabla \shat{m}\big)  +\nabla \cdot \big( \big(\nabla \shat{m}\cdot (\nabla m_{u_1} +\nabla m_{u_2} - 2\ \nabla m_d)\big)\nabla z_{u_1}\big) \\
&\ \ \ \ \ \ + \nabla \cdot \big(|\nabla m_{u_2}-\nabla m_d|^2\nabla \shat{z}\big).	
\end{flalign*}

Again appealing to Lemma \ref{AL-SLS} we can find an estimate on the solution of the system \eqref{DOAO} as follows:
\begin{eqnarray}\label{DAO-E2}
	\lefteqn{ \|\shat{\psi}\|^2_{L^2(0,T;H^1(\Omega))} + \|\shat{\psi}_t\|^2_{L^2(0,T;H^1(\Omega)^*)} \leq \left(\|z_{u_1}(T)-z_{u_2}(T) \|^2_{L^2(\Omega)}+ \sum_{k=1}^{8}  \|\mathcal{B}_k\|^2_{L^2(0,T;H^1(\Omega)^*)}\right) }\nonumber\\
	&&\hspace{0.5cm}\times \ \exp\bigg\{C  \left(\|m_{u_1}\|^2_{L^\infty(0,T;H^2(\Omega))}\|m_{u_1}\|^2_{L^2(0,T;H^3(\Omega))}+\|m_{u_1}\|^2_{L^\infty(0,T;H^2(\Omega))} \|u_1\|^2_{L^2(0,T;H^1(\Omega))}\right) \bigg\}.\ \ \ \ 
\end{eqnarray}
By doing calculations similar to Part-(i), we can find $\|\mathcal{B}_k\|_{L^2(0,T;H^1(\Omega)^*)} \leq C\ \|u_1-u_2\|_{L^2(0,T;H^1(\Omega))} \|u\|_{L^2(0,T;H^1(\Omega))}$. Also, for the final time data, we will use the estimate $\|z_{u_1}(T)-z_{u_2}(T)\|^2_{L^2(\Omega)} \leq \|z_{u_1}-z_{u_2}\|^2_{L^\infty(0,T;L^2(\Omega))} \leq C\ \|u_1-u_2\|^2_{L^2(0,T;H^1(\Omega))}\ \|u\|^2_{L^2(0,T;H^1(\Omega))}$. Therefore, by substituting these estimates in \eqref{DAO-E2}, we obtain the required result \eqref{LC-AD}. Hence the proof.
\end{proof}

\end{Pro}

\begin{Cor}\label{C-CTCD}
	The Fr\'echet derivative of the control-to-costate operator satisfies the following estimate:
	\begin{equation}\label{FDCTC-S}
		\|\Phi'(u)[v]\|_{\mathcal{Z}} \leq C\ \|u\|_{L^2(0,T;H^1(\Omega))} \ \|v\|_{L^2(0,T;H^1(\Omega))}.
	\end{equation}
\end{Cor}

\subsection{Second Order Optimality Condition}\label{S-SOOC}
%Since the problem is non-convex we can not guarantee the global optimality from the first-order necessary condition. In order to further analyze the given optimal control problem, let us investigate the second-order optimality condition. 

In the previous subsection, we proved the Fr\'echet derivative of the functional $\mathcal{J}$ in terms of state variable and co-state variable. Since both the control-to-state and control-to-costate operators are Fr\'echet differentiable, so we can evaluate the second order derivative of the cost function.

\begin{proof}[Proof of Theorem \ref{T-SOLO}]
	
	Let $\widetilde{u} $ be a fixed control  in the set $\mathcal{U}_{ad}$. Now, since the control set $\mathcal{U}$ is an open set, choose the value of $\epsilon$ small enough such that for $u\in \mathcal{U}_{ad}$ with $\|u-\widetilde{u}\|_{L^2(0,T;H^1(\Omega))}<\epsilon$ ensures that for any $\theta\in (0,1)$ the control $\widetilde{u}+\theta(u-\widetilde{u})\in \mathcal{U}$. Now, by Proposition \ref{P-CTS}, it is clear that the control-to-state operator $G$ is Fr\'echet differentiable  over this set of controls. Indeed, we can show that it is infinitely differentiable (see, Proposition 5.1 of \cite{SPSK}). Consequently, the reduced functional $\mathcal{I}$ is infinitely differentiable. Therefore, the Taylor's expansion for the functional $\mathcal{I}$ can be given by 
%	Now, from Remark 5.1 of \cite{SPSK}, it is clear that the control-to-state operator $G$ is infinitely differentiable over this set of controls and so as the reduced functional $\mathcal{I}$. Therefore, the Taylor's expansion for the functional $\mathcal{I}$ can be given by 
\begin{equation}\label{SOE1}
	\mathcal{J}(m,u)=\mathcal{I}(u)=\mathcal{I}(\widetilde{u})+\mathcal{I}'(\widetilde{u})(u-\widetilde{u})+ \frac{1}{2}\ \mathcal{I}''\big(\widetilde{u}+\theta(u-\widetilde{u})\big)\ (u-\widetilde{u})^2,
\end{equation}
	where $\theta \in (0,1)$. From the first order optimality condition \eqref{FOOC}, it is clear that 	$\mathcal{I}'(\widetilde{u})(u-\widetilde{u})\geq 0.	$
	From the definition of critical cone, it also clearly appears that, $u-\widetilde{u} \in  \Lambda(\widetilde{u})$, \ for all $u \in \mathcal{U}_{ad}$.
	
	 Since both the control-to-state and control-to-costate operators are Fr\'echet differentiable, so for any controls $u_1,u_2\in L^2(0,T;H^1(\Omega))$, the second derivative of the functional $\mathcal{I}$ is given by
	 \begin{eqnarray*}
	 	\lefteqn{\mathcal{I}''(\widetilde{u})[u_1,u_2]=\int_{\Omega_T} u_1\cdot u_2\ dx\ dt+\int_{\Omega_T} \nabla u_1\cdot \nabla u_2\ dx\ dt+ \int_{\Omega_T} \big(\phi'_{\widetilde{u}}[u_2]\times m_{\widetilde{u}}\big)\cdot u_1\ dx \ dt  }\nonumber\\
	 	&& + \int_{\Omega_T} \big( \phi_{\widetilde{u}} \times m'_{\widetilde{u}}[u_2]\big) \cdot u_1\ dx\ dt + \int_{\Omega_T} \big(m'_{\widetilde{u}}[u_2]\times (\phi_{\widetilde{u}}\times m_{\widetilde{u}})\big)\cdot u_1\ dx \ dt\nonumber\\
	 	&& +\int_{\Omega_T} \big(m_{\widetilde{u}} \times (\phi'_{\widetilde{u}}[u_2]\times m_{\widetilde{u}})\big)\cdot u_1 \ dx\ dt + \int_{\Omega_T} \big(m_{\widetilde{u}} \times (\phi_{\widetilde{u}} \times m'_{\widetilde{u}}[u_2])\big) \cdot u_1\ dx\ dt.
	 \end{eqnarray*}

 Next, we will show that the second derivative of the functional $\mathcal{I}$ satisfies the following continuity property 
 $$\left|\left[\mathcal{I}''\big(\widetilde{u}+\theta(u-\widetilde{u})\big)\ -\mathcal{I}''\big(\widetilde{u}\big) \right] (u-\widetilde{u})^2\right| \leq C\ \|u-\widetilde{u}\|^3_{L^2(0,T;H^1(\Omega))},$$
for any $\theta \in (0,1)$. In order to prove this inequality, we need the following estimates. 
%Here, we will only show the estimation of few terms, and the rest will follow a similar structure. 

 Applying H\"olders inequality and then implementing the Lipschitz continuity of Fr\'echet derivative of control-to-costate operator \eqref{LC-AD}, continuous dependency of the control-to-state operator \eqref{LCCTSO} and estimate \eqref{FDCTC-S}, we obtain 
%\begin{flalign*}
%	&\left| \int_{\Omega_T} \left( \phi'_{\widetilde{u}+\theta (u-\widetilde{u})}[u-\widetilde{u}] \times m_{\widetilde{u}+\theta (u-\widetilde{u})}\right) \cdot (u-\widetilde{u}) \ dx\ dt -\int_{\Omega_T} \left( \phi'_{\widetilde{u}}[u-\widetilde{u}] \times m_{\widetilde{u}}\right) \cdot (u-\widetilde{u}) \ dx\ dt \ \right|\\
%	&\hspace{1cm}\leq  \int_{\Omega_T} \left|\phi'_{\widetilde{u}+\theta (u-\widetilde{u})}[u-\widetilde{u}]-\phi'_{\widetilde{u}}[u-\widetilde{u}]\right|\ |m_{\widetilde{u}+\theta (u-\widetilde{u})}|\ |u-\widetilde{u}|\ dx\ dt\\
%	&\hspace{3cm}+ \int_{\Omega_T} |\phi'_{\widetilde{u}}[u-\widetilde{u}] \ |m_{\widetilde{u}+\theta (u-\widetilde{u})}-m_{\widetilde{u}}|\ |u-\widetilde{u}|\ dx\ dt\\
%	&\hspace{1cm}\leq \|\phi'_{\widetilde{u}+\theta (u-\widetilde{u})}[u-\widetilde{u}]-\phi'_{\widetilde{u}}[u-\widetilde{u}]\|_{L^2(0,T;L^2(\Omega))} \ \|u-\widetilde{u}\|_{L^2(0,T;L^2(\Omega))}\\
%	&\hspace{3cm}+ \|\phi'_{\widetilde{u}}[u-\widetilde{u}]\|_{L^\infty(0,T;L^2(\Omega))} \ \|m_{\widetilde{u}+\theta (u-\widetilde{u})}-m_{\widetilde{u}}\|_{L^2(0,T;H^1(\Omega))} \ \|u-\widetilde{u}\|_{L^2(0,T;H^1(\Omega))}\\
%	& \hspace{1cm} \leq C\ \|u-\widetilde{u}\|^3_{L^2(0,T;H^1(\Omega))}.
%\end{flalign*}
\begin{flalign*}
	&\left| \int_{\Omega_T} \left( \phi'_{\widetilde{u}+\theta (u-\widetilde{u})}[u-\widetilde{u}] \times m_{\widetilde{u}+\theta (u-\widetilde{u})}\right) \cdot (u-\widetilde{u}) \ dx\ dt -\int_{\Omega_T} \left( \phi'_{\widetilde{u}}[u-\widetilde{u}] \times m_{\widetilde{u}}\right) \cdot (u-\widetilde{u}) \ dx\ dt \ \right|\\
	%	&\hspace{1cm}\leq  \int_{\Omega_T} \left[\left|\phi'_{\widetilde{u}+\theta (u-\widetilde{u})}[u-\widetilde{u}]-\phi'_{\widetilde{u}}[u-\widetilde{u}]\right|\ |m_{\widetilde{u}+\theta (u-\widetilde{u})}|+ |\phi'_{\widetilde{u}}[u-\widetilde{u}] \ |m_{\widetilde{u}+\theta (u-\widetilde{u})}-m_{\widetilde{u}}|\right]\ |u-\widetilde{u}|\ dx\ dt\\
	&\hspace{1cm}\leq \|\phi'_{\widetilde{u}+\theta (u-\widetilde{u})}[u-\widetilde{u}]-\phi'_{\widetilde{u}}[u-\widetilde{u}]\|_{L^2(0,T;L^2(\Omega))} \ \|u-\widetilde{u}\|_{L^2(0,T;L^2(\Omega))}\\
	&\hspace{3cm}+ \|\phi'_{\widetilde{u}}[u-\widetilde{u}]\|_{L^\infty(0,T;L^2(\Omega))} \ \|m_{\widetilde{u}+\theta (u-\widetilde{u})}-m_{\widetilde{u}}\|_{L^2(0,T;H^1(\Omega))} \ \|u-\widetilde{u}\|_{L^2(0,T;H^1(\Omega))}\\
	& \hspace{1cm} \leq C\ \|u-\widetilde{u}\|^3_{L^2(0,T;H^1(\Omega))}.
\end{flalign*}

Applying H\"olders inequality,   the estimates \eqref{LC-CTS},  \eqref{LCCTSO},  \eqref{LCI},   \eqref{FDCTS-S},  and \eqref{ALSSE},  we get
\begin{align*}
	&\left|\int_{\Omega_T} \left[\left(m'_{\widetilde{u}+\theta (u-\widetilde{u})}[u-\widetilde{u}] \times (\phi_{\widetilde{u}+\theta (u-\widetilde{u})}\times m_{\widetilde{u}+\theta (u-\widetilde{u})})\right)-  \left(m'_{\widetilde{u}}[u-\widetilde{u}] \times (\phi_{\widetilde{u}}\times m_{\widetilde{u}})\right)\right]\cdot (u-\widetilde{u})\ dx\ dt\ \right|\\
	%	&\left|\int_{\Omega_T} \big(m'_{\widetilde{u}+\theta (u-\widetilde{u})}[u-\widetilde{u}] \times (\phi_{\widetilde{u}+\theta (u-\widetilde{u})}\times m_{\widetilde{u}+\theta (u-\widetilde{u})})\big)\cdot (u-\widetilde{u})\ dx\ dt\right.\\
	%	&\left. -\int_{\Omega_T} \big(m'_{\widetilde{u}}[u-\widetilde{u}] \times (\phi_{\widetilde{u}}\times m_{\widetilde{u}})\big)\cdot (u-\widetilde{u})\ dx\ dt\ \right|\\
	%	& \leq \int_{\Omega_T} \bigg[ \left|m'_{\widetilde{u}+\theta (u-\widetilde{u})} [u-\widetilde{u}]-m'_{\widetilde{u}}[u-\widetilde{u}]\right| \ \left|\phi_{\widetilde{u}+\theta(u-\widetilde{u})}\right| \ \left|m_{\widetilde{u}+\theta (u-\widetilde{u})}\right|\ \\
	%	&\hspace{1cm} + \left|m'_{\widetilde{u}}[u-\widetilde{u}]\right| \ \left|\phi_{\widetilde{u}+\theta (u-\widetilde{u})}-\phi_{\widetilde{u}}\right|\ \left|m_{\widetilde{u}+\theta (u-\widetilde{u})}\right| \  + \big|m'_{\widetilde{u}}[u-\widetilde{u}]\big| \ \big|\phi_{\widetilde{u}}\big|\ \big|m_{\widetilde{u}+\theta (u-\widetilde{u})}-m_{\widetilde{u}}\big|\bigg]\   \big|u-\widetilde{u}\big| \ dt\\
	&\leq \int_0^T \|m'_{\widetilde{u}+\theta (u-\widetilde{u})}[u-\widetilde{u}]-m'_{\widetilde{u}}[u-\widetilde{u}]\|_{L^4(\Omega)}\  \|\phi_{\widetilde{u}+\theta(u-\widetilde{u})}\|_{L^2(\Omega)} \ \|m_{\widetilde{u}+\theta (u-\widetilde{u})}\|_{L^8(\Omega)}\  \|u-\widetilde{u}\|_{L^8(\Omega)} \ dt\\
	& \hspace{1cm}+ \int_0^T \|m'_{\widetilde{u}}[u-\widetilde{u}]\|_{L^\infty(\Omega)} \ \|\phi_{\widetilde{u}+\theta (u-\widetilde{u})}-\phi_{\widetilde{u}}\|_{L^2(\Omega)}\ \|m_{\widetilde{u}+\theta (u-\widetilde{u})}\|_{L^4(\Omega)} \  \|u-\widetilde{u}\|_{L^4(\Omega)}\ dt \\
	&\hspace{1cm}+ \int_0^T  \|m'_{\widetilde{u}}[u-\widetilde{u}]\|_{L^\infty(\Omega)} \ \|\phi_{\widetilde{u}}\|_{L^2(\Omega)}\ \|m_{\widetilde{u}+\theta (u-\widetilde{u})}-m_{\widetilde{u}}\|_{L^4(\Omega)} \  \|u-\widetilde{u}\|_{L^4(\Omega)} \ dt\\
	&\leq C\ \Big[ \|m'_{\widetilde{u}+\theta (u-\widetilde{u})}[u-\widetilde{u}]-m'_{\widetilde{u}}[u-\widetilde{u}]\|_{L^2(0,T;H^1(\Omega))}\ \| \phi_{\widetilde{u}+\theta (u-\widetilde{u})}-\phi_{\widetilde{u}}\|_{L^\infty(0,T;L^2(\Omega))}\ \|m_{\widetilde{u}+\theta (u-\widetilde{u})}\|_{L^\infty(0,T;H^1(\Omega))}\\
	&\hspace{1cm}+ \|m'_{\widetilde{u}}[u-\widetilde{u}]\|_{L^\infty(0,T;H^2(\Omega))} \ \|\phi_{\widetilde{u}+\theta (u-\widetilde{u})}-\phi_{\widetilde{u}}\|_{L^2(0,T;L^2(\Omega))}\ \|m_{\widetilde{u}+\theta (u-\widetilde{u})}\|_{L^\infty(0,T;H^1(\Omega))} \\
	&\hspace{1cm} +\   \|m'_{\widetilde{u}}[u-\widetilde{u}]\|_{L^\infty(0,T;H^2(\Omega))} \  \|\phi_{\widetilde{u}}\|_{L^\infty(0,T;L^2(\Omega))}\ \|m_{\widetilde{u}+\theta (u-\widetilde{u})}-m_{\widetilde{u}}\|_{L^2(0,T;H^1(\Omega))} \Big]\  \|u-\widetilde{u}\|_{L^2(0,T;H^1(\Omega))}\\
	&\leq C\ \|u-\widetilde{u}\|^3_{L^2(0,T;H^1(\Omega))}.
\end{align*}

To establish the validity of other estimates, we can employ analogous arguments. Therefore, using these estimates and  assumption \eqref{AOSOD}, we conclude the following estimation
\begin{eqnarray}\label{SOE3}
\mathcal{I}''\big(\widetilde{u}+\theta(u-\widetilde{u})\big)\ (u-\widetilde{u})^2
&= &\mathcal{I}''\big(\widetilde{u}\big)\ (u-\widetilde{u})^2+ \left[\mathcal{I}''\big(\widetilde{u}+\theta(u-\widetilde{u})\big)\ -\mathcal{I}''\big(\widetilde{u}\big) \right](u-\widetilde{u})^2\\
&\geq& \delta \ \|u-\widetilde{u}\|^2_{L^2(0,T;H^1(\Omega))}-C\ \|u-\widetilde{u}\|^3_{L^2(0,T;H^1(\Omega))}
\geq \frac{\delta}{2} \  \|u-\widetilde{u}\|^2_{L^2(0,T;H^1(\Omega))},	 \nonumber
\end{eqnarray}
in which we have chosen $\epsilon>0$  small such that $\|u-\widetilde{u}\|_{L^2(0,T;H^1(\Omega))}\leq \frac{\delta}{2C}$, where $\delta>0$ is the constant coming from \eqref{AOSOD}. Finally, from \eqref{SOE1}-\eqref{SOE3}, it is evident that 
 $\mathcal{I}(u)\geq \mathcal{I}(\widetilde{u})+ \sigma\ \|u-\widetilde{u}\|^2_{L^2(0,T;H^1(\Omega))},$
 with $\sigma=\delta/4$, provided $\|u-\widetilde{u}\|_{L^2(0,T;H^1(\Omega))}\leq \epsilon$ for a sufficiently small $\epsilon>0$. Hence the proof. 
\end{proof} 
%\begin{proof}[Proof of Lemma \ref{L-SOC}]\ 
%	\begin{flalign*}
 %(i)\ \ \left| \int_{\Omega_T} \left[\left( \phi'_{\widetilde{u}+\theta (u-\widetilde{u})}[u-\widetilde{u}] \times m_{\widetilde{u}+\theta (u-\widetilde{u})}\right) -\left( \phi'_{\widetilde{u}}[u-\widetilde{u}] \times m_{\widetilde{u}}\right)\right] \cdot (u-\widetilde{u}) \ dx\ dt \ \right| \leq C\ \|u-\widetilde{u}\|^3_{L^2(0,T;H^1(\Omega))}.&
%	\end{flalign*}
%\end{proof} 

\section{Global Optimal Control}\label{S-GOC}
The second-order sufficient conditions primarily provide localized information and often fall short of enabling a definitive determination regarding whether a given point constitutes a global minimum for the OCP under consideration. This pursuit leads us to a specific condition pertaining to the boundedness of a certain norm of the co-state variable by a constant quantity that relies exclusively on the given data and is explicitly quantifiable. 
%This condition not only sheds light on the global optimality of our control strategy but also contributes valuable insights into the underlying system dynamics, as we elucidate in the subsequent sections of this research paper.
%To gain a deeper understanding and ascertain global optimality, further scrutiny through second-order conditions becomes imperative. However, it is important to note that these second-order conditions typically furnish localized insights and may not be sufficient for definitively determining whether a particular point represents a global minimum for the optimization problem (P).
%In the context of our research, it is important to acknowledge that the state equation generally exhibits nonlinearity, rendering the associated control problem inherently nonconvex. Consequently, we may encounter multiple solutions when considering the necessary first-order conditions.
%Finally, we will derive a sufficient optimality condition for the global optimal control. This condition will serve as a criterion to assess whether a given control strategy achieves global optimality for the considered system. By establishing this condition, we aim to enhance our understanding of the control problem and enable the identification of control strategies that yield superior overall performance.
\begin{proof}[Proof of Theorem \ref{T-GO}]
%	In order to prove the global optimality of the control $\widetilde{u}$, we claim that $ \mathcal{I}(u)-\mathcal{I}(\widetilde{u})\geq 0$ for all $u\in \mathcal{U} \backslash \{\widetilde{u}\}$. A straightforward estimation yields the following equality:

 Let us define $\shat{m}:= m-\widetilde{m}$ and $\shat{u}:= u-\widetilde{u}$.	To establish the global optimality of the control $\widetilde{u}$, we assert that $ \mathcal{I}(u)-\mathcal{I}(\widetilde{u})\geq 0$ for all $u\in \mathcal{U}_{ad} \backslash \{\widetilde{u}\}$. By doing a straightforward calculation, we arrive at the following equality:
\begin{align}\label{GO-1}
	\mathcal{I}(u)-\mathcal{I}(\widetilde{u}) =& \ \frac{1}{2}\int_{\Omega_T} |\shat{u}|^2\ dx\ dt + \int_{\Omega_T} \shat{u}  \cdot \widetilde{u} \ dx\ dt + \frac{1}{2} \int_{\Omega_T} |\nabla \shat{u}|^2 \ dx\ dt+ \int_{\Omega_T} \nabla \shat{u} \cdot \nabla \widetilde{u} \ dx\ dt\nonumber\\
	&+ \frac{1}{2} \int_{\Omega} |\shat{m}(x,T)|^2\ dx\ dt+ \int_{\Omega} \shat{m}(x,T)\cdot \big(\widetilde{m}(x,T)-m_{\Omega}(x)\big) \ dx\nonumber\\
	& + \frac{1}{4} \int_{\Omega_T} |\nabla \shat{m}|^4\ dx\ dt+ \int_{\Omega_T} \nabla \shat{m} \cdot (\nabla \widetilde{m}-\nabla m_d)\ |\nabla \widetilde{m}-\nabla m_d|^2 \ dx\ dt\nonumber\\
	&+ \frac{3}{2} \int_{\Omega_T} |\nabla \shat{m}|^2 \cdot |\nabla \widetilde{m}-\nabla m_d|^2 \ dx\ dt+ \int_{\Omega_T} |\nabla \shat{m}|^2 \ \nabla \shat{m}\cdot (\nabla \widetilde{m}-\nabla m_d)\ dx\ dt.
\end{align}
Now, applying the estimate $\int_{\Omega_T} |\nabla \shat{m}|^2 \ \nabla \shat{m}\cdot (\nabla \widetilde{m}-\nabla m_d)\ dx\ dt\geq -\frac{3}{2} \int_{\Omega_T} |\nabla \shat{m}|^2 \ |\nabla \widetilde{m}-\nabla m_d|^2 \ dx\ dt - \frac{1}{6} \int_{\Omega_T} |\nabla \shat{m}|^4 \ dx\ dt$ 
and the variational inequality \eqref{FOOC} in \eqref{GO-1}, we find 
%\begin{align}
%	\mathcal{I}(u)-\mathcal{I}(\widetilde{u}) \geq & \ \frac{1}{2}\int_{\Omega_T} |\shat{u}|^2\ dx\ dt  + \frac{1}{2} \int_{\Omega_T} |\nabla \shat{u}|^2 \ dx\ dt+ \frac{1}{2} \int_{\Omega} |\shat{m}(x,T)|^2\ dx\ dt-\int_{\Omega_T} (\phi \times \widetilde{m})\cdot \shat{u} \ dx\ dt\nonumber\\
%	& -\int_{\Omega_T} \big(\widetilde{m} \times (\phi \times \widetilde{m}) \big)v\cdot \shat{u}\ dx\ dt+ \int_{\Omega} \shat{m}(x,T)\cdot \big(m_{\widetilde{u}}(x,T)-m_{\Omega}(x)\big) \ dx\nonumber\\
%	& + \frac{1}{12} \int_{\Omega_T} |\nabla \shat{m}|^4\ dx\ dt+ \int_{\Omega_T} \nabla \shat{m} \cdot (\nabla m_{\widetilde{u}}-\nabla m_d)\ |\nabla m_{\widetilde{u}}-\nabla m_d|^2 \ dx\ dt.
%\end{align}
\begin{align}\label{GO-2}
	\mathcal{I}(u)-\mathcal{I}(\widetilde{u}) \geq&  \ \frac{1}{2}\int_{\Omega_T}\left( |\shat{u}|^2+|\nabla \shat{u}|^2 \right)\ dx\ dt+ \frac{1}{2} \int_{\Omega} |\shat{m}(x,T)|^2\ dx\ dt+ \frac{1}{12} \int_{\Omega_T} |\nabla \shat{m}|^4\ dx\ dt + \mathcal{R},
\end{align}\vspace{-0.5cm}
\begin{flalign*}
	\text{where}\ \mathcal{R} :=& -\int_{\Omega_T} (\phi \times \widetilde{m})\cdot \shat{u} \ dx\ dt + \int_{\Omega} \shat{m}(x,T)\cdot \big(\widetilde{m}(x,T)-m_{\Omega}(x)\big) \ dx&\\
	&-\int_{\Omega_T} \big(\widetilde{m} \times (\phi \times \widetilde{m}) \big) \cdot \shat{u}\ dx\ dt + \int_{\Omega_T} \nabla \shat{m} \cdot \big(\nabla \widetilde{m}-\nabla m_d\big)\ |\nabla \widetilde{m}-\nabla m_d|^2 \ dx\ dt.
\end{flalign*}
%\begin{equation}
%	\mathcal{I}(u)-\mathcal{I}(\widetilde{u}) \geq  \ \frac{1}{2}\int_{\Omega_T} |\shat{u}|^2\ dx\ dt  + \frac{1}{2} \int_{\Omega_T} |\nabla \shat{u}|^2 \ dx\ dt+ \frac{1}{2} \int_{\Omega} |\shat{m}(x,T)|^2\ dx\ dt+ \frac{1}{12} \int_{\Omega_T} |\nabla \shat{m}|^4\ dx\ dt
%\end{equation}
%Suppose $m_{u}\in \mathcal{M}$ be the regular solution of system \eqref{NLP} corresponding to the control $u \in \mathcal{U}_{R}$.
From the definition of $\shat{u}$ and $\shat{m}$, we know that $(\shat{m},\shat{u})$ satisfies the following system 
\begin{equation}\label{GO-S}
	\begin{cases}
		\begin{array}{l}
			\mathcal{L}_{\widetilde{u}}\shat{m}=|\nabla \shat{m}|^2 \ \shat{m}+ |\nabla \shat{m}|^2 \widetilde{m} +2 \big(\nabla \shat{m}\cdot \nabla \widetilde{m}\big)\shat{m}+ \shat{m} \times \Delta \shat{m} + \shat{m} \times \shat{u} +\widetilde{m} \times \shat{u} - \shat{m} \times (\shat{m} \times \shat{u}) \\
			\hspace{2cm}-\shat{m} \times (\shat{m} \times \widetilde{u}) -\shat{m} \times (\widetilde{m} \times \shat{u}) - \widetilde{m} \times ( \shat{m} \times \shat{u}) - \widetilde{m} \times (\widetilde{m}\times \shat{u}) \ \ \ \ \text{in}\ \Omega_T,\\
			\frac{\partial \shat{m}}{\partial \eta}=0 \ \ \ \ \ \  \text{in}\ \partial \Omega_T,\ \ \
			\shat{m}(x,0)=0\ \ \text{in}\ \Omega,
		\end{array}
	\end{cases}	
\end{equation}
where the operator $\mathcal{L}_{\widetilde{u}}$ is defined in \eqref{CLO}.

Taking $L^2(0,T;L^2(\Omega))$ inner product of \eqref{GO-S} with the weak solution $\phi$ of the adjoint sytem \eqref{AS}, and then doing an integration by parts, we derive
\begin{align}\label{GO-3}
	&\int_0^T \big<\mathcal{E}_{\widetilde{u}}\phi, \shat{m}\big>\ dt +\int_0^T\big(|\nabla \shat{m}|^2 \ \shat{m},\phi\big)\ dt+ \int_0^T \big(|\nabla \shat{m}|^2 \widetilde{m},\phi \big)\ dt +2 \int_0^T\big(\big(\nabla \shat{m}\cdot \nabla \widetilde{m}\big)\shat{m},\phi \big)\ dt&\nonumber\\
	&\hspace{1cm}+ \int_0^T\big(\shat{m} \times \Delta \shat{m},\phi\big)\ dt+ \int_0^T \big(\shat{m} \times \shat{u},\phi \big)\ dt- \int_0^T \big( \shat{m} \times (\shat{m} \times \shat{u}),\phi \big)\ dt	 \nonumber\\
	&\hspace{1cm}-\int_0^T\big(\shat{m} \times (\shat{m} \times \widetilde{u}), \phi \big)\ dt -\int_0^T\big(\shat{m} \times (\widetilde{m} \times \shat{u}), \phi \big)\ dt - \int_0^T \big(\widetilde{m} \times ( \shat{m} \times \shat{u}), \phi \big)\ dt\nonumber\\
	&\hspace{1cm}= \int_0^T\big(\shat{m}(T),\phi(T)\big)\ dt-\int_0^T \big(\widetilde{m} \times \shat{u},\phi \big)\ dt+ \int_0^T \big(\widetilde{m} \times (\widetilde{m}\times \shat{u}), \phi \big)\ dt,
\end{align}
where we also used the identity similar to \eqref{TIP} for $\int_0^T \big(\shat{m}_t,\phi\big)\ dt$. Now, from the weak formulation of the adjoint problem \eqref{AS}, we have
\begin{align}\label{GO-4}
	&\int_0^T  \big< \phi', \shat{m} \big> \ dt- \int_0^T \big( \nabla \phi, \nabla \shat{m}\big)\ dt +\int_0^T \big( |\nabla \widetilde{m}|^2 \ \phi, \shat{m}\big)\ dt -2 \int_0^T  \big( \nabla \cdot \left\{(\widetilde{m}\cdot \phi)\nabla \widetilde{m}\right\}, \shat{m}\big) \ dt\nonumber\\
	&\hspace{0.5cm}- \int_0^T \big( \nabla (\phi \times \widetilde{m}), \nabla \shat{m}\big)\ dt +\int_0^T \big( \Delta \widetilde{m} \times \phi , \shat{m}\big)\ dt- \int_0^T \big( \phi \times \widetilde{u}, \shat{m}\big)\ dt + \int_0^T \big( (\phi \times \widetilde{m})\times \widetilde{u}, \shat{m}\big) \ dt\nonumber\\
	&\hspace{0.5cm} + \int_0^T \big( \phi \times(\widetilde{m} \times \widetilde{u}), \shat{m}\big)\ dt= -\int_0^T \left(|\nabla \widetilde{m}-\nabla m_d|^2 \big(\nabla \widetilde{m}-\nabla m_d), \nabla \shat{m}\right)\ dt.
\end{align}
Subtracting equality \eqref{GO-4} from \eqref{GO-3}, we get
\begin{align}\label{GO-5}
	&\int_0^T \big(|\nabla \shat{m}|^2 \ \shat{m},\phi\big)\ dt+ \int_0^T \big(|\nabla \shat{m}|^2 \widetilde{m},\phi \big)\ dt +2 \int_0^T \big( \big(\nabla \shat{m}\cdot \nabla \widetilde{m}\big)\ \shat{m},\phi \big)\ dt+ \int_0^T \big(\shat{m} \times \Delta \shat{m},\phi\big)\ dt&\nonumber\\
	&\hspace{0.5cm}+ \int_0^T \big(\shat{m} \times \shat{u},\phi \big)\ dt- \int_0^T \big( \shat{m} \times (\shat{m} \times \shat{u}),\phi \big)\ dt	-\int_0^T \big(\shat{m} \times (\shat{m} \times \widetilde{u}), \phi \big)\ dt  \nonumber\\
	&\hspace{0.5cm}-\int_0^T \big(\shat{m} \times (\widetilde{m} \times \shat{u}), \phi \big)\ dt - \int_0^T \big(\widetilde{m} \times ( \shat{m} \times \shat{u}), \phi \big)\ dt= \int_0^T \big(\shat{m}(T),\phi(T)\big)\ dt\nonumber\\
	&\hspace{0.3cm}-\int_0^T \big(\widetilde{m} \times \shat{u},\phi \big)\ dt+ \int_0^T \big(\widetilde{m} \times (\widetilde{m}\times \shat{u}), \phi \big)\ dt+ \int_0^T \left(|\nabla \widetilde{m}-\nabla m_d|^2  \big(\nabla \widetilde{m}-\nabla m_d), \nabla \shat{m}\right)\ dt.
\end{align}
Therefore, as a result of equality \eqref{GO-5} and $a\cdot(b\times c)=-(b\times a)\cdot c$, the value of $\mathcal{R}$ from \eqref{GO-2} becomes 
\begin{align}\label{GO-6}
	&\mathcal{R}=\int_0^T \big(|\nabla \shat{m}|^2 \ \shat{m},\phi\big)\ dt+ \int_0^T \big(|\nabla \shat{m}|^2 \widetilde{m},\phi \big)\ dt +2 \int_0^T \big( \big(\nabla \shat{m}\cdot \nabla \widetilde{m}\big)\ dt\shat{m},\phi \big)\ dt&\nonumber\\
	&\hspace{1cm}+ \int_0^T \big(\shat{m} \times \Delta \shat{m},\phi\big)\ dt+ \int_0^T \big(\shat{m} \times \shat{u},\phi \big)\ dt- \int_0^T \big( \shat{m} \times (\shat{m} \times \shat{u}),\phi \big)\ dt	\nonumber\\
	&\hspace{1cm} -\int_0^T \big(\shat{m} \times (\shat{m} \times \widetilde{u}), \phi \big)\ dt-\int_0^T \big(\shat{m} \times (\widetilde{m} \times \shat{u}), \phi \big)\ dt - \int_0^T \big(\widetilde{m} \times ( \shat{m} \times \shat{u}), \phi \big)\ dt. 
\end{align}

Now, let us estimate the bounds for each term on the right hand side of $\mathcal{R}$ individually. For the first three terms, using the fact that $|\shat{m}|=|m-\widetilde{m}|\leq |m|+|\widetilde{m}|=2$ and applying H\"older's inequality and the continuous embedding $H^1(\Omega)\hookrightarrow L^4(\Omega)$, followed by the Lipschitz continuity of control-to-state operator given by estimate \eqref{LCCTSO}, we obtain
\begin{flalign*}
	\int_0^T \big(|\nabla \shat{m}|^2\ \shat{m}, \phi\big) \ dt & \leq 2 \int_0^T \|\nabla \shat{m}(t)\|^2_{L^4(\Omega)}\  \|\phi(t)\|_{L^2(\Omega)}\ dt \leq C \int_0^T \|\nabla \shat{m}(t)\|^2_{H^1(\Omega)} \ \|\phi(t)\|_{L^2(\Omega)}\ dt&\\
	& \leq C\ \|\phi\|_{L^2(0,T;L^2(\Omega))} \ \|\shat{m}\|^2_{L^{\infty}(0,T;H^2(\Omega))}\leq C\ \|\phi\|_{L^2(0,T;L^2(\Omega))} \ \|\shat{u}\|^2_{L^2(0,T;H^1(\Omega))},\\
	\int_0^T \big(|\nabla \shat{m}|^2\ \widetilde{m}, \phi\big)\ dt &\leq  C\ \|\phi\|_{L^2(0,T;L^2(\Omega))} \ \|\shat{u}\|^2_{L^2(0,T;H^1(\Omega))},&	
%\int_0^T \|\nabla \shat{m}(t)\|^2_{L^4(\Omega)}\  \|\phi(t)\|_{L^2(\Omega)}\ dt \leq C \int_0^T \|\nabla \shat{m}(t)\|^2_{H^1(\Omega)} \ \|\phi(t)\|_{L^2(\Omega)}\ dt& \\
%	&\leq C\ \|\phi\|_{L^2(0,T;L^2(\Omega))} \ \|\shat{m}\|^2_{L^{\infty}(0,T;H^2(\Omega))}\leq C\ \|\phi\|_{L^2(0,T;L^2(\Omega))} \ \|\shat{u}\|^2_{L^2(0,T;H^1(\Omega))},
\end{flalign*}\vspace{-0.3cm}
\begin{flalign*}
	\text{and}\ \ \ 2\int_0^T \big( \big(\nabla \shat{m} \cdot \nabla \widetilde{m}\big)\shat{m},\phi\big) \ dt &\leq 2 \int_0^T \|\nabla \shat{m}(t)\|_{L^4(\Omega)} \ \|\nabla \widetilde{m}(t)\|_{L^8(\Omega)}\  \|\shat{m}(t)\|_{L^8(\Omega)} \ \|\phi(t)\|_{L^2(\Omega)}\ dt&\\
	&\leq C\ \|\widetilde{m}\|_{L^2(0,T;H^2(\Omega))} \ \|\phi\|_{L^2(0,T;L^2(\Omega))} \ \|\shat{m}\|^2_{L^\infty(0,T;H^2(\Omega))} \\
	&\leq C\ \|\widetilde{m}\|_{L^2(0,T;H^2(\Omega))} \ \|\phi\|_{L^2(0,T;L^2(\Omega))} \ \|\shat{u}\|^2_{L^2(0,T;H^1(\Omega))}.
\end{flalign*}
The terms $\big(\shat{m} \times \Delta \shat{m},\phi\big)$ and $\big(\shat{m}\times \shat{u},\phi \big) $ can be estimated by the same upper bound.  Further, applying the H\"older's inequality, the embedding $H^1(\Omega)\hookrightarrow L^p(\Omega)$ for $p\in [1,\infty)$ and estimate \eqref{LCCTSO} for the subsequent terms involving the control, we derive
%\begin{flalign*}
%	\big(\shat{m} \times \Delta \shat{m},\phi\big) &\leq \int_0^T \|\shat{m}(t)\|_{L^\infty(\Omega)}\ \|\Delta \shat{m}(t)\|_{L^2(\Omega)} \ \|\phi(t)\|_{L^2(\Omega)}\ dt \leq C\ \int_0^T \|\phi(t)\|_{L^2(\Omega)}\ \|\shat{m}\|^2_{H^2(\Omega)}\ dt&\\
%	&\leq  C\ \|\phi\|_{L^2(0,T;L^2(\Omega))}\ \|\shat{m}\|^2_{L^\infty(0,T;H^2(\Omega))} \leq C\ \|\phi\|_{L^2(0,T;L^2(\Omega))} \|\shat{u}\|^2_{L^2(0,T;H^1(\Omega))}
%\end{flalign*}
%\begin{flalign*}
%	\big(\shat{m}\times \shat{u},\phi \big) &\leq  \int_0^T \|\shat{m}(t)\|_{L^4(\Omega)}\ \|\shat{u}(t)\|_{L^4(\Omega)} \ \|\phi(t)\|_{L^2(\Omega)}\ dt \leq C\int_0^T \|\shat{m}(t)\|_{H^1(\Omega)} \ \|\shat{u}(t)\|_{H^1(\Omega)}\ \|\phi(t)\|_{L^2(\Omega)}\ dt&\\
%	&\leq C\ \|\phi\|_{L^2(0,T;L^2(\Omega))}\ \|\shat{m}\|_{L^\infty(0,T;H^1(\Omega))}\ \|\shat{u}\|_{L^2(0,T;H^1(\Omega))} \leq C\ \|\phi\|_{L^2(0,T;L^2(\Omega))} \|\shat{u}\|^2_{L^2(0,T;H^1(\Omega))}
%\end{flalign*}
\begin{flalign*}
	\int_0^T &\big(\shat{m} \times (\shat{m}\times \shat{u}),\phi\big)\ dt \leq  2\int_0^T \|\shat{m}(t)\|_{L^4(\Omega)}\ \|\shat{u}(t)\|_{L^4(\Omega)} \ \|\phi(t)\|_{L^2(\Omega)}\ dt&\\
%	&\leq C\int_0^T \|\shat{m}(t)\|_{H^1(\Omega)} \ \|\shat{u}(t)\|_{H^1(\Omega)}\ \|\phi(t)\|_{L^2(\Omega)}\ dt\\
	&\hspace{1cm}\leq C\ \|\phi\|_{L^2(0,T;L^2(\Omega))}\ \|\shat{m}\|_{L^\infty(0,T;H^1(\Omega))}\ \|\shat{u}\|_{L^2(0,T;H^1(\Omega))} \leq C\ \|\phi\|_{L^2(0,T;L^2(\Omega))} \ \|\shat{u}\|^2_{L^2(0,T;H^1(\Omega))},
\end{flalign*}\vspace{-0.6cm}
\begin{flalign*}
&\text{and}\ \ 	\int_0^T \big(\shat{m} \times (\shat{m}\times \widetilde{u}),\phi\big) \ dt\leq  2\int_0^T \|\shat{m}(t)\|^2_{L^8(\Omega)}\ \|\widetilde{u}(t)\|_{L^4(\Omega)} \ \|\phi(t)\|_{L^2(\Omega)}\ dt&\\
	%&\leq C\int_0^T \|\shat{m}(t)\|^2_{H^1(\Omega)} \ \|\widetilde{u}(t)\|_{H^1(\Omega)}\ \|\phi(t)\|_{L^2(\Omega)}\ dt&\\
	&\hspace{0.5cm}\leq C\ \|\phi\|_{L^2(0,T;L^2(\Omega))}\ \|\widetilde{u}\|_{L^2(0,T;H^1(\Omega))}\ \|\shat{m}\|^2_{L^\infty(0,T;H^1(\Omega))}\leq C\ \|\phi\|_{L^2(0,T;L^2(\Omega))}\  \|\widetilde{u}\|_{L^2(0,T;H^1(\Omega))}\  \|\shat{u}\|^2_{L^2(0,T;H^1(\Omega))}.
\end{flalign*}
The rest of the terms in \eqref{GO-6} can be estimated in a similar way. Now, combining all the above estimates, we find the following bound for $\mathcal{R}$:
\begin{equation*}
	|\mathcal{R}| \leq C(\Omega,T)\ \left\{ 1+ \|\widetilde{m}\|_{L^\infty(0,T;H^2(\Omega))}+\|\widetilde{u}\|_{L^2(0,T;H^2(\Omega))} \right\} \ \|\phi\|_{L^2(0,T;L^2(\Omega))} \ \|\shat{u}\|^2_{L^2(0,T;H^1(\Omega))}.
\end{equation*}
Substituting this bound in estimate \eqref{GO-2}, we find our required result. The proof is thus completed.
\end{proof}
\noindent {\bf Data availability statement:} We do not generate any datasets, because our work proceeds within a theoretical and mathematical approach.

\end{document}